\documentclass{article}[12pt]
\usepackage[utf8]{inputenc}
\usepackage{dsfont}
\usepackage{amsmath, amsthm, amssymb, mathrsfs, amsfonts}
\usepackage[abbrev,non-sorted-cites]{amsrefs}
\usepackage[all]{xy}
\usepackage{CJK}
\usepackage[abs]{overpic}
\usepackage{grffile}
\usepackage{graphicx}
\usepackage[usenames, dvipsnames]{color}
\usepackage[dvipsnames]{xcolor}
\usepackage{mathtools}
\usepackage[all]{xy}
\usepackage{mathrsfs}
\usepackage{tikz,stackrel}
\usepackage[all]{xy}
\usepackage[margin=1in]{geometry}
\usetikzlibrary{calc}
\usetikzlibrary{matrix,arrows,decorations.pathmorphing}
\usetikzlibrary{intersections}
\usepackage{tikz-cd}
\usepackage[english]{babel}
\usepackage{stmaryrd}
\usepackage{float}
\usepackage{comment}
\usepackage{caption}
\usepackage{youngtab}
\usepackage{ragged2e}
\usepackage{relsize}
\usepackage{anyfontsize}
\usepackage{bm}
\usepackage{bbm}
\usepackage{scalerel}
\usepackage{marvosym}
\usetikzlibrary{arrows,snakes,backgrounds}
\usepackage{hyperref}
\usepackage{makecell}

\newcommand{\PP}{{\mathbb P}}
\newcommand{\C}{{\mathbb C}}

\newcommand{\Q}{{\mathbb Q}}
\newcommand{\R}{{\mathbb R}}
\newcommand{\Z}{{\mathbb Z}}
\newcommand{\N}{{\mathbb N}}

\newcommand{\cX}{{\mathcal X}}

\newcommand{\frakd}{\mathfrak{d}}

\newcommand{\frakD}{\mathfrak{D}}

\newcommand{\frakm}{\mathfrak{m}}

\newcommand{\cA}{\mathcal{A}}

\DeclareMathOperator{\Hom}{Hom}

\DeclareMathOperator{\Int}{Int}
\DeclareMathOperator{\Spec}{Spec}
\DeclareMathOperator{\rank}{rank}

\newcommand{\gone}{{\gamma_1'}}
\newcommand{\gtwo}{{\gamma_2'}}
\newcommand{\gthree}{{\gamma_3'}}
\newcommand{\gfour}{{\gamma_4'}}
\newcommand{\gi}{{\gamma_i'}}
\newcommand{\gj}{{\gamma_j'}}

\usepackage{upgreek}
\usepackage{caption}
\usetikzlibrary{shapes,arrows,patterns,decorations.pathreplacing,shapes.geometric}
\tikzstyle{sing}=[star, draw=black, fill=white, inner sep=0pt, minimum size=5pt]

\newcommand{\YGHK}{{Y}}
\newcommand{\DGHK}{{D}}
\newcommand{\GHK}{\mathrm{GHK}}
\newcommand{\BGHK}{B_\mathrm{GHK}}
\usepackage{graphicx}
\definecolor{darkblue}{rgb}{0,0,0.5} 
\usepackage{transparent}

\newtheorem{thm}{Theorem}[section]
\newtheorem{cor}[thm]{Corollary}
\newtheorem{lem}[thm]{Lemma}
\newtheorem{claim}[thm]{Claim}
\newtheorem{ex}[thm]{Example}
\newtheorem{definition}[thm]{Definition}
\newtheorem{def/thm}[thm]{Definition/Theorem}
\newtheorem{rmk}[thm]{Remark}

\newtheorem{remark}[thm]{Remark}

	\title{Some Examples of Family Floer Mirrors}
	\author{Man-Wai Cheung, Yu-Shen Lin}
\begin{document}

	\maketitle
\begin{abstract}
    In this article, we give explicit calculations for the family Floer mirrors of some non-compact Calabi-Yau surfaces. We compare them with the mirror construction of Gross-Hacking-Keel-Siebert for suitably chosen log Calabi-Yau pairs and with rank two cluster varieties of finite type. In particular, the analytifications of the latter two give partial compactifications of the family Floer mirrors that we computed. 
\end{abstract}

	\tableofcontents
\section{Introduction}

The Strominger-Yau-Zaslow (SYZ) conjecture predicts that Calabi-Yau manifolds have the structure of special Lagrangian fibrations and that their mirrors can be constructed via dual special Lagrangian fibrations. Moreover, the Ricci-flat metrics of Calabi-Yau manifolds receive instanton corrections from holomorphic discs with boundaries on the special Lagrangian torus fibres. The conjecture not only gives a geometric way to construct the mirror, it also gives intuitive reasoning for mirror symmetry, for instance see \cites{FLTZ, LYZ}. 
The SYZ philosophy has become a helpful tool for studying mirror symmetry and many of its implications have been proven.
However, the difficulty of the analysis involving singular special Lagrangian fibres makes progress toward the original conjecture relatively slow (see \cites{CJL,CJL2, L16} for the recent progress). 

To understand instanton corrections rigorously in the mathematical context,
Fukaya \cite{F3} proposed a way to understand the relation between instanton corrections from holomorphic curves/discs and the mirror complex structure via a Floer theoretic approach. Kontsevich-Soibelman \cite{KS1} and Gross-Siebert \cite{GS1} later systematically formulated how to construct the mirror in various settings via algebraic approaches. These approaches opened up a window to understand mirror symmetry intrinsically.   

In their algebraic-geometric approach, Gross-Siebert first constructed affine manifolds with singularities from toric degenerations of Calabi-Yau manifolds. Then there is a systematic way of constructing so-called scattering diagrams, which captures the information of the instanton corrections, on the affine manifolds. The data of the scattering diagrams encode how to glue the expected local models into the mirror Calabi-Yau manifolds.
On the other hand, family Floer homology as proposed by Fukaya \cite{F4} lays out the foundation for realizing mirror symmetry intrinsically from the symplectic geometry point of view. Given a Lagrangian fibration, Fukaya's trick introduced later in Section \ref{sec:fukayatrick} provides  pseudo-isotopies between the $A_{\infty}$ structures of fibres after compensation by symplectic flux. In particular, the pseudo-isotopies induce canonical isomorphisms of the corresponding Maurer-Cartan spaces. The family Floer mirror is then the gluing of the Maurer-Cartan spaces via these isomorphisms. 
Not only have the family Floer mirrors been constructed \cites{A1, T4, Y, Y2}, but
Abouzaid proved that the family Floer functor induces homological mirror symmetry \cites{A2, A3}. 
It is natural to ask if the mirrors constructed via the Gross-Siebert program and the family Floer homology approach coincide or not.  

The following is an expected dictionary connecting the two approaches:

\begin{table}[H]
  \begin{center}
    
    \label{tab:table1}
    \begin{tabular}{|c|c|} 
    \hline
      \textbf{family Floer SYZ} & \textbf{GHKS mirror construction} \\
      \hline
      Large complex structure limit & Toric degeneration or Looijenga pair\\
      \hline
      \makecell{Base of SYZ fibration \\ with complex affine structure} & \makecell{Dual intersection complex \\ of the toric degeneration or $\BGHK$} \\ 
      \hline
      Loci of SYZ fibres bounding MI=0 holomorphic discs &
      Rays in scattering diagram \\
       \hline
      Homology of the boundary of a holomorphic disc &
      Direction of the ray \\
      \hline
      \makecell{Exp of the generating function \\of open Gromov-Witten invariants \\of Maslov index zero} &
      Wall functions attached to the ray \\
      \hline
      \makecell{Coefficients of the superpotential =\\Open Gromov-Witten invariants of Maslov index 2 discs } &
      \makecell{Coefficients of theta functions = \\ Counting of broken lines  }\\
      \hline
      \makecell{Isomorphisms of Maurer-Cartan spaces \\induced by pseudo isotopies } &
     \makecell{Wall crossing transformations} \\
    \hline
      Lemma \ref{compatibility} in this article &
      Consistency of the scattering diagrams \\  
      \hline
    Family Floer mirror &
      GHK/GHKS mirror \\
      \hline
    \end{tabular}
  \end{center}
  \caption{Dictionary between the symplectic and algebraic approaches of mirror construction.}
\end{table}

However, it is hard to have good control over all possible discs in a Calabi-Yau manifold due to the wall-crossing phenomenon. Thus, it is generally hard to write down the family Floer mirror explicitly. 
 In the examples of that family Floer mirrors can possibly be computed in the literature \cite{A4, CLL, Y3}, there always exist torus symmetries, and one can write down all the possible holomorphic discs explicitly. In particular, the loci of Lagrangian fibres bounding Maslov index zero discs do not intersect and thus excludes the presence of more complicated bubbling phenomena. 
 
 In this paper, we engineer some $2$-dimensional examples where the family Floer mirrors can be computed explicitly and realize most part of the above dictionary step by step.
 We first prove that the complex affine structures of the bases of special Lagrangian fibrations coincide with the affine manifolds with singularities constructed in Gross-Hacking-Keel \cite{GHK} from some log Calabi-Yau surfaces. See the similar results for the case of $\mathbb{P}^2$ \cite{LLL}, general del Pezzo surfaces relative to smooth anti-canonical divisors \cite{LLL2}, rational elliptic surfaces \cite{CJL2} and Fermat hypersurfaces \cite{L16}.
 When a Calabi-Yau surface admits a special Lagrangian fibration, it is well-known that the special Lagrangian torus fibres bounding holomorphic discs are fibres above certain affine lines with respect to the complex affine coordinates on the base. Using Fukaya's trick, the second author identified a version of open Gromov-Witten invariants with tropical discs counting \cites{L8,L14}, which lays out a foundation for the connection between family Floer mirrors and Gross-Siebert/Gross-Hacking-Keel mirror. The examples are engineered such that all the wall functions are polynomials. Therefore, there are no convergence issues in the gluing procedure and the complexity is kept to minimal. 
 On the other hand, one can compare it with this process of Gross-Hacking-Keel: we can construct a corresponding log Calabi-Yau pair $(\YGHK,\DGHK)$ such that the induced affine manifold with singularities coincides with the complex affine structure of the base of the special Lagrangian fibration. Then we identify the loci of special Lagrangian fibres bounding holomorphic discs with the rays of the canonical scattering diagram and the corresponding wall-crossing transformations in Gross-Hacking-Keel \cite{GHK}. The technical part is to prove that the family Floer mirror has a partial compactification by the gluing of rigid analytic tori. Notice that directly computing the family Floer mirror of $Y$ would lead to only a small subset of the mirror from the Gross-Hacking-Keel mirror construction. One usually requires a certain renormalization procedure (see \cite{A4}\cite{L4}) and such machinery has not been developed for family Floer mirror yet. Here it is crucial that we use a complete K\"ahler metric on the non-compact Calabi-Yau so that the two mirror constructions can be possibly comparable.  
 
By comparing with the calculation of Gross-Hacking-Keel, we then know that the family Floer mirror has a partial compactification that is the analytification of the mirror of $(Y,D)$ constructed in Gross-Hacking-Keel. The mirror construction of Gross-Hacking-Keel produces a family, where the base can be viewed as the complexified K\"ahler moduli of $Y$. We further determine the distinguished point that corresponds to the family Floer mirror. 
 Let $Y'_*$ be the extremal rational elliptic surface with exactly two singular fibres over $0,\infty\in \mathbb{P}^1$ and singular fibre over $0$ is of type $*$, where $*=II,III,IV$. Let $X'_*$ be the complement of the other singular fibre.
We will let $X_*$ be a suitable hyperK\"ahler rotation of $X'_*$.
 The following is a summary of Theorem \ref{mirror II}, Theorem \ref{mirror III} and Theorem \ref{mirror IV}
\begin{thm} \label{main thm}
The analytification of an $\mathcal{X}$-cluster variety of type $A_2$ (or $B_2$ or $G_2$) or the Gross-Hacking-Keel mirror of a suitable log Calabi-Yau pair $(Y,D)$ is a partial compactification of the
 family Floer mirror of $X_{II}$ (or $X_{III}$ or $X_{IV}$ respectively).
\end{thm}
We will sketch the proof of the $A_2$ case, and the other cases are similar. We first compute the complex affine structure of the SYZ fibration in $X_{II}$ by taking advantage of the fact that its hyperK\"ahler rotation $X'_{II}$ is an elliptic fibration. 
Consider then the local Gromov-Witten invariants computed by the second author \cite{L12}, we use the split attractor flow mechanism to prove that there are exactly five families of SYZ fibres bounding holomorphic discs. The loci parametrizing such fibres in the base of the SYZ fibration are affine lines with respect to the complex affine structure and thus naturally give a cone decomposition on the base (see Figure \ref{fig:1branchcut}). It is not too hard to engineer a Looijenga pair $(Y,D)$ such that its corresponding affine manifold with singularity and cone decomposition are the ones derived from the SYZ fibration of $X_{II}$. In this case, $Y$ is the del Pezzo surface of degree five with $D$ an anti-canonical cycle consisting of five rational curves. The canonical scattering diagram of $(Y,D)$ contains only five rays, and each wall function is a polynomial (see Figure \ref{fig:dp5curve}). Thus, the Gross-Hacking-Keel mirror of $(Y,D)$, after deleting finitely many points, is a gluing of finitely many tori via certain birational transformations. On the other hand, the family Floer mirror $\check{X}$ is the gluing of building blocks of the form $\mathfrak{Trop}^{-1}(U)$, where $U$ is a rational domain. The sets $\mathfrak{Trop}^{-1}(U)$ are "small" subsets in $(\mathbb{G}_m^{an})^2$, but we prove that the inclusion $\mathfrak{Trop}^{-1}(U)\hookrightarrow \check{X}$ can be extended to $\bigg((\mathbb{G}_m^{an})^2\setminus \mathfrak{Trop}^{-1}(0)\bigg)\hookrightarrow \check{X}$. 
Thus, the family Floer mirror $\check{X}$ can also be realized as the gluing of finitely many rigid analytic tori up to some rigid analytic closed subset. Finally, we identify the gluing functions, which are polynomials, from the two mirror constructions from the identification of the affine manifolds. Furthermore, by choosing different branch cuts of on the base of Lagrangian fibration and pushing the singularities to infinity (see Figure \ref{fig:dp5 branch cut'}), we compare the family Floer mirror with the $A_2$-cluster variety.

In addition to being the first example of explicit computation of a family Floer mirror without $S^1$-symmetries and providing a comparison of two mirror constructions, the above theorem has other significance. For instance, it is not clear that the Gross-Hacking-Keel mirror would in general satisfy homological mirror symmetry. On the other hand, the family Floer mirror is designed to prove the homological mirror symmetry conjecture. Abouzaid proved that (when there is no singular fibre) the family Floer mirrors implies homological mirror symmetry \cite{A3}. The comparison of the two mirror constructions provides an intermediate step towards the homological mirror symmetry for Gross-Hacking-Keel mirrors. The current work relies on the fact that the relevant scattering diagrams contain only finitely many rays and the wall functions are all polynomials, and it seems to completely rely on that. However, the more crucial part is the identification (or certain weaker version of equivalence) of the scattering diagrams on the SYZ base and the canonical scattering diagrams. As these are carried out in the cases of $\mathbb{P}^2$ relative to a smooth anti-canonical divisor \cite{L14} and the del Pezzo surface of degree three \cite{BCHL}, the authors expect the equivalence of the two mirror constructions with certain modifications to the treatment of rigid analytic geometry.

\subsection{Structure}
The structure of the paper is arranged as follows: In Section \ref{AG}, we review the definition of cluster varieties and the mirror construction in Gross-Hacking-Keel \cite{GHK} and Gross-Hacking-Keel-Siebert \cites{ghks_cubic}. 
In Section \ref{section: setup}, we describe the surfaces for which we are going to compute the family Floer mirror. 
They come from hyperK\"ahler rotation of the complement of a prescribed fibre in certain rational elliptic surfaces.

In Section \ref{section: family Floer mirror}, we review the family Floer mirror construction and the open Gromov-Witten invariants. 
In Section \ref{section: dp5}, we compute the family Floer mirror of a non-compact Calabi-Yau surface $X_{II}$ explicitly in full detail. Then we compare it with the analytification of the $A_2$-cluster variety. We also compare it with the Gross-Hacking-Keel mirror for a del Pezzo surface of degree five. In particular, the family Floer mirror of $X_{II}$ can be compactified to a del Pezzo surface of degree five via the product formula of the theta functions and prove Theorem \ref{main thm}. 
In Section \ref{section:dp6} and Section \ref{section:dp4}, we sketch the calculation for the family Floer mirror of $X_{III}$ and $X_{IV}$, pointing out the differences from the case of $X_{II}$. 

\section*{Acknowledgment} 
The authors would like to thank Mark Gross and Shing-Tung Yau for the constant support and encouragement. The authors would also like to thank Dori Bejleri, Paul Hacking, Hansol Hong, Chi-Yun Hsu, Laura Friedrickson, Tom Sutherland, Hang Yuan for helpful discussion. The authors appreciate many helpful comments from the reviewers of the earlier version of the manuscript and Sam Bardwell-Evans for revision on the language. 
The first author is supported by NSF grant DMS-1854512. The second author is supported by Simons Collaboration Grant \# 635846 and NSF grant DMS-2204109.

\section{Cluster Varieties and GHK Mirrors}\label{AG}

\subsection{Gross-Hacking-Keel Mirror Construction} \label{RES}
Building upon the work of \cite{GS}\cite{GPS}, Gross-Hacking-Keel \cite{GHK} utilized the enumerative invariants coming from $\mathbb{A}^1$-curves in a log Calabi-Yau surface to recover its mirror family. By the use of theta functions and broken lines, they constructed the mirror family over a version of the complexified ample cone. 
 Heuristically, as the SYZ fibres move to infinity in a direction, the Maslov index zero holomorphic discs with suitable boundary homology classes close up to holomorphic curves with exactly one intersection with some boundary divisor, and hence become $\mathbb{A}^1$-curves. The $\mathbb{A}^1$-curves tropicalize to rays and the counting of $\mathbb{A}^1$-curves will determine the wall functions of the canonical scattering diagram. In this section, we review the mirror construction of Gross-Hacking-Keel \cite{GHK} and Gross-Hacking-Keel-Siebert \cite{ghks_cubic}. 

Consider the pair $(Y,D)$, where $Y$ is a smooth projective rational surface, and $D= D_1 + \cdots + D_n$ is an anti-canonical cycle of rational curves. 
Then $X:= Y \setminus D$ is a non-compact Calabi-Yau surface.\footnote{Note that $X$ is denoted instead by $U$ in \cites{GHK,ghks_cubic}.}
The tropicalization of $(Y,D)$ is a pair $( B_{\mathrm{GHK}}, \Sigma)$, where $B_{\mathrm{GHK}}$ is  homeomorphic to $\R^2$ and has the structure of an integral affine manifold with singularity at the origin, and $\Sigma$ is a decomposition of $B_{\mathrm{GHK}}$ into cones.
The construction of $(B_{\mathrm{GHK}}, \Sigma)$ starts by associating each node $p_{i, i+1}:= D_i \cap D_{i+1}$ with a rank two lattice $M_{i, i+1}$ with basis $v_i$, $v_{i+1}$ and with the cone $\sigma_{i, i+1} \subset M_{i, i+1} \otimes_{\Z} \R$ generated by $v_i$ and $v_{i+1}$. 
Then 
$\sigma_{i, i+1}$ are glued to $\sigma_{i-1, i}$ along the rays $\rho_i:= \R_{\geq} v_i$ to obtain a piecewise-linear manifold $B_{\mathrm{GHK}}$ and a decomposition
\begin{align*}
 \Sigma =  \{ \sigma_{i, i+1} | i = 1, \dots, n \} \cup \{ \rho_i | i = 1, \dots, n \} \cup \{ 0 \}\subseteq \mathbb{R}^2.  
\end{align*}
Define \[
U_i = \mathrm{Int} ( \sigma_{i-1, i} \cup \sigma_{i, i+1}). 
\]
The integral affine structure on $B_{\mathrm{GHK},0} = B_{\mathrm{GHK}} \setminus \{ 0\}$ is defined by the charts
\[
\psi_i : U_i \rightarrow M_{\R}, 
\]
\[
\psi_i(v_{i-1}) = (1,0), \ \psi_i(v_i) = (0,1), \ \psi_i(v_{i+1}) = (-1, -D^2_i), 
\]
with $\psi_i$ linear on $\sigma_{i-1, i} $ and $\sigma_{i, i+1}$. 
It may be worth noting here that at the end of the gluing process, $\rho_{n+1}$ may not agree with $\rho_{1}$. 
This induces an affine structure on $B_{\mathrm{GHK},0} $ which might not extend over the origin when we identify $\rho_{n+1}$ with $\rho_{1}$.
We are going to demonstrate the affine structures explicitly in examples later in this article. 
 Heuristically, one collides all the singular fibres of the SYZ fibration into one to get $\BGHK$.  
Note that if we consider three successive rays $\rho_{i-1}, \rho_i, \rho_{i+1}$, there is the relation
\begin{align} \label{affine relation}
    \psi_i(v_{i-1}) + D^2_i \psi_i( v_i) + \psi_i(v_{i+1}) =0.
\end{align}

Consider a toric monoid $P$. 
A toric monoid $P$ is a commutative monoid whose Grothendieck group $P^{\mathrm{gp}}$ is a finitely generated free abelian group and $P = P^{\mathrm{gp}} \cap \sigma_{P}$, where $\sigma_{P} \subseteq P^{\mathrm{gp}} \otimes_{\Z} \R =P^{\mathrm{gp}}_{\R}$ is a convex rational polyhedral cone.
We will assume that $P$ comes with a homomorphism $\eta: \mathrm{NE}(Y) \rightarrow P$ of monoids, where $NE(Y)$ is the intersection of the cone generated by effective curves with $A_1(Y,\mathbb{Z})$. In later discussion, we will in particular choose $P= \mathrm{NE}(Y)$ and take $\eta$ to be the identity.

Next we define a  multi-valued $\Sigma$-piecewise linear function as a continuous function $\varphi: | \Sigma | \rightarrow P^{\mathrm{gp}}_{\R}$  such that for each $\sigma_{i,i+1} \in \Sigma_{\max}$, $\varphi_i=\varphi |_{\sigma_{i,i+1}}$ is given by an element in $ \Hom_{\Z} (M , P^{\mathrm{gp}}) = N \otimes_{\Z} P^{\mathrm{gp}}$. For each codimension one cone $\rho=\mathbb{R}_+v_i \in \Sigma$ contained in two maximal cones $\sigma_{i-1,i}$ and $\sigma_{i,i+1}$, we have
\begin{align} \label{bending parameter}
\varphi_{i+1} - \varphi_{i} = n_{\rho} \otimes [D_i],
\end{align}
where $n_{\rho} \in N$ is the unique primitive element annihilating $\rho$ and positive on $\sigma_{i,i+1}$. 
Such data $\{ \varphi_i \}$ gives a local system $\mathcal{P}$ on $B_{\mathrm{GHK},0} $ with the structure of $P^{\mathrm{gp}}_{\R}$-principal bundle $\pi: \PP_0 \rightarrow B_{\mathrm{GHK},0} $.
To determine such a local system, we first construct an affine manifold $\PP_0$ by gluing $U_i \times P^{\mathrm{gp}}_{\R}$ to $U_{i+1} \times P^{\mathrm{gp}}_{\R}$ along $(U_i \cap U_{i+1} )\times P^{\mathrm{gp}}_{\R}$ by
\[
(x,p) \mapsto \left(x,p+\varphi_{i+1}(x) - \varphi_i(x) \right).
\]
The local sections $x \mapsto (x, \varphi_i(x))$ patch to give a piecewise linear section $\varphi: B_{\mathrm{GHK},0}  \rightarrow \PP_0$.
Let $\Lambda_B$ denote the sheaf of integral constant vector fields, and $\Lambda_{B ,\R} := \Lambda_{B} \otimes_{\Z} \R$. 
We can then define 
\[
\mathcal{P} := \pi_* \Lambda_{B,\PP_0} \cong \varphi^{-1} \Lambda_{B,\PP_0}
\]
on $B_{\mathrm{GHK},0} $. There is an exact sequence
\begin{align} \label{eqn:exact seq}
    0 \rightarrow \underline{P}^{\mathrm{gp}} \rightarrow \mathcal{P} \xrightarrow{r} \Lambda_B \rightarrow 0
\end{align}
of local systems on $B_{\mathrm{GHK},0} $, where $r$ is the derivative of $\pi$. Then \eqref{bending parameter} is equivalent to 
 \begin{align}\label{lifting of affine relation}
     \varphi_i(v_{i-1})+\varphi_i(v_{i+1})=[D_i]-D_i^2\varphi_i(v_i), \end{align}
which is the lifting of \eqref{affine relation} to $\mathcal{P}$.
We will describe the symplectic meaning of $P$, $P^{\mathrm{gp}}$, and $\mathcal{P}$ in Section \ref{comparison w/ GHK}, particularly see \eqref{identification of short exact sequence}.


Next we define the canonical scattering diagram $\mathfrak{D}_{can}$ on $(B_{\mathrm{GHK}}, \Sigma)$. 
We will first state the definition of a scattering diagram as in \cite{ghks_cubic} and then restrict to the finite case in this article. 
A {ray} in $\mathfrak{D}_{can}$ is a pair $(\frakd, f_{\frakd} )$ where
\begin{itemize}
    \item $\frakd \subset \sigma_{i,i+1}$ for some $i$, called the {support} of a ray, is a ray generated by $a v_i+bv_{i+1} \neq 0$, $a,b \in \Z_{\geq 0}$;
    \item $\log{f_{\frakd}} = \sum_{k \geq 1}k c_k X_i^{-ak} X_{i+1}^{-bk}\in \Bbbk[P][[X_i^{-a}X_{i+1}^{-b}]] $ with  $c_k$ in the maximal ideal $\frakm \subseteq \Bbbk[P]$.\footnote{At first glance, it looks like there is a sign change comparing this expression with the wall functions of the scattering diagram from Floer theory in Definition \ref{thm:wallfun}. However, such discrepancy is explained in the discussion after Lemma \ref{identification space}. }
\end{itemize}
The coefficient $c_k$ is the generating function of relative Gromov-Witten invariants,
  \begin{align*}
     c_k=\sum_{\beta} N_{\beta}z^{\beta},
  \end{align*} where the summation is over all possible classes $\beta\in H_2(Y,\mathbb{Z})$ with incidence relations $\beta.D_i=ak,\beta.D_{i+1}=bk$ and $\beta.D_j=0$, for $j\neq i,i+1$. The coefficient $N_{\beta}$ is the counting of $\mathbb{A}^1$-curves in such a class $\beta$. We will refer the readers to \cite{GHK}*{Section 3} for technical details of the definition of relative Gromov-Witten invariants and remark that this is mostly replaced by logarithmic Gromov-Witten theory nowadays \cite{GS3}\cite{C3}.

Roughly speaking, a {scattering diagram} for the data $(B_{\mathrm{GHK}}, \Sigma)$ is a set $\frakD = \{ (\frakd, f_{\frakd}) \}$ such that, for every monomial ideal $I$ with $k[P]/I$ being Artinian, $f_{\frakd} (\mbox{ mod }I)$ is a finite sum for each wall function $f_{\frakd}$ and there are only finitely many $f_{\frakd} \neq 1 (\mbox{ mod }I)$ . For notation simplicity, we will write $\mathfrak{D}=\mathfrak{D}_{can}$ later. 
Note that scattering diagrams may give a refinement of the original fan structure given by $\Sigma$. We will call the maximal cones of this refinement chambers.

The scattering diagram will lead to a flat family over $\Spec A_I$, where $A_I = \Bbbk[P]/I$ and $I \subseteq \Bbbk[P]$ is any monomial ideal with $\Bbbk[P]/I$ Artinian. 
Now consider each $\rho_i$ as the support of a ray $(\rho_i, f_i)$ in $\frakD_{can}$.
Define 
\begin{align} \label{can isom}
    R_{i,I} &:= A_I [X_{i-1}, X_i^{\pm 1}, X_{i+1}] / ( X_{i-1} X_{i+1} - z^{[D_i]} X_i^{-D_i^2} f_i ), \notag \\
    R_{i, i+1, I}& := A_I [X_i^{\pm 1}, X_{i+1}^{\pm 1}]\cong (R_{i,I})_{X_{i+1}},
\end{align}
where $z^{[D_i]}$ is the monomial in $\Bbbk[P]$ corresponding to the class of $[D_i]$. 
Let \[U_{i, I} :=  \Spec R_{i,I} \text{ and  }U_{i, i+1, I} := \Spec R_{i, i+1,I}.\] 
One would then like to glue $U_{i,I}$ and $U_{i+1, I}$ over the identified piece $U_{i, i+1, I}$ to obtain a scheme $X_I^{\circ}$ flat over $\Spec A_I$.

To obtain a quasi-affine $X_I^{\circ}$, one needs to consider an automorphism $R_{i,i+1,I}$, called the {path ordered product}, associated to a path $\gamma: [0,1] \rightarrow \Int (\sigma_{i, i+1})$. 
Suppose $\gamma$ crosses a given ray $\left(\frakd = \R_{\geq 0} (av_i + bv_{i+1}), f_{\frakd}\right)$. 
The $A_I$-algebra homomorphism $\theta_{\gamma, \frakd}: R_{i, i+1, I} \rightarrow R_{i, i+1, I}$ is defined by $X_i^{k_i}X_{i+1}^{k_{i+1}} \mapsto X_i^{k_i}X_{i+1}^{k_{i+1}}  f_{\frakd}^{\pm (-bk_i + ak_{i+1} )}$, where the sign $\pm$ is positive if $\gamma$ goes from $\sigma_{i-1, i}$ to $\sigma_{i, i+1}$ when passing through $\frakd$ and is negative if $\gamma$ goes in the opposite direction. One can see this is the same as the wall crossing transformation stated in \eqref{eq:wallcrossing}. 
If $\gamma$ passes through more than one ray, we define the path ordered product by composing each individual path ordered product of each ray in the order $\gamma$ passes them. 
Choosing a path $\gamma$ by starting very close to $\rho_i$ and ending near $\rho_{i+1}$ in $\sigma_{i, i+1}$, we see that $\gamma$ passes all the rays in $\sigma_{i, i+1}$.
We then define $X_{I, \frakD}^{\circ}=\coprod_i U_{i,I}/\sim$ with the gluing given by 
\[
U_{i,I} \hookleftarrow U_{i, i+1, I} \xrightarrow{\theta_{\gamma, \frakD}} U_{i, i+1, I} \hookrightarrow U_{i+1, I}. 
\]
 Let $X^{\circ}_{\mathfrak{D}}=\varprojlim X^{\circ}_{I,\mathfrak{D}}$, which is defined over $\hat{k[P]}$ and $\hat{k[P]}$ is some completion of $k[P]$ with respect to the $I$-adic topology. 
The following observation is important later for the comparison between the Gross-Hacking-Keel mirror and the family Floer mirror in the examples considered in this paper. The observation should be well-known to experts but authors cannot find such statement in the literature. So we include the proof for being self-contained.
\begin{lem} \label{repacement}
	Assume that $(Y,D)$ is a Looijenga pair such that 
	\begin{enumerate}
		\item the intersection matrix $(D_i\cdot D_j)$ is not negative definite.
		\item There are only finitely many rays with non-trivial wall functions in the canonical scattering diagram of $(Y,D)$.
		\item All the wall-functions are polynomials. 
	\end{enumerate}
Then $X^{\circ}_{\mathfrak{D}}$ is defined over $k[P]$ and a generic closed fibre of $X^{\circ}_{\mathfrak{D}}\rightarrow \mbox{Spec}k[P]$ is a partial compactification of gluing of finitely many tori via birational transformation by adding finitely many points. 
\end{lem}
\begin{proof}
	The first assumption is to guarantee that $X^{\circ}_{\mathfrak{D}}$ is defined over $k[P]$ \cite[Theorem 0.2]{GHK}. Let $U_i:=\varprojlim U_{i,I}$ and $U_{i,i+1}:=\varprojlim U_{i,i+1,I}$. Then the one has the canonical isomorphism 
	\begin{align*}
	  \mbox{Spec}(k[P])\times \mathbb{G}_m^2\cong  U_{i,i+1}\cong \{X_{i+1}\neq 0\}\subseteq U_i
	\end{align*} from \eqref{can isom}. Now over a generic closed point of $\mbox{Spec}k[P]$, the fibre of $U_i\rightarrow \mbox{Spec}(k[P])$ is given by $\{X_{i-1}X_{i+1}=z^{[D_i]}X_i^{-D_i^2}f_i\}$\footnote{Over a generic closed point, we have $z^{-[D_i^2]}\in \mathbb{C}^*$.} which is the union of two toric $(\mathbb{G}_m)^2_{(X_i,X_{i+1})},(\mathbb{G}_m)^2_{(X_{i-1},X_{i})}$ and finitely many points $\{X_{i-1}=X_{i+1}=0,f_i=0\}$\footnote{This is because that $f_i$ is a polynomial.} with embeddings
	 \begin{align*}
	    (\mathbb{G}_m)^2_{(X_i,X_{i+1})}\hookrightarrow \{X_{i-1}X_{i+1}=z^{[D_i]}X_i^{-D_i^2}f_i\} \\
	    (X_i,X_{i+1})\mapsto (X_{i+1}^{-1}z^{[D_i]}X_i^{-D_i^2}f_i,X_i,X_{i+1})
	    	 \end{align*} and 
	   \begin{align*}
	  (\mathbb{G}_m)^2_{(X_{i-1},X_{i})}\hookrightarrow \{X_{i-1}X_{i+1}=z^{[D_i]}X_i^{-D_i^2}f_i\} \\
	  (X_i,X_{i+1})\mapsto (X_{i-1},X_i,X_{i-1}^{-1}z^{[D_i]}X_i^{-D_i^2}f_i).
	  \end{align*} 	 
	It is easy to see that $\{X_{i-1}X_{i+1}=z^{[D_i]}X_i^{-D_i^2}f_i\} \setminus \{X_{i-1}=X_{i+1}=0,f_i=0\}$ is the gluing of two tori via the birational transformation 
	 \begin{align*}
	     (\mathbb{G}_m)^2_{(X_{i-1},X_i)}&\dashrightarrow (\mathbb{G}_m)^2_{(X_{i},X_{i+1})} \\
	       (X_{i-1},X_i)&\mapsto (X_i, X_{i-1}^{-1}z^{[D_i]}X_{i}^{-D_i^2}f_i).
	 \end{align*}
	The gluing map $\theta_{\gamma,\mathfrak{D}}$ restricts on the fibre is a composition of birational map and the lemma follows. 
	
    We also point out that one can replace $(Y,D)$ by a minimal resolution on the corners of $D$ such that all the $\mathbb{A}^1$-curves are transversal to the toric boundary divisors. The integral affine manifold $\BGHK$ remains the same \cite[Lemma 1.6]{GHK}. The corresponding relative/log Gromov-Witten invariants are the same \cite{GPS}\cite[Theorem 1.1.1]{AJ}. Thus, the corresponding canonical scattering diagrams are the same under such birational modification.
	
\end{proof}


The next step in \cites{GHK, ghks_cubic} is considering the broken lines to define consistency of a scattering diagram and to construct the theta functions.
Since we will focus on the cases with only finitely many rays in the canonical scattering diagram and wall functions are polynomials in this paper, we can make use of path-ordered products directly without the use of broken lines.
For the definition of broken lines and theta functions, one can refer to \cite[Section 2.2]{GHK}. 
Instead, to define consistency, we can extend the definition of path ordered product to a path $\gamma: [0,1]\rightarrow B_0(\Z)$ with starting point $q$, and endpoint $Q$, where neither $q$ and $Q$ lies on any ray. 
Then the path ordered product $\theta_{\gamma, \frakD}$ can be defined similarly by composing $\theta_{\gamma, \frakd}$'s of the walls $\frakd$ passed by $\gamma$.
Then the canonical scattering diagram $\frakD$ is {consistent} in the sense that the path ordered product $\theta_{\gamma, \frakD}$ only depends on the two end points $q$ and $Q$. The two notions of consistency of scattering diagrams coincide \cite{CPS} (see also \cite[Section 3.2]{GHK}).

For a point $q \in B_0(\Z)$, let us assume $q = av_{i-1} + bv_{i} \in \sigma_{i-1, i}$ and associate the monomial $X_{i-1}^{a} X_{i}^{b}$ to $q$.
Consider now another point $Q\in B_0\setminus \bigcup_{\mathfrak{d}\in \mathfrak{D}_{can}}\mbox{Supp}\mathfrak{d}$ and a path $\gamma$ from $q$ to $Q$, then $\vartheta_{q,Q}=\varprojlim\theta_{\gamma,\mathfrak{D}}(X_{i-1}^aX_i^b) (\mbox{ mod }I)$.
This is well-defined because the canonical scattering diagram $\frakD$ is consistent.
We will define $\vartheta_{0,Q } = \vartheta_0 =1$. 
Thus the $\vartheta_{q, Q}$ for various $Q$ can be glued to give the global function $\vartheta_q \in \Gamma( X_{I, \frakD}^{\circ}, \mathcal{O}_{X_{I, \frakD}^{\circ}})$.
Then, by \cite{GHK}*{Theorem 2.28}, $X_{I, \frakD}:=\Spec \Gamma \left( X_{I, \frakD}^{\circ}, \mathcal{O}_{X_{I, \frakD}^{\circ}} \right) $ is a partial compactification of $X_{I, \frakD}^{\circ}$. 

\subsection{Cluster varieties}\label{sec: cluster}

We will first recall some notations used in the definition of a cluster varieties. 
A \emph{fixed data} consists of a lattice $N$ with a skew-symmetric bilinear form $\{ \cdot , \cdot \} : N \times N \rightarrow \Q$,
an index set $I$ with $|I| = \rank N$,
positive integers $d_i$ for $i \in I$,
a sublattice $N^{\circ} \subseteq N$ of finite index with some integral properties,  the dual lattice $M = \Hom (N, \Z)$ and the corresponding $M^{\circ} = \Hom(N^{\circ}, \Z)$. 
One can refer to \cite{GHK_bir} for the full definition of fixed data. 
Consider $N_{\R} = N \otimes \R$ and $M_{\R} = M \otimes \R$. 

Given this fixed data, a $\emph{seed data}$ for this fixed data is $\textbf{s} := ( e_i \in N \mid i \in I)$,
where  $\{ e_i \}$ is a basis for $N$.
The basis for $M^{\circ}$ would then be $f_i = \frac{1}{d_i} e_i^*$ . 
One can then associate the seed tori 
\[\mathcal A_{\textbf{s}} = T_{N^{\circ}} = \Spec \Bbbk [M^{\circ}], \quad
\quad \mathcal X_{\textbf{s}} = T_{M} = \Spec \Bbbk [N].\]
We will denote the coordinates as $X_i = z^{e_i}$ and $A_i = z^{f_i}$ and they are called the \emph{cluster variables}. 
Similar to the definition of cluster algebras, there is a procedure, called \emph{mutation}, to produce a new seed data $\mu(\textbf{s})$ from a given seed $\textbf{s}$. 
The mutation formula is stated in \cite[Equation 2.3]{GHK_bir} which we will skip here. 
The essence is that we will obtain new seed tori $\mathcal A_{\mu(\textbf{s})}$, $\mathcal X_{\mu(\textbf{s})}$ from the mutated seed. 
Between the tori, there are birational maps $\mu_{\mathcal X}: \mathcal X_{\textbf{s}} \dashrightarrow \mathcal X_{\mu(\textbf{s})}$, $\mu_{\mathcal A}: \mathcal A_{\textbf{s}} \dashrightarrow \mathcal A_{\mu(\textbf{s})}$ as

\[
\mu_{\mathcal X, k}: \mathcal X_{\textbf{s}} \dashrightarrow \mathcal X_{\mu_({\textbf{s}})},
\]
\[
\mu_{\mathcal A, k}: \mathcal A_{{\textbf{s}}} \dashrightarrow \mathcal A_{\mu_({\textbf{s}})}.
\]
via pull-back of functions
\begin{equation} \label{eq:xmutatepull}
\mu_{\cX,k}^* (z^n )= z^n (1+ z^{e_k} )^{-[n,e_k]}  \text{ for } n \in N, 
\end{equation}
\begin{equation}  \label{eq:amutatepull}
\mu_{\cA,k}^* (z^m) = z^m (1+ z^{v_k} )^{-\langle d_k e_k, m \rangle}  \text{ for } m \in M^{\circ}. 
\end{equation}
Note that those birational maps are basically the mutations of cluster variables as in Fomin and Zelevinsky \cite{cluster1}. 
Let $V_{\mathcal A}$ be an union of tori glued by $\mathcal A$-mutations of the form $\mu_{\mathcal A}$.
A smooth scheme $V$ is a {\emph{$\cA$-cluster variety}} if there is a birational map $ \mu: V \dashrightarrow V_{\mathcal A}$ which is an isomorphism outside codimension two subsets of the domain and range. {\emph{$\cX$-Cluster varieties}} are defined analogously. 
Gross-Hacking-Keel-Kontsevich \cite{GHKK} constructed cluster scattering diagrams and showed that the cluster monomials can be expressed as theta functions defined from the cluster scattering diagrams. 
The collections of the theta functions form the bases of the (middle) cluster algebras defined by Fomin-Zelevinsky \cite{cluster1}.

Now we will focus on the three dimension two cases relevant in this paper:
the fixed data\footnote{We follow the definition of fixed data as in \cite{GHK_bir}} are given by the bilinear form 
$\begin{pmatrix}
  0 & 1\\
  -1 & 0
\end{pmatrix}$ and the scalars $d_1,  d_2 \in \N$. We say that the cluster algebra is of 
\begin{enumerate}
	\item Type $A_2$ if $d_1=d_2=1$.
	\item Type $B_2$ if $d_1=1,d_2=2$.
	\item Type $G_2$ if $d_1=1,d_2=3$.
\end{enumerate}
In these cases, the relations between the generators $\vartheta_i$ in the cluster complex of the (middle) $\mathcal X$ cluster algebras can be expressed as 
\begin{align} \label{eq:rank2eq}
     \vartheta_{i-1} \cdot \vartheta_{i+1} &= 
     \begin{cases}
(1+\vartheta_i)^{d_1}, & \mbox{if $i$ is odd} \\
(1+\vartheta_i)^{d_2}, & \mbox{if $i$ is even},
\end{cases}
\end{align}
where $i \in \Z$. 
Also the so-called injectivity assumption holds and the skew-symmetric forms are of full rank in these cases. Thus, the $\cA$-scattering diagrams in Figure \ref{fig:rank2} and $\cX$-cluster scattering diagrams are defined \cite{GHKK}\cite{broken}. We include them in Figure \ref{fig:rank2} and Figure \ref{fig:Xrank2} below. 
\begin{figure}[h!]   
	\begin{center} 
		\centerline{\includegraphics[scale= .7]{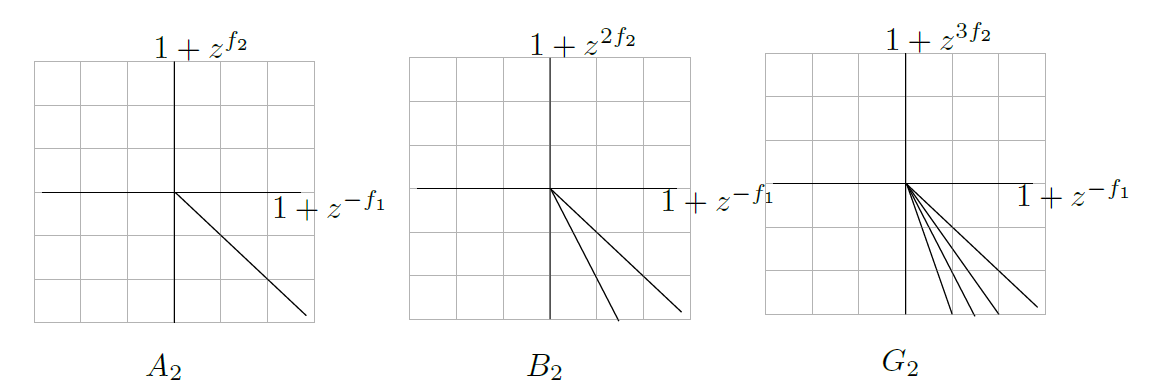}}
		\caption{$\cA$-scattering diagrams for rank 2 finite type \cite{vianna}.} \label{fig:rank2}
	\end{center}
\end{figure}
\begin{figure}[h!]
	\begin{center}
		\centerline{\includegraphics[scale= .7]{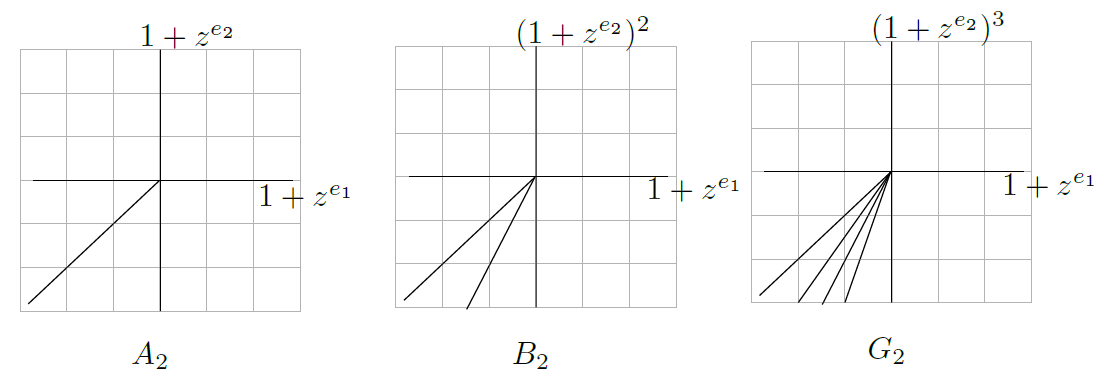}}
		  \caption{$\cX$-scattering diagrams for rank 2 finite type \cite{vianna}.} \label{fig:Xrank2}
	\end{center}
\end{figure}
 In particular, in the cases of concerned (type $A_2,B_2,G_2$) are all of finite type, i.e. the walls of the scattering diagrams do not accumulated and the walls of the scattering diagram divided $\mathbb{R}^2$ into finitely many chambers. The algebra generated by theta functions coincides with ring of regular functions on the following geometry: each tori corresponds a chamber of the scattering diagram and the a gluing birational map associate to adjacent chambers determined by the wall function of the wall \cite[Theorem 0.12 (2)]{GHKK}.\footnote{The theorem shows that the cluster algebra is contained in the algebra generated by theta functions. The finiteness assumption and the definition of cluster algebra implies that the other direction of the inclusion.}
 Since the $\mathcal{X}$-cluster varieties in these cases are affine \cite[Theorem 0.17]{GHKK}, 
we conclude that the $\cX$-cluster varieties, a priori defined as gluing to tori associated to seeds, can also be derived from the gluing of tori with each tori corresponds to a chamber in the corresponding scattering diagrams. Notice that the latter perspective is similar to the mirror construction of Gross-Hacking-Keel \cite{GHKlog}. 

\section{Set-Up of the Geometry} \label{section: setup}
We will start from the geometry we will be computing the family Floer mirrors. 
Consider $Y'$ an extremal rational elliptic surface with its singular configuration one of the following: $II^*II$, 
$ III^*III$, $IV^*IV$. Here we use the notation in the Kodaira classification of the singular fibres. Let $Y'=Y'_*$, where $* = II, III$ or $IV$ is the type of singular fibre over zero.
These rational elliptic surfaces can be constructed explicitly. 
For instance, consider the compact complex surface given by the minimal resolution of  
  \begin{align}  \label{15}
     \{ ty^2z=tx^3+uz^3 \}\subseteq \mathbb{P}^2_{(x,y,z)}\times \mathbb{P}^1_{(u,t)}.
  \end{align} 
By the Tate algorithm \cite{T5}, this is an elliptic surface with a type $II^*$ singular fibre over $u=\infty$. Straight-forward calculation shows that it has singular configuration $II^*II$ and produces the extremal rational elliptic surface $Y'_{II}$. By Castelnuovo's criterion of rationality, $Y'$ is rational and thus a rational elliptic surface. 
The other extremal rational elliptic surfaces $Y'_{III},Y'_{IV}$ can be constructed similarly with the corresponding affine equations below \cite{Miranda}*{p.545}:  
  \begin{align*}
     y^2&=x^4+u\\ 
     y^2&=x^3+t^2s^4  
  \end{align*} 

Recall that any rational elliptic surface $Y'$ has canonical bundle $K_{Y'}=\mathcal{O}_{Y'}(-D')$, where $D'$ denotes an elliptic fibre. Thus, there exists a meromorphic $2$-form $\Omega'$ with a simple pole along an assigned fibre which is unique up to a $\mathbb{C}^*$-scaling. In particular, the non-compact surface $X'=Y'\backslash D'$ can be viewed as a non-compact Calabi-Yau surface. Indeed,
  \begin{thm} \cite{H}
  	 There exists a Ricci-flat K\"ahler form $\omega'$ on $X'$ for any choice of the fibre $D'$. In particular, $2\omega'^2=\Omega'\wedge \bar{\Omega'}$, and $X'$ is hyperK\"ahler. 
  \end{thm} 
  The Ricci-flat K\"ahler form may not be unique (even in a given cohomology class \cite{CV}), but we will just fix one for later purposes and family Floer mirror computations will not be affected by such choices.   
 Consider $D'_*$ to be the infinity fibre in $Y'_{*}$ and denote the hyperK\"ahler rotation of $X'_{*}=Y'_{*}\setminus D'_*$ by $X=X_{*}$. Explicitly, $X_*$ has the same underlying space as $X'_{*}$ and is equipped with the K\"ahler form and the holomorphic volume form 
  \begin{align}\label{HK}
   \omega&=\mbox{Re}\Omega' \notag \\
   \Omega&=\omega'-i\mbox{Im}\Omega'.
  \end{align}
   Then the elliptic fibration $X'_*\rightarrow \mathbb{C}$ becomes the special Lagrangian fibration $X_* \rightarrow B$, where $B\cong \mathbb{R}^2$ \cite{HL} (see the diagram below) topologically. Let $B_0\cong \mathbb{R}^2\setminus \{0\}$ be the complement of the discriminant locus. We will refer the readers to \cite{CJL}*{P.35} for more explicit calculation of hyperK\"ahler (HK) rotation. We will omit the subindex when there is no confusion.    

\begin{center}
    \begin{tikzcd}
Y'  \arrow[d] \arrow[r, hookleftarrow]&   X' =  \arrow[d]  Y' \setminus D' \arrow[r, leftrightarrow, "\mathrm{HK}"]   & X  \arrow[d]\\
\PP^1    \arrow[r, hookleftarrow] & B \cong \C \arrow[r] & B \cong \R^2
\end{tikzcd}
\end{center}

 The fibrewise relative homology $H_2(X,L_u)$ glues to a local system of lattices $\Gamma$ over $B_0$. For any relative class $\gamma\in H_2(X,L_u)$, we define the {central charge} 
 \begin{align*}
     Z_{\gamma}(u):=\int_{\gamma}\Omega'
 \end{align*}
to be a function from the local system $\Gamma$ to $\mathbb{C}$. Notice that $B_0\cong\mathbb{C}^*$ admits a complex structure structure and $Z_{\gamma}$ is locally\footnote{Since there is monodromy, it is a multi-valued function on $B_0$} a holomorphic function in $u$ by \cite[Corollary 2.8]{L4}. The central charge will help to locate the special Lagrangian torus fibre bounding holomorphic discs in Section \ref{section: BPS rays}.

\subsection{Affine Structures of the Base} \label{affine}
Let $(X,\omega)$ be a K\"ahler surface with holomorphic volume form $\Omega$ satisfying $2\omega^2=\Omega\wedge\bar{\Omega}$. Assume that $X$ admits a special Lagrangian fibration $\pi: X\rightarrow B$, possibly with singular fibres with respect to $(\omega,\Omega)$, i.e., $\omega|_L=\mbox{Im}\Omega|_L=0$ for any fibre $L$.
We will use $L_u$ to denote the fibre over $u\in B$.  There are two natural integral affine structures defined on $B_0$ by Hitchin \cite{H2}. One is called the symplectic affine structure and the other one is the complex affine structure, both are described below. Given a reference point $u_0\in B_0$ and a choice of the basis $\check{e}_1,\check{e}_2\in H_1(L_{u_0},\mathbb{Z})$, we will define the local affine coordinates around $u_0$. For any $u\in B_0$ in a contractible neighborhood of $u_0$, we choose a path $\phi$ contained in such neighborhood connecting $u,u_0$. Let $C_k$ be the $S^1$-fibration over the image of $\phi$ such that the fibres are in the homology class of the parallel transport of $\check{e}_k$ along $\phi$. Then the local symplectic affine coordinates are defined by
   \begin{align}\label{affine coor}
      x_k(u)=\int_{C_k}\omega.
   \end{align} It is straight-forward to check that the transition functions fall in $GL(2,\mathbb{Z})\rtimes\mathbb{R}^2$,\footnote{  If there is a global Lagrangian section, then the transition functions fall in $GL(2,\mathbb{Z})$. } and thus the above coordinates give an integral affine structure on $B_0$. Replacing $\omega$ in (\ref{affine coor}) by $\mbox{Im}\Omega$ gives the complex integral affine coordinates $\check{x}_k(u)=\int_{C_k}\mbox{Im}\Omega$. 
    \begin{remark}\label{rk:affine}
   By the construction, primitive classes $\check{e}\in H_1(L_u,\mathbb{Z})$ are in one-to-one correspondence with the primitive integral vectors in $(T_{\mathbb{Z}}B_0)_u$. Indeed, each $v\in T_uB_0$ has a corresponding functional $  \int_{-}\iota_{v}\mbox{Im}\Omega$ on $H_1(L_u,\mathbb{Z})$ and thus corresponds to a primitive element in $H_1(L_u,\mathbb{Z})$ via its natural symplectic pairing and the Poincare duality. 
   \end{remark}
 
   \begin{remark} \label{affine line}
 Since special Lagrangians have vanishing Maslov classes, all the special Lagrangian torus fibres only bound Maslov index zero discs and the assumption of the Lemma \ref{compatibility} holds. In general, it is hard to control all the bubbling of pseudo-holomorphic discs (of Maslov index zero). However, when the Lagrangian fibration is further special, the special Lagrangian fibres bounding holomorphic discs of a fixed relative class maps to an affine line with respect to the complex affine structure by the fibration map $\pi$. Indeed, if $u_t$ is a path in $B_0$ such that
 each $L_{u_t}$ bounds a holomorphic disc in class $\gamma$ (we identify the relative classes via parallel transport along the path $u_t$), then $\int_{\gamma}\mbox{Im}\Omega=0$ along $u_t$ since $\gamma$ can be represented by a holomorphic cycle and cannot support a non-zero holomorphic $2$-form. In particular, $u_t$ is contained in an affine line with respect to the complex affine structure, which locally denoted by $l_{\gamma}$.\footnote{The notation $\gamma$ and thus $l_{\gamma}$ are only locally defined since the monodromy of the fibration acts non-trivially on the local system $\Gamma$} Notice that $l_{\gamma}$ is naturally oriented such that the symplectic area of $\gamma$ is increasing along $l_{\gamma}$. 
From the expected dictionary in the introduction, these affine lines correspond to rays in the scattering diagrams, and Lemma \ref{compatibility} translates to the consistency of scattering diagrams.
\end{remark}
   
   We will use both integral affine structures later: the complex affine structures will be used to locate the fibres bounding holomorphic discs (see Section \ref{section: BPS rays}) while the symplectic affine structures will be used to define the family Floer mirrors (see Section \ref{sec:construction}).

 \section{Floer Theory and Family Floer Mirror} \label{section: family Floer mirror}
 In this section, we will talk about the background for the explicit calculation of a family Floer mirror in Section \ref{section: dp5}. We will review the construction of family Floer mirror of Tu \cite{T4} in Section \ref{sec:construction}. 
Given a Lagrangian torus fibration $X\rightarrow B$ with fibre $L_u$ over $u\in B$, Fukaya-Oh-Ohta-Ono \cite{FOOO} constructed $A_{\infty}$ algebras on de Rham cohomologies of the fibres. 
 Assuming that the fibres are unobstructed, then the exponential of the corresponding Maurer-Cartan spaces are the analogue of the dual tori for the original Lagrangian fibration. 
 Then the family Floer mirror is the gluing of these exponential of Maurer-Cartan spaces as a set. The gluing morphisms, known as the "quantum correction" to the mirror complex structure, are induced by the wall-crossing of the Maurer-Cartan spaces. Such wall-crossing phenomena receive contributions from the holomorphic discs of Maslov index zero with boundaries on SYZ fibres. We review the relation of the open Gromov-Witten invariants and the gluing morphisms in Section \ref{sec:fukayatrick}. To have a better understanding of the gluing morphisms, in Section \ref{section: BPS rays} we study the location of all fibres with non-trivial open Gromov-Witten invariants within the geometry discussed in Section \ref{section: setup} by taking advantage of the special Lagrangian boundary conditions.

  \subsection{Fukaya's Trick and Open Gromov-Witten Invariants}\label{sec:fukayatrick}
  We will first review the so-called Fukaya's trick, which is a procedure for comparing the variation of the $A_{\infty}$ structures of a Lagrangian and those of its nearby deformations. We will use Fukaya's trick to detect the open Gromov-Witten invariants.
  
  Let $X$ be a symplectic manifold with Lagrangian fibration $X\rightarrow B$.
Recall the definition of the {Novikov field},
  \begin{align*}
    \Lambda:=\left\{\sum_{i\in \mathbb{N}}c_iT^{\lambda_i} \middle| \quad \lambda_i\in \mathbb{R},\lim_{i\rightarrow \infty}\lambda_i=\infty, c_i\in \mathbb{C}\right\}.
  \end{align*}  Denote  $\Lambda_+=\{\sum_{i\in \mathbb{N}}c_iT^{\lambda_i}|\lambda_i>0\}$
  and $\Lambda^*=\Lambda\backslash \{0\}$. There is a natural discrete valuation 
  \begin{align*}
      \mbox{val}: \Lambda^* \longrightarrow &\mathbb{R} \\
                  \sum_{i\in \mathbb{N}}c_iT^{\lambda_i}\mapsto & \lambda_{i_0},
  \end{align*}
where $i_0$ is the smallest $i$ with $c_i\neq 0$. 
One can extend the domain of $\mbox{val}$ to $\Lambda$ by setting $\mbox{val}(0)=\infty$. 

Let $B_0$ be the complement of the discriminant locus of the special Lagrangian fibration and $L_u$ be the fibre over $u\in B_0$. Fix an almost complex structure $J$ which is compatible with $\omega$. For our later purposes, $X$ will be K\"ahler and we will take $J$ to be its complex structure. Given a relative class $\gamma\in H_2(X,L_u)$, we use $\mathcal{M}_{\gamma}((X,J),L_u)$ to denote the moduli space of stable $J$-holomorphic  discs in relative class $\gamma$ with respect to the almost complex structure $J$. We may omit $J$ if there is no confusion. Fukaya-Oh-Ohta-Ono \cite{FOOO} constructed a filtered unital $A_{\infty}$ structure $\{m_k\}_{k\geq 0}$ on $H^*(L_u,\Lambda)$ by considering the boundary relations of $\mathcal{M}_{\gamma}((X,J),L_u)$, for all $\gamma\in H_2(X,L_u)$. We will assume that there exist only Maslov index zero discs in $X$. If we restrict to $\mbox{dim}_{\mathbb{R}}X=4$, then the moduli space $M_{\gamma}((X,J),L_u)$ has virtual dimension negative one. In particular, the Maurer-Cartan space associate to the $A_{\infty}$ structure is simply $H^1(L_u,\Lambda_+)$. 
  
  Now we explain {Fukaya's trick}.
  Given $p\in B_0$ and a path $\phi$ contained in a small neighborhood of $p$ such that  $\phi(0)=u_-, \phi(1)=u_+$, one can choose a $1$-parameter family of paths $\phi_s(t)$ such that $\phi_0(t)=\phi(t), \phi_1(t)=p$ and
       $\phi_s(t)$ is contained in a small enough neighborhood of $p$. 
It is illustrated as follows:
\begin{figure}[H]
\centering
\def\svgwidth{150pt}
\begingroup%
  \makeatletter%
  \providecommand\color[2][]{%
    \errmessage{(Inkscape) Color is used for the text in Inkscape, but the package 'color.sty' is not loaded}%
    \renewcommand\color[2][]{}%
  }%
  \providecommand\transparent[1]{%
    \errmessage{(Inkscape) Transparency is used (non-zero) for the text in Inkscape, but the package 'transparent.sty' is not loaded}%
    \renewcommand\transparent[1]{}%
  }%
  \providecommand\rotatebox[2]{#2}%
  \newcommand*\fsize{\dimexpr\f@size pt\relax}%
  \newcommand*\lineheight[1]{\fontsize{\fsize}{#1\fsize}\selectfont}%
  \ifx\svgwidth\undefined%
    \setlength{\unitlength}{280.15807144bp}%
    \ifx\svgscale\undefined%
      \relax%
    \else%
      \setlength{\unitlength}{\unitlength * \real{\svgscale}}%
    \fi%
  \else%
    \setlength{\unitlength}{\svgwidth}%
  \fi%
  \global\let\svgwidth\undefined%
  \global\let\svgscale\undefined%
  \makeatother%
  \begin{picture}(1,0.78302071)%
    \lineheight{1}%
    \setlength\tabcolsep{0pt}%
    \put(0,0){\includegraphics[width=\unitlength,page=1]{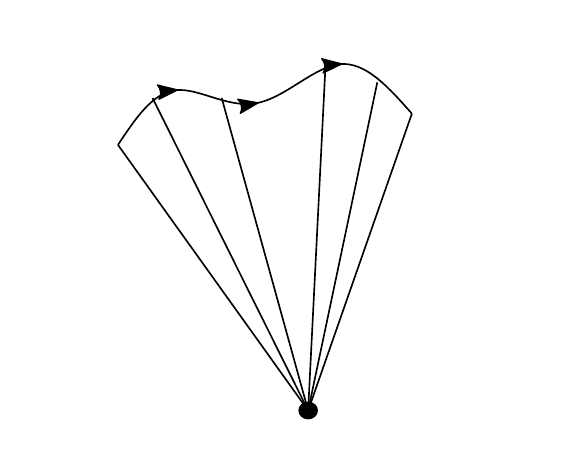}}%
    \put(0.53913857,0.03408317){\color[rgb]{0,0,0}\makebox(0,0)[lt]{\lineheight{1.25}\smash{\begin{tabular}[t]{l}$p$\end{tabular}}}}%
    \put(0.34689794,0.67438226){\color[rgb]{0,0,0}\makebox(0,0)[lt]{\lineheight{1.25}\smash{\begin{tabular}[t]{l}$\phi(t)$\end{tabular}}}}%
    \put(0.05953242,0.50101906){\color[rgb]{0,0,0}\makebox(0,0)[lt]{\lineheight{1.25}\smash{\begin{tabular}[t]{l}$\phi(0)$\\=$u_-$\end{tabular}}}}%
    \put(0.7211385,0.5471822){\color[rgb]{0,0,0}\makebox(0,0)[lt]{\lineheight{1.25}\smash{\begin{tabular}[t]{l}$\phi(1)$\\$=u_+$\end{tabular}}}}%
    \put(0,0){\includegraphics[width=\unitlength,page=2]{path.pdf}}%
  \end{picture}%
\endgroup%

    \caption{Fukaya's trick} \label{path}
\end{figure}

  Then there exists a $2$-parameter family of fibrewise preserving diffeomorphisms $f_{s,t}$ of $X$ such that
  \begin{enumerate}
      \item $f_{s,1}=id$. 
      \item $f_{s,t}$ sends $L_{\phi_s(t)}$ to $L_p$.
      \item $f_{s,t}$ is the identity outside the preimage of a compact subset of $B_0$.
  \end{enumerate}
   Then $J_t=(f_{1,t})_*J$ is a $1$-parameter family of almost complex structures tamed with respect to $\omega$ when $\phi$ is contained in a small enough neighborhood of $p$. There is a canonical isomorphism of moduli spaces of holomorphic discs 
        \begin{align} \label{two moduli}
 \mathcal{M}_{k,\gamma} ( (X,J),L_{\phi(t)})
 \cong \mathcal{M}_{k,(f_{1,t})_*\gamma}\big((X,(f_{1,t})_*J), L_p\big)
         \end{align}
which carries over to the identification of the Kuranishi structures. However, the $A_{\infty}$ structures from two sides of (\ref{two moduli}) are not the same under the parallel transport $H^*(L_{\phi(t)},\Lambda)\cong H^*(L_{p},\Lambda)$ because of the difference of the corresponding symplectic area (also known as the flux)  
         \begin{align*}
            \int_{{f_{1,t}}_*\gamma}\omega-\int_{\gamma}\omega =\left\langle \sum_{k=1}^n \big(x_k(\phi(t))-x_k(p)\big) e_k, \  \partial \gamma \right\rangle,
         \end{align*} where $e_i\in H^1(L_p,\mathbb{Z})$ is an integral basis. 
         
         From the $1$-parameter family of almost complex structures $J_t$, one can construct a pseudo-isotopy of unital $A_{\infty}$ structures on $H^*(L_p,\Lambda)$, connecting the $A_{\infty}$ structures on $H^*(L_p,\Lambda)$ from $u_{\pm}$ (for instance see \cite{F2}). This induces a pseudo-isotopy of $A_{\infty}$ structures from $H^*(L_p,\Lambda)$ to itself. In particular, this induces an isomorphism on the corresponding Maurer-Cartan spaces, which are isomorphic to $H^1(L_p,\Lambda_+)$ because of the dimension,
 \begin{align}\label{wc}
 \Phi:H^1(L_p,\Lambda_+)\rightarrow H^1(L_p,\Lambda_+).
 \end{align}                   A priori this is not the identity if $L_{\phi(t)}$ bounds holomorphic discs of Maslov index zero for some $t\in [0,1]$ \cite{FOOO}. 
The following lemma states that $\Phi$ only depends on the homotopy class of the path $\phi$.          \begin{lem}\cite{T4}*{Theorem 2.7} \label{compatibility}
         	$\Phi\equiv 1$ (mod $\Lambda_+$) and $\Phi$ only depends on the homotopy class of $\phi$ assuming no appearance of negative Maslov index discs in the homotopy. In particular, if $\phi$ is a contractible loop, then the corresponding $\Phi=1$ (before modulo $\Lambda_+$).\footnote{Lemma 4.3.15 \cite{FOOO} showed that the homotopic $A_{\infty}$-homomorphisms induced the same map on the Maurer-Cartan spaces.}
         \end{lem}

The explicit form of $\Phi$ can be computed in the case of the hyperK\"ahler surfaces assumption and one can see that $\Phi $ acts like wall crossing in the Gross-Siebert program from the theorem below.
\begin{thm} \label{thm:wallcrossing}
\cite{L8}*{Theorem 6.15} Let $\gamma$ be a primitive relative class and
assume that $k\gamma$, with $k\in \mathbb{N}$ are the only relative classes such that $L_{\phi(t)}$ bound holomorphic discs. Suppose that $\mbox{Arg}Z_{\gamma}(u_-)<\mbox{Arg}Z_{\gamma}(u_+)$ (Check Remark \ref{rk:loop} for the discussion of the signs).
	Then the transformation $\Phi$ is given by 
	\begin{align} \label{eq:wallcrossing}
	\mathcal{K}_{\gamma}:z^{\partial\gamma'}\mapsto z^{\partial \gamma'}f_{\gamma}^{\langle \gamma',\gamma\rangle}, 
	\end{align} for some power series $f_{\gamma}\in 1+T^{\omega(\gamma)}z^{\partial\gamma}\mathbb{R}[[T^{\omega(\gamma)}z^{\partial\gamma}]]$. Here $\langle \gamma',\gamma\rangle$ denotes the intersection pairing of the corresponding boundary classes in the torus fibre.
\end{thm}

The coefficients of $\log{f_{\gamma}}$ have enumerative meanings: counting Maslov index zero discs bounded by the $1$-parameter family of Lagrangians \cite{L14} or counting rational curves with certain tangency conditions \cite{GPS}. This motivates the following definition. More precisely, we define the open Gromov-Witten invariants $\tilde{\Omega}(\gamma;u)$ below:
\begin{definition} \cite{L8}\label{thm:wallfun} With the notation as in Theorem \ref{thm:wallcrossing} and 
 $u\in B_0$ the intersection of the image of $\phi$ and $l_{\gamma}$, the open Gromov-Witten invariants $\tilde{\Omega}(\gamma;u)$ are defined via the formula 
  \begin{align*}
     \log{f_{\gamma}}=\sum_{d\in\mathbb{N}} d\tilde{\Omega}(d\gamma;u)(T^{\omega(\gamma)}z^{\partial \gamma})^d.
  \end{align*} For other choices of $(u,\gamma)$, we set $\tilde{\Omega}(\gamma;u)=0$. In other words, $\tilde{\Omega}(\gamma;u)\neq 0$ only if $\mathcal{M}_{\gamma}((X,J),L_u)\neq \emptyset$ from \eqref{two moduli}. 
\end{definition} 
Then BPS rays are defined to have the support equal to the loci with non-trivial open Gromov-Witten invariants of the same homology classes (up to parallel transport).
\begin{definition}
	Given $u\in B_0$ and a relative class $\gamma\in H_2(X,L_u)$, then the associated BPS ray is defined (locally) to be  
	  \begin{align*}
	      l_{\gamma}:=\{u'\in B_0 \ | \
	      \tilde{\Omega}(\gamma;u')\neq 0 \mbox{ and $Z_{\gamma}(u')\in \mathbb{R}_+$}\}.
	  \end{align*}
\end{definition}

Recall that the geometry $X=X_*$ is not compact and a priori the moduli spaces of $J$-holomorphic discs may not be compact. However, from the curvature decay and injectivity radius decay in \cite[Theorem 1.5]{H} and the qualitative version of Gromov compactness theorem (see for instance \cite[Proposition 5.3]{CJL} or \cite{Gro}), the moduli spaces of discs in $X$ with compact Lagrangian boundary conditions are compact. 
To compute the open Gromov-Witten invariants on $X$, we first recall the following fact:
given a rational elliptic surface $Y'$ and a fibre $D'$, there exists a $1$-parameter deformation $(Y_t,D_t)$ such that $D'_t\cong D'$ and $Y'_t$ are rational elliptic surfaces with only type $I_1$ singular fibres except $D'_t$. 
The following theorem explains how to compute the local open Gromov-Witten invariants  near a general singular fibre other than those of type $I_1$. We will denote by $X_t$ the hyperK\"ahler rotation of $Y'_t\backslash D'_t$ with relation similar to (\ref{HK}). Then let $X_t\rightarrow B_t$ be a $1$-parameter family of hyperK\"ahler surfaces with special Lagrangian fibration and $X_0=X$. We will identify $B_t\cong B_0=B$ topologically. 

\begin{thm} \label{399} \cite{L12}*{Theorem 4.3}
Given any $u\in B_0$, $\gamma\in H_2(X,L_u)$, there exists $t_0$ and a neighborhood $\mathcal{U}\subseteq B_0$ of $u$ such that 
	\begin{enumerate}
	    \item If $\tilde{\Omega}(\gamma;u)=0$, then $\tilde{\Omega}(\gamma;u')=0$ for $u'\in \mathcal{U}$. 
	    \item If $\tilde{\Omega}(\gamma;u)\neq 0$, then $l^t_{\gamma}\cap \mathcal{U}\neq \emptyset$ and  \begin{align*}
	\tilde{\Omega}_t(\gamma;u')=\tilde{\Omega}(\gamma;u),
	\end{align*} 
	for $u'\in l^{t}_{\gamma}\cap \mathcal{U}$ and $t$ with $|t|<t_0$.
	\end{enumerate} 
 Here $\tilde{\Omega}_t(\gamma;u)$ denotes the open Gromov-Witten invariant of $X_t$. 
\end{thm}
For instance, in the case where the singular configuration of $Y'$ is  $II^*II$, then the BPS rays of $X_t$ would look like the following picture with the notation defined in Section \ref{section: dp5}:
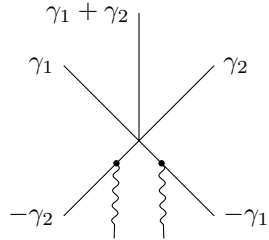
\begin{figure}[H]
\centering
\begin{tikzpicture}
\draw
(0,0) -- (1,1) node[right] {$\gamma_2$}
(0,0) -- (0,1.7) node[left] {$\gamma_1 + \gamma_2$}
(0,0) -- (-1,1) node[left] {$ \gamma_1$}
(0,0) -- (-1,-1) node[left] {$-\gamma_2$}
(0,0) -- (1,-1) node[right] {$-\gamma_1$};
\filldraw(-.3, -.3) circle(1pt);
\draw [snake=snake,
segment amplitude=.4mm,
segment length=2mm,
line after snake=1mm] (-.3, -.3)--(-.33, -1.3);
\filldraw(.3, -.3) circle(1pt);
\draw [snake=snake,
segment amplitude=.4mm,
segment length=2mm,
line after snake=1mm] (.3, -.3)--(.33, -1.3);
\end{tikzpicture}
\caption{BPS rays on $B_t$ for the case discussed in Section \ref{section: dp5} } \label{fig:sgenA2}
\end{figure}

\subsection{Location of BPS Rays}\label{section: BPS rays}
In this section, we will restrict to the case where $Y'$ has exactly two singular fibres at $0,\infty$ and the monodromy of each singular fibre is of finite order. The examples listed in Section \ref{section: setup}  are exactly those possible $Y'$ except the one with singular configuration $I_0^*I_0^*$. We will show that the BPS rays divide the base into chambers which are in one-to-one correspondence with the torus charts of the family Floer mirror later.  
In particular, the following observation simplifies the explicit computation of family Floer mirror. 

\begin{lem} \label{514} Let $\gamma$ be one of the relative classes in Theorem \ref{913}.
   Then $l_{\gamma_i}$ does not intersect each other.
 Specifically, $B_0$ is divided into chambers by $l_{\gamma_i}$. 
\end{lem}
\begin{proof}
    Let $v\in TB$ and 
	recall that one has
	$vZ_{\gamma}=\int_{\partial \gamma}\iota_{\tilde{v}}\Omega$, where $\tilde{v}$ is a lifting of $v$, by direct computation. Together with $\Omega$ being holomorphic symplectic, $Z_{\gamma}$ has no critical point in $B_0$.
	Let $l_{\gamma}$ be a BPS ray, then by definition the holomorphic function $Z_{\gamma}$ has phase $0$ along $l_{\gamma}$. Now take $v$ to be the tangent of $l_{\gamma}$ at $u\in l_{\gamma}$ pointing away from the origin. 
	Then we have 
  	  $vZ_{\gamma}(u)\neq 0$, since otherwise, $u$ is a critical point of $Z_{\gamma}$ and contradicts to the fact that $\Omega$ is holomorphic symplectic. 
In other words, the function $|Z_{\gamma}|$ is strictly increasing along $l_{\gamma}$.

Next we claim that
  $l_{\gamma}$ can not be contained in a compact set. Otherwise, there exists a sequence of points $u_i\in l_{\gamma}$ converging to some point $u_{\infty}\in B_0$. 
Since the order of the monodromy is finite, there are only finitely possibly relative classes among $\gamma_{u_i}$ with respect to the trivialization of the local system $H_2(X,L_u)$ in a small neighborhood of $u_{\infty}$. Here $\gamma_{u_i}$ is the parallel transport of $\gamma$ to $u_i$. After passing to a subsequence, one has $\lim_{i\rightarrow \infty}Z_{\gamma}(u_i)=Z_{\gamma}(u_{\infty})$. If $u_{\infty}\in B_0$, then $l_{\gamma}$ can be extended over $u_{\infty}$ and leads to a contradiction. Therefore, $l_{\gamma}$ connects $0$ and $\infty$. 
Recall the asymptotic geometry of $X'$ from \cite[p.208]{BPV}: consider 
the model space $X_{mod}$ defined by 
    \begin{align*}
        X_{mod}:=\{(u,v)\in \mathbb{C}\times \mathbb{C}||u|>R, \mbox{Arg}(u)\in [0,2\pi \beta] \}/\sim 
    \end{align*} where the equivalence relation is 
    \begin{align*}
       (u,v)&\sim (e^{2\pi i\beta}u, e^{-2\pi i\beta}v), \mbox{ for } u\in \mathbb{R}_+ \\ (u,v)&\sim (u,v+m+n\tau), \mbox{ for }m,n \in \mathbb{Z}.
    \end{align*} Here $\beta,\tau$ are suitable constants depending on the type of the fibre at infinity \cite[p.208, Table 5]{BPV}.
Then there exists a compact set $K\subseteq X'$, a diffeomorphism $\Psi: X'\setminus K\rightarrow X_{mod}$ and a holomorphic volume form $\Omega'$ on $X'$ such that
  \begin{align*}
      (\Psi^{-1})^*\Omega'=du\wedge dv+O(|u|^{-1}).
  \end{align*} Here $O(|u|^{-1})$ means differential forms whose components in the $(u,v)$-coordinates are of order $O(|u|^{-1})$. 
Then straight-forward calculations shows that $|Z_{\gamma}|\nearrow \infty$ along $l_{\gamma}$.

Notice that the above argument holds for $l_{\gamma}^{\theta}$, where $l_{\gamma}^{\theta}$ is the loci where $Z_{\gamma}$ has phase $\theta\in S^1$. This implies that
$|Z_{\gamma}(u)|\rightarrow \infty$ as $u\rightarrow \infty$. Recall that
  $Z_{\gamma}(u)$ is a multi-valued holomorphic function on $B_0\cong \mathbb{C}^*$. Since $\pi_1(B_0)\cong \mathbb{Z}$ and the monodromy is of order $k$, we have $Z_{\gamma}(u^k)$ is a well-defined 
 holomorphic function $\mathbb{C}^*\rightarrow \mathbb{C}^*$. By straight-forward calculation one has $\lim_{u\rightarrow 0}Z_{\gamma}(u^k)=0$ and thus $u=0$ is a removable singularity. The previous discussion implies that $\infty$ is a pole and  the holomorphic function $Z_{\gamma}(u^k)$ extends to $\mathbb{P}^1\rightarrow \mathbb{P}^1$ and fixing $0,\infty$. Thus, we reach that
  \begin{align} \label{central charge}
      Z_{\gamma}(u^k)=cu,
  \end{align}
  for some constant $c\in \mathbb{C}^*$ and the lemma follows. 
 \end{proof}

\begin{rmk} \label{analytic difficulty}
  Let $Y$ be the del Pezzo surface of degree five and $D$ be an anti-canonical divisor consists of a wheel of five rational curves. Set $X=Y\setminus D$. It is known that $X$ is the moduli space of flat connections on punctured sphere \cite{Bo}. There exists a hyperKähler metric on it such that suitable hyperK\"ahler rotation becomes some meromorphic Hitchin moduli space, which is $X'$, the complement of the $II^*$ fibre of the rational elliptic surface $Y'$ with singular configuration $II^*II$. It is not clear if the holomorphic volume form $\Omega'$ on $X'$ extends as a meromorphic form with a simple pole along the $II^*$ fibre. However, the Hitchin metric is exponentially asymptotic to some semi-flat metric at infinity \cite{CL}, the proof of Lemma \ref{514} also applies to this case. 
   
\end{rmk}

\subsection{Construction of the Family Floer Mirror} \label{sec:construction}

We will briefly recall the construction of family Floer mirror constructed by Tu \cite{T4} in this section. 
We will refer the details of the analytic geometry to \cite{C}. 

\begin{definition} Let $U\subseteq B_0$ be an open set and $\psi:U\rightarrow \mathbb{R}^n$ be the affine coordinate such that $\psi(u)=0$ for some $u\in U$. Then we say $U$ is a rational domain if $U\cong\psi(U)=P\subseteq \mathbb{R}^n$ is a  rational convex polytope. The Tate algebra $T_U$ associated to a rational domain $U$ consists of the power series of the form 
	   \begin{align*}
	      \sum_{\mathbf{k}\in \mathbb{Z}^n} a_{\mathbf{k}}z_1^{k_1}\cdots z_n^{k_n},
	   \end{align*} where $\mathbf{k}=(k_1,\cdots,k_n)$ with the following conditions:
	   \begin{enumerate}
	   	\item $a_{\mathbf{k}}\in \Lambda$ the Novikov field and 
	   	\item (convergence in $T$-adic topology)
	   	  \begin{align}\label{convergence}
	   	 \lim_{\mathbf{k}\rightarrow \infty}\mbox{val}(a_{\mathbf{k}})+\langle\mathbf{k}, \mathbf{x}\rangle\rightarrow \infty ,
	   	 \end{align} as $\mathbf{k}\rightarrow \infty$,
	   	 for each $\mathbf{x}=(x_1,\cdots, x_n)\in U$. 
	   \end{enumerate}   
 For such rational domain $U$, we denote by $\mathcal{U}$ the maximum spectrum of the associated Tate algebra $T_U$.
\end{definition}
\begin{remark} \label{chart}
   Recall that the Novikov field $\Lambda$ is algebraically closed. From Proposition 3.1.8 (c) \cite{EKL}, $\mathcal{U}$ is identified with $\mbox{Val}^{-1}(\psi(U))$, where $\mbox{Val}:(\Lambda^*)^n\rightarrow \mathbb{R}^n$ is component-wise the valuation on $\Lambda^*$. If $f\in T_U$, then $f(x)$ converges for all points $x\in \mathcal{U}$ and thus we may view $f$ as a function defined on $\mbox{Val}^{-1}(\psi(U))$. 
\end{remark}

Take an open cover of rational domains $\{U_i\}_{i\in I}$ of $B_0$ with affine coordinates $\psi_i:U_i\rightarrow \mathbb{R}^n$ such that $\psi_i(u_i)=0\in \mathbb{R}^n$ for some $u_i\in U_i$. For each pair $i,j$ with $U_i\cap U_j\neq \emptyset$, there is a natural gluing data 
    \begin{align*}
       \Psi_{ij}: \mathcal{U}_i\rightarrow \mathcal{U}_j, 
    \end{align*} which now we will explain below:  Choose $p\in U_i\cap U_j$ and $f_{u_i,p}$ fibrewise preserving diffeomorphism sending $L_{u_i}$ to $L_p$ and is identity outside $U_i$. Let $(x^i_1,\cdots, x^i_n)$ (and $(x^j_1,\cdots, x^j_n)$) be the local symplectic affine coordinates on $U_i$ (and $U_j$ respectively) with respect to the same set of basis up to parallel transport within $U_i,U_j$. 
    The corresponding functions in $T_{U_i}$ are denoted by $(z^i_1,\cdots, z^i_n)$, where $\mbox{val}(z^i_k)=x^i_k$.
   
The difference of symplectic affine coordinates is
  \begin{align*}    \int_{{f_{u_i,p}}_*\gamma}\omega-\int_{\gamma}\omega =\left\langle \sum_{k=1}^n \big( x_k(p)-x_k(u_i)\big) e^i_k,\partial \gamma \right\rangle. 
      \end{align*}
Denote $T_{U_{i,p}}$ the Tate algebra satisfying the convergence in $T$-adic topology \eqref{convergence} on the rational domain $\psi_i(U_i)-\psi_i(p)$, the translation of $\psi(U_i)$, and we denote by $\mathcal{U}_{i,p}$ the spectrum of $T_{U_{i,p}}$.  Then
    there is the transition map 
    \begin{align}
      S_{u_i,p}: \mathcal{U}_i&\rightarrow \mathcal{U}_{i,p}\notag \\
       z^i_k&\mapsto T^{x_k(u_i)-x_k(p)}z^i_k.
    \end{align} Then define the gluing data $\Psi_{ij}$ be the composition
 \begin{align} \label{2}
           \Psi_{ij}: \mathcal{U}_{ij}\xrightarrow{S_{u_i,p}} \mathcal{U}_{ij,p}\xrightarrow{\Phi_{ij}} \mathcal{U}_{ji,p} \xrightarrow{S^{-1}_{u_j,p}}  \mathcal{U}_{ji},
      \end{align} where $\Phi_{ij}$ is defined in \eqref{wc}.
  The gluing data $\Psi_{ij}$ satisfies (Section 4.9 \cite{T4})
  \begin{enumerate}
  	\item Independent of the choice of reference point $p\in U_i\cap U_j$. 
  	\item $\Psi_{ij}=\Psi_{ji}$ and $\Psi_{ij}\Psi_{jk}=\Psi_{ik}$. 
  	\item For $p\in U_i \cap U_j \cap U_k$, we have $\Psi_{ij}(\mathcal{U}_{ij}\cap \mathcal{U}_{ik})\subseteq \mathcal{U}_{ji}\cap \mathcal{U}_{jk}$. 
  \end{enumerate}
Then the family Floer mirror $\check{X}$ is defined to be the gluing of the affinoid domains 
    \begin{align} \label{SG mirror}
        \check{X}:=\bigcup_{i\in I}\mathcal{U}_i/\sim,
    \end{align}
    where $\sim$ is defined by (\ref{2}). The natural projection map $T_U\rightarrow U$ from the valuation glue together and gives the family Floer mirror a projection map 
    \begin{align*}
       \mathfrak{Trop}: \check{X}\rightarrow B_0. 
    \end{align*}
The following example is straight forward from the construction.
\begin{ex} \label{no discs}
Recall that the rigid analytic torus $(\mathbb{G}_m^{an})^2$ admits a valuation map $\mathfrak{Trop}:(\mathbb{G}_m^{an})^2\rightarrow \mathbb{R}^2$. Let $\pi:X\rightarrow U$ be a Lagrangian fibration such that for any path $\phi$ connecting $u_i,u_j\in U$, the corresponding $F_{\phi}=id$. Assume that the symplectic affine coordinates give an embedding $U\rightarrow \mathbb{R}^2$ and we will simply identify $U$ with its image. Then the gluing \begin{align*}
  \Psi_{ij}:\mathcal{U}_{ij}&\rightarrow \mathcal{U}_{ji}\\
   z_k^i &\mapsto  T^{x_k(u_i)-x_k(u_j)} z_k^{j},
\end{align*}
 is simply translation from equation \eqref{2}. 
Thus the corresponding family mirror $\check{X}$ is simply $\mathfrak{Trop}^{-1}(U)\rightarrow U$. In particular, when $U\cong \mathbb{R}^2$, then the family Floer mirror is simply the rigid analytic torus $ (\mathbb{G}_m^{an})^2$ . It worth noticing that if $U\subseteq \mathbb{R}^2$ is a proper subset, then $\mathfrak{Trop}^{-1}(U)$ is not a dense subset of $(\mathbb{G}_m^{an})^2$. 
\end{ex}
\begin{remark}
   It is pointed out by H. Yuan \cite[Remark 4.2]{Y3} that for the family Floer mirror above to equip with the correct rigid analytic topology, one needs to use a generalized version of Groman-Solomon's reversed isoperimetric inequality \cite[Theorem 1.1]{GS4}\cite[Appendix B]{Y} to ensure the convergence of the analytic gluing.

\end{remark}


\section{Family Floer Mirror of $X_{II}$} \label{section: dp5}
In this section, we will have a detailed computation of the family Floer mirror of $X=X_{II}$ (see Section \ref{section: setup}. We sketch the proof below:
We will first identify the locus of special Lagrangian fibres bounding holomorphic discs to be simply five rays $l_{\gamma_i}$ connecting $0,\infty$.
Then we compute their corresponding wall-crossing transformations which are analytification of some birational maps. Thus, the family Floer mirror $\check{X}$ can be glued from five charts. Then we will prove that the embedding of each of the five charts into $\check{X}$ can be extended to an embedding of the analytic torus $\mathbb{G}_m^{an}$ into $\check{X}$. In other words, $\check{X}$ is the gluing of five analytic tori. 
On the other hand, consider the del Pezzo surface $Y$ of degree $5$ and let $D$ be the cycle of five rational curves. Let $B_{\GHK}$ be the affine manifold with the singularity constructed in Section \ref{RES}. We identify the complex affine structure on $B$ with the one on $B_{\GHK}$, the rays and the corresponding wall-crossing transformations. Then from \cite{GHK}*{Example 3.7}, we know that $\check{X}$ is the analytification of the del Pezzo surface of degree five.  
Furthermore, we would choose the branch cuts on $B$ in a different way. This would induce another realization of $\check{X}$ as the gluing of five tori but with different gluing morphisms, which we will later identify as the $\cX$-cluster variety of type $A_2$. 

As discussed in Section \ref{section: setup}, the base $B$ with the complex structure from the elliptic fibration $Y'_{II}\rightarrow \mathbb{P}^1$ is biholomorphic to $\mathbb{C}$ as a subset of $\mathbb{P}^1$. We may choose a holomorphic coordinate $u$ on $B$ such that the fibres over $u=0,\infty$ are the type $II,II^*$ singular fibres. Such coordinate is unique up to scaling $\mathbb{C}^*$. We will determine such scaling later. 
First we apply Theorem \ref{399} to the $1$-parameter family of hyperK\"ahler rotation of the rational elliptic surfaces described in Section \ref{section: setup}, one get 
\begin{thm} 
\cite{L12}*{Theorem 4.11} \label{local discs} Choose a branch cut from the singularity to infinity and a basis $\{\gamma_1',\gamma_2'\}$ of  $H_2(X,L_u)\cong \mathbb{Z}^2$ such that $\langle \gone,\gtwo\rangle=1$ and the counter-clockwise monodromy $M$ around the singularity is 
\begin{align} \label{monodromy II}
    \gone &\mapsto -\gtwo \notag\\
    \gtwo &\mapsto \gone+\gtwo.
\end{align}
 Then for $|u|\ll 1$, one has
  \begin{align*}
      \tilde{\Omega}(\gamma;u)=\begin{cases}
        \frac{(-1)^{d-1}}{d^2}, & \mbox{if } \gamma=\pm d\gone,\pm d\gtwo,\pm d(\gone+\gtwo), \mbox{ for }d\in \mathbb{N}  \\
        0, &\mbox{otherwise.}
      \end{cases}
  \end{align*}
\end{thm}
\begin{remark} \label{multiple cover}
    See Remark 2.10 \cite{L12} for the relation between $\tilde{\Omega}(\gamma;u)$ and ${\Omega}(\gamma;u)$.
\end{remark}
\begin{remark} \label{transitive}
    The monodromy \eqref{monodromy II} acts transitively on the set $\{\pm \gone,\pm \gtwo, \pm (\gone+\gtwo)\}$. 
\end{remark}
\begin{remark} \label{commutative}
    If $A\in GL(2,\mathbb{Z})$ such that $A\begin{pmatrix}0 & -1 \\ 1 & 1 \end{pmatrix}=\begin{pmatrix}0 & -1 \\ 1 & 1 \end{pmatrix}A$, then $A=\begin{pmatrix}0 & -1 \\ 1 & 1 \end{pmatrix}^k$ for some $k\in \mathbb{N}$. Therefore, if $\gamma_1,\gamma_2\in H_2(X,L_u)$ such that the monodromy is the form of \eqref{monodromy II}, then $\gamma_1=M^k\gone,\gamma_2=M^k\gtwo$ for some $k\in \mathbb{N}$. 
\end{remark}
Furthermore, we will next show that these are the only possible relative classes with non-trivial open Gromov-Witten invariants even {\bf globally} in $X_{II}$. 
\begin{cor} \label{split attractor flow}
  If $\tilde{\Omega}(\gamma;u)\neq 0$, then $\gamma$ is one of the relative classes in Theorem \ref{local discs}, $u\in l_{\gamma}$ and $\tilde{\Omega}(\gamma;u)=\frac{(-1)^{d-1}}{d^2}$, where $d$ is the divisibility of $\gamma$. 
\end{cor}
\begin{proof}
   This is a direct consequence of the split attractor flow mechanism of the open Gromov-Witten invariants $\tilde{\Omega}(\gamma;u)$ (see  \cite{L8}*{Theorem 6.32}). We will sketch the proof here for being self-contained. Let $l_{\gamma}$ be a ray emanating from $u$ such $\omega_{\gamma}$ is decreasing along $l_{\gamma}$. From Gromov compactness theorem, the loci where $\tilde{\Omega}(\gamma)$ jumps on $l_{\gamma}$ are discrete. 
   Assume that $\tilde{\Omega}(\gamma)$ is invariant along $l_{\gamma}$, then the holomorphic disc representing $\gamma$ can only fall into a tubular neighborhood of the singular fibre over $0$ by \cite{CJL}*{Proposition 5.3} when the symplectic area decrease to small enough. We will show below that $\gamma$ is one of the relative classes in Theorem \ref{local discs}.  
   Otherwise, assume that $u_1$ is the first point where $\tilde{\Omega}(\gamma)$ jumps. Apply Lemma \ref{compatibility} to a small loop around $u_1$, there exists $\gamma_{\alpha}$, $\alpha\in A$ such that $\tilde{\Omega}(\gamma_{\alpha};u_1)\neq 0$ and $\gamma=\sum_{\alpha\in A} \gamma_{\alpha}$ with $|A|\geq 2$. In particular, $\omega(\gamma_{\alpha})<\omega(\gamma)$. One may replace $(\gamma,u)$ by $(\gamma_{\alpha},u_1)$ and run the procedure. Again by Gromov compactness theorem, after finitely many splittings, all the relative classes are among the one listed in Theorem \ref{local discs}. To sum up, there exists a rooted tree $T$ and a continuous map $f$ such that the root maps to $u$, each edge is mapped to an affine line segment and all the $1$-valent vertices are mapped to $0$.
   Since the affine lines corresponding to the relative classes in Theorem \ref{local discs} do not intersect by Lemma \ref{514}, the corollary follows. 
\end{proof}
Although it looks like there are six classes of relative classes support non-trivial open Gromov-Witten invariants, next we explain that there are actually only five BPS rays emanating from $u=0$ due to the monodromy.

Thanks to Remark \ref{transitive}, we will choose the scaling of the coordinate $u$ such that the branch cut is $\mbox{Arg}(u)=0$ and $l_{\gamma_1}=\{u\in \mathbb{R}_+\}$, where $\gamma_1=-\gone$. Let $\gamma_{i+1}=M\gamma_i$ on the complement of branch cut, for $i\in \mathbb{Z}$. Straight-forward calculation shows that $\gamma_{i+6}=\gamma_i$ and $\gamma_{i-1}-\gamma_i+\gamma_{i+1}=0$. Denote the symplectic and complex affine coordinate (with respect to $\gamma_k,\gamma_{k+1}$) discussed in Section \ref{affine} by 
\begin{align} \label{ac}
   x_k&=\int_{\gamma_{k}}\omega,  \qquad y_k=\int_{\gamma_{k+1}}\omega \\
     \check{x}_k&=\int_{\gamma_{k}}\mbox{Im}\Omega , \qquad \check{y}_k=\int_{\gamma_{k+1}}\mbox{Im}\Omega.
 \end{align}
We will also denote by 
\begin{align} \label{cac}
     x&=\int_{-\gone}\omega, \qquad y=\int_{\gtwo}\omega, \notag \\
     \check{x}&=\int_{\gtwo}\mbox{Im}\Omega, \qquad \check{y}=\int_{-\gone}\mbox{Im}\Omega,
\end{align} which give another set of symplectic/complex affine coordinates.

Recall that  $x_k(u)-i\check{x}_k(u)=Z_{\gamma_k}(u)$ is a holomorphic function with respect to the above complex structure on $B$ defined on the complement of the branch cut and can be analytic continued to a multi-valued holomorphic function on $\mathbb{C}^*$. In particular, if $\gamma$ is a relative class in Theorem \ref{local discs}, then $x_k(u)>0$ and $\check{x}_k(u)=0$ along a BPS ray $l_{\gamma_k}$. From \eqref{central charge}, one have 
   \begin{align} \label{central charge explicit}
     x_k-i \check{x}_k=c_ku^{\frac{a  }{6}},  \quad k\in \mathbb{Z},
   \end{align} for some constant $a\in \mathbb{N}, c_k\in \mathbb{C}^*$. With more analysis, we have the following lemma
  \begin{lem} \label{central charge II} With the above choice of coordinate $u$ on $B_0\cong \mathbb{C}^*$, we have 
    \begin{align}
      x_k-i\check{x}_k=e^{2\pi i(k-1)\frac{5}{6}}u^{\frac{5}{6}}.
    \end{align} In particular, the angle between $l_{\gamma_k}$ and $l_{\gamma_{k+1}}$ is $\frac{2\pi}{5}$ with respect to the conformal structure after hyperK\"ahler rotation.\footnote{Notice that there is no well-defined notion of angle with only an affine structure on $B_0$ and thus one does not see this aspect on the affine manifold used in Gross-Hacking-Keel.}
  \end{lem} 
  \begin{proof}
     From the normalization, we have  $x_1-i\check{x}_1=u^{\frac{a}{6}}$. Recall that $Z_{\gamma_k}:=x_k-i\check{x}_k$. From the monodromy $M{\gamma_k}=\gamma_{k+1}$, we have 
 \begin{align*}
     Z_{\gamma_{k+1}}(u)=Z_{\gamma_{k}}(ue^{2\pi i})=e^{2\pi i \frac{a}{6}}Z_{\gamma_k}(u).
 \end{align*} Here recall that $Z_{\gamma_i}(u)$ is a priori only defined on the complement of the branch cut and we use $Z_{\gamma}(ue^{2\pi i})$ to denote the value of analytic continuation across the branch cut counter-clockwise once at $u$. 
 
 Now it suffices to show that $a=5$ or show that $Z_{\gamma_i}(u)=O(|u^{\frac{5}{6}}|)$. This can be seen by direct computation. Indeed, for $u$ close enough to the origin $0\in B$, the representatives of $\gamma_i$ can be chosen to be in a neighborhood of the singular point of the type $II$ singular fibre. In such neighborhood, $X'_{II}$ is defined by $y^2=x^3+u$ from \eqref{15}. One can write $\Omega'=\frac{2f(u)}{u}dy\wedge dx=f(u)du\wedge \frac{dx}{y}$ for some holomorphic function $f(u)$ with $f(0)\neq 0$, since $\Omega'$ is a non-where vanishing holomorphic $2$-form on $X'_{II}$.\footnote{Notice that $u=y^2-x^3$ is a well-defined function on the chart.} Recall that the fibre over $u$ is topologically the compactification of $y^2=x^3+u$, a double cover of the $x$-plane ramified at three points $\zeta^i(-u)^{\frac{1}{3}},i=0,1,2$ where $\zeta=\exp{(2\pi i/3)}$. A path connecting $\zeta^i(-u)^{\frac{1}{3}},\zeta^j (-u)^{\frac{1}{3}}$ in the $x$-plane lifts to an $S^1$ in the fibre.
 Consider the $2$-chain $\gamma_{i,j}, i\neq j$, which is an $S^1$-fibration over a line segment from $u=0$ to $u=u_0$ such that the $S^1$-fibre in $L_u$ is the double cover of path connecting  $\zeta^i(-u)^{\frac{1}{3}},\zeta^j (-u)^{\frac{1}{3}}$ in $x$-plane. Each of $\gamma_{i,j}$ can be represented by the $2$-chain parameterized by $u=tu_0$, and the double cover of $x=s\zeta^i (-u)^{\frac{1}{3}}+(1-s)\zeta^j (-u)^{\frac{1}{3}}$, with $t\in [0,1],s\in [0,1]$.
 Since $\partial \gamma_i$ are vanishing cycles and generate $H_1(L_u)\cong H_2(X,L_u)$, $\gamma_i$ can be represented by some linear combination $a\gamma_{0,1}+b\gamma_{1,2}$ with $a,b\in \mathbb{Z}$ and $a^2+b^2\neq 0$.\footnote{The isomorphism can be easily seen from the Mayer-Vietoris sequence.}
 
 Then by direct calculation, one has 
 \begin{align}\label{asymptotic}
     Z_{\gamma_{i,j}}(u_0)=\int_{\gamma_{i,j}}\Omega'&=\int_{u=0}^{u=u_0}\int_{x=\zeta^i (-u)^{\frac{1}{3}}}^{x=\zeta^j (-u)^{\frac{1}{3}}} f(u)du\wedge\frac{dx}{y} \\
     &=\int_{u=0}^{u=u_0}\bigg(\int_{x=\zeta^i (-u)^{\frac{1}{3}}}^{x=\zeta^j (-u)^{\frac{1}{3}}} \frac{dx}{y}\bigg) du   +O(|u_0|) \notag\\
     &=\int_{u=0}^{u=u_0}\bigg(\int_{x=\zeta^i (-u)^{\frac{1}{3}}}^{x=\zeta^j (-u)^{\frac{1}{3}}} \frac{dx}{(x^3+u)^{\frac{1}{2}}}\bigg) du   +O(|u_0|)  \notag \\    &=\int_{u=0}^{u=u_0}\bigg(\int_{s=0}^{s=1}\frac{x'(s)ds}{\big((x(s)-\zeta^i(-u)^{\frac{1}{3}})(x(s)-\zeta^j(-u)^{\frac{1}{3}})(x(s)-\zeta^k(-u)^{\frac{1}{3}})\big)^{\frac{1}{2}}}  \bigg)du+O(|u_0|).
 \end{align} Here we have $k\in \{0,1,2\}\setminus \{i,j\}$ and we use the change of variable $x(s)=s\zeta^i (-u)^{\frac{1}{3}}+(1-s)\zeta^j (-u)^{\frac{1}{3}})$ in the forth line. Using $x'(s)=O(u^{\frac{1}{3}})$ and factoring out $u^{\frac{1}{2}}$ in the denominator of the last line of \eqref{asymptotic}, we have the part in the parenthesis is asymptotic to $\frac{u^{\frac{1}{3}}}{u^{\frac{1}{2}}}\int_0^1\frac{ds}{\big((s(1-s)\big)^{\frac{1}{2}}}=O(u^{-\frac{1}{6}})$. From the fact that $\int_0^{u_0}u^{-\frac{1}{6}}du=O(u_0^{\frac{5}{6}})$, we arrive at 
  \begin{align*}
     Z_{\gamma_{i,j}}(u)=C_{i,j}u^{\frac{5}{6}}+O(|u|),
  \end{align*}
 where $C_{i,j}$ is some constant independent of $u_0$  and $C_{0,1},C_{1,2}$ is linear independent over $\mathbb{Z}$. Thus, we have similar estimate 
  \begin{align*}
      Z_{\gamma_i}(u)=O(u^{\frac{5}{6}}).
  \end{align*}
 In particular, it implies that $Z_{\gamma_1}(u)=u^{\frac{5}{6}}$ from the choice of the normalization of $u$. Then
   \begin{align*}
       Z_{\gamma_2}(u)=Z_{M\gamma_1}(u)=(e^{2\pi i}u)^{\frac{5}{6}}=e^{2\pi i \frac{5}{6}}u^{\frac{5}{6}}.
   \end{align*} Thus, $Z_{\gamma_2}(u)\in \mathbb{R}_{>0}$ if and only if $u\in \mathbb{R}_{>0}e^{\frac{2\pi}{5} i}$. On the other hand, when $u\in l_{\gamma_{k}}$, one has  $Z_{\gamma_{k}}(u)=\int_{\gamma_k}\omega-i\int_{\gamma_k}\mbox{Im}\Omega\in \mathbb{R}_{>0}$ by Remark \ref{affine line} and the fact that symplectic area of a holomorphic disc is positive. Thus, we have $l_{\gamma_2}=\{u\in \mathbb{R}_{>0}e^{\frac{2\pi i}{5}}\}$, or the angle between $l_{\gamma_1},l_{\gamma_2}$ is $\frac{2\pi}{5}$. The general statement of the second part of the lemma can be then proved inductively. Finally we observe that $Z_{\gamma_6}\notin \mathbb{R}_+$ until $u$ varies across the branch cut counter-clockwisely. If one analytic continues $Z_{\gamma_6}$ across the branch cut counter-clockwisely then $Z_{\gamma_7}=Z_{\gamma_1}$ because $\gamma_7=M\gamma_6$ and $M^6=id$. Therefore, the corresponding BPS ray again has the same locus as $l_{\gamma_1}$.
  
  \end{proof}
\begin{remark} \label{five BPS rays}
    There are only five BPS rays in total instead of six. In other words, there are only five families of discs with non-trivial open Gromov-Witten invariant and contributing to the construction of the family Floer mirror.
\end{remark}
We conclude the above discussion now:
\begin{thm}\label{913}
With the notation above,
\[\gamma_1=-\gamma_1',\ \gamma_2=\gamma_2',\ \gamma_3=\gamma_1'+\gamma_2',\ \gamma_4=\gamma'_1,\ \gamma_5=-\gamma_2'.\] Then
	\begin{enumerate}
	    \item $f_{\gamma}(u)\neq 1$ if and only if $u\in l_{\gamma_i}$ and $\gamma=\gamma_i$ for some $i=1,\cdots, 5$ .
	    \item In such cases, $f_{\gamma_i}=1+T^{\omega(\gamma_i)}z^{\partial\gamma_i}$.
	    \item The branch cut can be chosen to be between $l_{\gamma_1}$ and $l_{\gamma_5}$.
	\end{enumerate}
\end{thm}
\begin{proof}
    Remark \ref{five BPS rays} explains that there are actually five BPS rays $l_{\gamma_i},i=1,\cdots, 5$. The second statement comes from Definition \ref{thm:wallfun} and Theorem \ref{local discs}. The third statement is how we define the branch cut below Corollary \ref{split attractor flow}.
\end{proof}
 The affine structure is illustrated in Figure \ref{fig:1branchcut} below. In Figure \ref{fig:1branchcut}, the curvy ray, between $l_{\gamma_5}$ and $l_{\gamma_1}$, represents the branch cut. The `monodromy' of the affine structure, can be seen as gluing the curvy ray with $l_{\gamma_1}$. 
 The shaded region indicates the gluing region. 
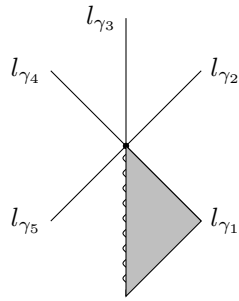
\begin{figure}[H]
\centering
\begin{tikzpicture}
\draw
(0,0) -- (1,1) node[right] {$l_{\gamma_2}$}
(0,0) -- (0,1.7) node[left] {$l_{\gamma_3}$}
(0,0) -- (-1,1) node[left] {$ l_{\gamma_4}$}
(0,0) -- (-1,-1) node[left] {$l_{\gamma_5}$}
(0,0) -- (1,-1) node[right] {$l_{\gamma_1}$};
\filldraw(0, 0) circle(1pt);
\draw [snake=snake,
segment amplitude=.4mm,
segment length=2mm] (0,0)--(0, -2);
\draw[fill=lightgray]  (0,0) -- (0,-2) -- (1,-1) -- cycle;
\end{tikzpicture}
\caption{BPS rays near the singularity.} \label{fig:1branchcut}
\end{figure}
Straight-forward calculation shows that
\begin{align}\label{boundary relation}
   \gamma_{i+2}=-\gamma_i+\gamma_{i+1},
\end{align} 
 which is the analogue of \eqref{affine relation}.

\subsection{Construction of Family Floer Mirror of $X_{II}$} \label{section: construction}


Let $U_k$ be the chamber bounded by $l_{\gamma_k}$ and $l_{\gamma_{k+1}}$ in $B_0$, $i=1,\cdots, 4$ and $U_5$ be the chamber bounded by $l_{\gamma_5}$ and $l_{\gamma_1}$. Thus there are only 5 chambers.
Recall that the dotted line represents a branch cut between $l_{\gamma_1}$ and $l_{\gamma_5}$. 
 With such branch cut and monodromy, we trivialize the local system $H_2(X,L_u)$ over the complement of the branch cut. Recall that we have  $M\gamma_i=\gamma_{i+1}$ by definition.


   Next, we compare the complex affine structure from the SYZ fibration with the affine structure from Gross-Hacking-Keel (see Section \ref{RES}). 
\begin{lem} \label{affine identification}
  The complex affine structure on $B_0$ is isomorphic to the affine manifold $B_{\GHK}$ with singularity constructed from del Pezzo surface of degree five relative to a cycle of five rational curves in \cite{GHK}. 
\end{lem}
\begin{proof}
First notice that one can compute the complex affine coordinates on $l_{\gamma_1},l_{\gamma_2}$,
\begin{align}
    l_{\gamma_1}&=\{\check{y}=0,\check{x}>0\} \notag \\
    l_{\gamma_2}&=\{\check{x}=0,\check{y}>0\}.
\end{align} 
Indeed, we have $\check{y}=0$ on $l_{\gamma_1}$ by Remark \ref{affine line}. From Lemma \ref{central charge II}, we have \begin{align*}
    \check{x}(u)=\int_{\gtwo}\mbox{Im}\Omega=\mbox{Re}Z_{\gamma}(u)=\mbox{Re}Z_{-M\gone}(u)=-\mbox{Re}(ue^{2\pi i})^{\frac{5}{6}}>0,
\end{align*} for $u\in l_{\gamma_1}$. One can compute the case of $u\in l_{\gamma_2}$ similarly. 
Therefore, with respect to the complex affine structure, the primitive tangent vectors of  $l_{\gamma_1},l_{\gamma_2}$ are given by $\frac{\partial}{\partial \check{x}}, \frac{\partial}{\partial \check{y}}$. To compare with the affine structure from Gross-Hacking Keel, we will identify them with $\mathbb{R}_{>0}(1,0),\mathbb{R}_{>0}(0,1)$. 
Then $(-1,1),(-1,0),(0,-1)$ are the tangents of $l_{\gamma_3},l_{\gamma_4},l_{\gamma_5}$ respectively by Lemma \ref{central charge II} and the relation $-Z_{\gamma_i}+Z_{\gamma_{i+1}}=Z_{\gamma_{i+2}}$ which is the analogue of \eqref{affine relation}. 
Therefore, the complex affine coordinates on the region in $\{u\in B_0|0< \mbox{Arg}u<\frac{8\pi}{5} \}$ is isomorphic to the one on the sector (without the vertex) from $(1,0)$ counter-clockwise to $(0,-1)$ viewed as an affine submanifold of $\mathbb{R}^2$. To understand the monodromy of the complex affine structure on $B_0$, one need to do the similar calculation across the branch cut. Consider the affine structure on the universal cover of $B_0$.\footnote{By abuse of notation, we still use the coordinate $u$ for the corresponding holomorphic coordinate and $\check{x},\check{y}$ for the pull back of the complex affine structure.} Then similar calculation shows that the complex affine coordinates on the region in $\{u\in B_0|0< \mbox{Arg}(u)< 2\pi  \}$ is isomorphic to the one on the sector (without the vertex) from $(1,0)$ counter-clockwise to $(-1,-1)$, denoted by $\mathcal{V}_1$, viewed as an affine submanifold of $\mathbb{R}^2$. 
If one change the location of the branch cut to $\mbox{Arg}(u)=-\frac{2\pi}{5}^-$, the complex affine structure on the region in $\{u\in B_0|-\frac{2\pi}{5} < \mbox{Arg}u< \frac{8\pi}{5} \}$ is isomorphic to the one on the sector (without the vertex) from $(-1,-1)$ counter-clockwise to $(0,-1)$, denoted by $\mathcal{V}_2$, viewed as an affine submanifold of $\mathbb{R}^2$.\footnote{Alternatively, one may extend the affine structure across the original branch cut clock-wisely and then $(-1,-1)$ is the primitive tangent of $l_{\gamma_0}$.} The affine coordinates on $\{u\in B_0| 0<\mbox{Arg}(u)<\frac{8\pi}{5}\}$ from pull-back from $\mathcal{V}_1,\mathcal{V}_2$ coincide, so the complex affine structure on $\{-\frac{2\pi}{5}<\mbox{Arg}(u)<2\pi\}$ (viewed as a subset of universal cover of $B_0$) is isomorphic to the natural affine structure on $\mathcal{V}_1\cup \mathcal{V}_2$ as a subset of $\mathbb{R}^2$ (but not with respect to the affine structure inherited from $\mathbb{R}^2$). 
Recall that we have $l_{\gamma_6}=l_{\gamma_1}$ and $l_{\gamma_5}=l_{\gamma_0}$ from Remark \ref{five BPS rays}. Therefore, $B_0$ as an affine manifold is simply the gluing of the sector bounded by $(0,-1),(-1,-1)$ in $\mathcal{V}_1$ and the sector bounded by $(-1,-1),(1,0)$ in $\mathcal{V}_2$. Denote by $M:\mathbb{R}^2\rightarrow \mathbb{R}^2$ the linear map sending $(0,-1)$ to $(-1,-1)$ and $(-1,-1)$ to $(1,0)$. Explicitly, we have
     $B_0=\mathcal{V}_1\cup \mathcal{V}_2/\sim $
  as affine manifolds, where $x\sim y$ if $x$ is contained in the sector bounded by $(0,-1),(-1,-1)$ and $y=Mx\in \mathcal{V}_2$. 
 This is exactly the description of $B_{GHK}$. Moreover, one sees that $\{U_i\}_{i=1,\cdots, 5}$ coincides with the decomposition $\Sigma$ in Section \ref{RES}. 

\end{proof} 
 \begin{rmk}
    Write $\check{x}'(u)=\check{x}(ue^{2\pi i}),\check{y}'(u)=\check{y}(ue^{2\pi i})$ as the continuation of $\check{x},\check{y}$ counter-clockwisely. From \eqref{monodromy II}\eqref{cac}, one has  
      \begin{align*}
          d\check{x}'&=d\check{x}-d\check{y}  \\
          d\check{y}'&=d\check{x} 
      \end{align*} or equivalently 
      \begin{align*}
         d\check{x}&=d\check{y}'\\
         d\check{y}&=-d\check{x}'+d\check{y}'.
      \end{align*} Dually, the monodromy on the complex affine coordinate is given by 
      \begin{align*}
          \frac{\partial}{\partial \check{x}'}&=\frac{\partial}{\partial \check{x}}+\frac{\partial}{\partial \check{y}}\\
        \frac{\partial}{\partial \check{y}'}&=-\frac{\partial}{\partial \check{x}},
      \end{align*} which is exactly the gluing at the end of Lemma \ref{affine identification}, sending $(-1,0)$ (and $(-1,-1)$) to $(-1,-1)$ (and $(-1,0)$ respectively). 
      
 \end{rmk}
 
\noindent
\begin{minipage}{\linewidth}
\captionsetup{type=figure}
\begin{center}
\begin{tikzpicture}

    \node [inner sep=0] (0) at (0,0) {$\bullet$};

    \draw[->] [name path=walls1] (0) -- ++(15:3) node [above right] {$l_{\gamma_i}= \{ \check{y}_i = 0 \}$}; 
    \draw[->] [name path=walls1m1] (0)--++(87:3) node [above] {$l_{\gamma_{i+1}}$}; 
    \draw[->] [name path=wall11] (0)--++(159:3) node [above] {$l_{\gamma_{i+2}}$};
    \draw[->] [name path=wall12] (0)--++(231:3) node [below left] {$l_{\gamma_{i+3}}$}; 

    \node at (45:3) {$U_i$};
    
    \draw[color=blue] (0)--++(123:3.5) node[above] {$y_i=0$};
    \draw[color=blue] (0)--++(195:3.5) node[above] {$x_i=0$};
    
     \draw[fill=green, opacity=0.2, draw=white] (0,0) --(10:3.4)--(95:3.4) ;
     
     \node[color=olive] at (45:2) {$U_i'$};
     
     \filldraw[gray] (88:2) circle (.25);
     \node[color=gray] at (95:2.5) {$V_i$};
\end{tikzpicture}

\caption{Illustration for the notations in the beginning of Section \ref{section: construction}.}
\end{center}
\end{minipage}

Notice that a priori $l_{\gamma_i}$ is only an affine line with respect to the complex affine coordinates. To compute the family Floer mirror, we need to have a better control of the BPS rays in terms of the symplectic affine structure.  The following observation comes from \eqref{central charge explicit} directly.
\begin{lem} \label{straight}
  Any ray with a constant phase is affine with respect to the symplectic affine structure. In particular, $l_{\gamma_i}$ is an affine line with respect to the symplectic affine structure. 
\end{lem}
\begin{proof}
   Any such ray can be parameterized by $z=Ct$ for some complex number $C$. From \eqref{central charge explicit}, the symplectic coordinates along the ray are given by
   ${x}_k=C'_k t^{\frac{2\pi k}{5}}, y_k=C''_k t^{\frac{2\pi k}{5}}$, for some $C'_k, C''_k\in \mathbb{R}$ and the lemma follows. In other words, such ray is given by the affine line $C''_kx_k=C'_ky_k$ with respect to the symplectic affine coordinates $(x_k,y_k)$.  
\end{proof}

Recall that the family mirror $\check{X}$ is defined by $\coprod \mathcal{U}_{\alpha}/\sim$, where $\mathcal{U}_{\alpha}$ is the maximum spectrum of $T_{U_{\alpha}}$, for refined enough (so that the Fukaya trick applies) open covering $\{U_{\alpha}\}_{\alpha\in A}$ and together with symplectic affine coordinates $\psi_{\alpha}:U_{\alpha}\rightarrow \mathbb{R}^2$ such that $\psi_{\alpha}(u_{\alpha})=0$ for some $u_{\alpha}\in U_{\alpha}$. 
We may take
 \begin{align*}
    \psi_{\alpha}(u)=
        (x(u)-x(u_{\alpha}),y(u)-y(u_{\alpha})).
 \end{align*} \footnote{Here we abuse the notation, denote $x,y$ the natural extension of clock-wisely across the branch cut if $U_{\alpha}$ intersects the branch cut and $u\in U_5$. In other words, one should replace $(x,y)$ by $(y,x-y)$ under the circumstances from \eqref{monodromy II}.}
 \begin{remark}\label{translation}
On one hand, from Remark \ref{chart}, we have $\mathcal{U}_{\alpha}$ identified with $\mbox{Val}^{-1}(\psi_{\alpha}(U))\subseteq (\Lambda^*)^2$. On the other hand, from Example \ref{no discs} and Theorem \ref{913}, if $U_{\alpha}\subseteq U_k$ for any $k\in 1,\cdots, 5$, then $\mathcal{U}_{\alpha}\cong \mathfrak{Trop}^{-1}(U_{\alpha})$. Here we use the symplectic affine coordinates $x,y$ on $U_{\alpha}$ to embed $U_{\alpha}$ into $\mathbb{R}^2$ as an affine submanifold and $\mathfrak{Trop}:(\mathbb{G}^{an}_m)^2\rightarrow \mathbb{R}^2$. Recall that there is a natural identification $(\Lambda^*)^2\cong (\mathbb{G}_m^{an})^2$ as sets such that the below diagram commutes.
    \begin{equation}
    \begin{tikzcd}[column sep=small] 
    (\Lambda^*)^2 \arrow[dr, "\mbox{Val}" left]
     \arrow[rr] & & \arrow[dl, "\mathfrak{Trop}"] (\mathbb{G}_m^{an})^2 \\
     & \R^2 &
     \end{tikzcd}
\end{equation} The two descriptions of $\mathcal{U}_{\alpha}$ simply differ by a translation 
  \begin{align*}
      \mbox{Val}^{-1}(U_{\alpha})&\rightarrow \mathfrak{Trop}^{-1}(U_{\alpha}) \\
       (z_1,z_2)&\mapsto (T^{x(u_{\alpha})}z_1, T^{y(u_{\alpha})}z_2) 
  \end{align*} under the above identification. 
\end{remark}

Let 
$\mathfrak{Trop}_i:(\mathbb{G}_m^{an})_i^2\rightarrow \mathbb{R}^2_i$
be the standard valuation map. Here we put an subindex $i$ for each analytic tori and later it would correspond to the five different tori. Now if $U_{\alpha}\subseteq U_i$ for some $i=1,\cdots, 5$ and $U_{\alpha}\cap U_{\alpha'}\neq \emptyset$ with the reference $u_{\alpha'}\in U_{i+1}$, then again from Ex \ref{no discs} and Theorem \ref{913}, one can naturally identify $\mathcal{U}_{\alpha}\coprod \mathcal{U}_{\alpha'}/\sim \cong \mathfrak{Trop}^{-1}(U_{\alpha}\cup U_{\alpha'})$. From the identification in Remark \ref{translation}, the function $T^{\omega(\gamma_{u_{\alpha}})}z^{\partial\gamma}\in T_{U_{\alpha}}$ and  $T^{\omega(\gamma_{u_{\alpha'}})}z^{\partial\gamma}\in T_{U_{\alpha'}}$ glue to a function on $\mathfrak{Trop}^{-1}(U_{\alpha}\cup U_{\alpha'})$, which is simply the restriction of $z^{\partial \gamma}$ on $(\Lambda^*)^2\cong (\mathbb{G}_m^{an})^2$.

Denote $U'_i=\cup_{\alpha} U_{\alpha}$, where $\alpha$ runs through those $u_{\alpha}\in U_i$. By taking refinement of the open cover, we may assume that $U_i \subseteq U'_i$ without loss of generality. Then we have the extension of the embedding $\mathfrak{Trop}^{-1}(U_i)\subseteq \check{X}$ to $\mathfrak{Trop}^{-1}(U'_i)\subseteq \check{X}$. From the previous discussion, the family Floer mirror is simply $\coprod_{i=1}^5\mathfrak{Trop}_i^{-1}(U_i')/\sim$.
Note that 
\begin{equation*}
    \begin{tikzcd}[remember picture]
    \check{X}= \bigcup_i \mathfrak{Trop}_i^{-1}(U_i')/\sim 
      & &  (\mathbb{G}_m^{an})^2 \\
     & \mathfrak{Trop}_i^{-1}(U_i) &
     \end{tikzcd}
     \begin{tikzpicture}[overlay,remember picture]
\path (\tikzcdmatrixname-1-1) to node[midway,sloped]{\Large$\supseteq$}
(\tikzcdmatrixname-2-2);
\path (\tikzcdmatrixname-1-3) to node[midway,sloped]{\Large$\subseteq$}
(\tikzcdmatrixname-2-2);
\end{tikzpicture}
\end{equation*} To distinguish the two inclusion, we will always view $\mathfrak{Trop}_i^{-1}(U_i)$ as a subset of $(\mathbb{G}_m^{an})^2$ and consider \[\alpha_i:\mathfrak{Trop}_i^{-1}(U_i)\rightarrow \check{X}.\] 
Notice that $\mathfrak{Trop}^{-1}_i(U'_i)$ only occupies a small portion of $(\mathbb{G}^{an}_m)^2$.
Thus we need to extend $\alpha_i$ to most part of  $(\mathbb{G}_m^{an})_i^2$. 
For the simplicity of the notation, we will still denote those extension of $\alpha_i$ by the same notation. 

Now we want to understand how $\mathcal{U}'_i$ glue with $\mathcal{U}'_{i+1}$. 
Let $V_i,V_{i+1}$ be any small enough rational domains on $B_0$ such that $V_i\subseteq U_i'$, $V_{i+1}\subseteq U_{i+1}'$ and the Fukaya's trick applies. Let $p\in V_i\cap V_{i+1}$ be the reference point and one has 
\begin{align*}
  (\mathbb{G}_m^{an})^2_i\supseteq\mathfrak{Trop}_i^{-1}(V_i) \supseteq \mathfrak{Trop}_i^{-1}(V_i\cap V_{i+1})\xrightarrow{\Psi_{i,i+1}} \mathfrak{Trop}_{i+1}^{-1}(V_i\cap V_{i+1}) \subseteq \mathfrak{Trop}_{i+1}^{-1}(V_{i+1})\subseteq (\mathbb{G}_m^{an})^2_{i+1},
\end{align*} 
where $\Phi_{i,i+1}=\alpha_{i+1}^{-1}\circ \alpha_i$ is given by $\Psi_{i,i+1}=S^{-1}_{u_{i+1},p}\circ \Phi_{i,i+1}\circ S_{u_i,p}$ by \eqref{2} and 
\begin{align*}
   \Phi_{i,i+1}: z^{\partial \gamma}\mapsto z^{\partial\gamma}(1+T^{\omega(\gamma_{i+1})}z^{\partial \gamma_i})^{\langle \gamma, \gamma_{i+1} \rangle} 
\end{align*} from Definition \ref{thm:wallfun} and Theorem \ref{913}. Here $\omega(\gamma_{i+1})$ is evaluated at $p$. From \eqref{boundary relation}, we have
$\langle \gamma_{i+1},\gamma_i\rangle=1$. Then with the notation and discussion below Remark \ref{translation}, we have $\Phi_{i,i+1}$ is simply the polynomial map 
  \begin{align} \label{mutation}
      &z^{\gamma_i}\mapsto z^{\gamma_i}(1+z^{\gamma_{i+1}})^{-1} \notag \\
      &z^{\gamma_{i+1}}\mapsto z^{\gamma_{i+1}}. 
  \end{align}
 Since near $l_{\gamma_{i+1}}$ one has $\omega(\gamma_{i+1})>0$, one has \begin{align} \label{eq:valequal}
   \mbox{val}(z^{\gamma})=\mbox{val}(z^{\gamma}(1+z^{\gamma_i})^{-1}). 
\end{align} Here we view $z^{\gamma}$ as a function on $(\Lambda^*)^2$ and $\mbox{val}$ is the valuation on $\Lambda^*$. 
Thus,the following commutative diagram holds under the identification $(\Lambda^*)^2\cong (\mathbb{G}_m^{an})^2$, 
\begin{equation} \label{eq:commdiag}
        \begin{tikzcd}
  &\mathfrak{Trop}_i^{-1}(V_i) \supseteq \mathfrak{Trop}_i^{-1}(V_i\cap V_{i+1}) \arrow[d, "\mathfrak{Trop}_i"]
  \arrow[r, "\Phi_{i,i+1}"] &\mathfrak{Trop}_{i+1}^{-1}(V_i\cap V_{i+1}) \arrow[d, "\mathfrak{Trop}_{i+1}"] \subseteq \mathfrak{Trop}_{i+1}^{-1}(V_{i+1}) \\
  &\R^2_i \supseteq V_i\cap V_{i+1} 
  \arrow[r, equal] & V_i\cap V_{i+1}  \subseteq \R^2_{i+1} 
\end{tikzcd}
\end{equation}
  We may view $(\Lambda^*)^2$ as the $\Lambda$-points of the scheme $(\mathbb{G}_m)^2=\mbox{Spec}\Lambda[z^{\pm\gamma_i},z^{\pm\gamma_{i+1}}]$. Then we have the commutative diagram from the functoriality of the GAGA map on objects:
\begin{equation} \label{eq:commdiag'}
\begin{tikzcd}
  &(\mathbb{G}_m^{an})^2 
  \arrow[r, "\Psi_{i,i+1}"] &(\mathbb{G}_m^{an})^2\\
  &(\mathbb{G}_m)^2  \arrow[r] \arrow[u , "\text{GAGA}"] & (\mathbb{G}_m)^2 \arrow[u, "\text{GAGA}" right]
\end{tikzcd}
\end{equation}
  Under the identification $(\Lambda^*)^2\cong (\mathbb{G}_m^{an})^2$, $\Psi_{i,i+1}$ is simply the restriction of the map $(\mathbb{G}_m^{an})^2\rightarrow (\mathbb{G}_m^{an})^2$ with the same equation as in (\ref{mutation}). Therefore, we have the same commutative diagram as in \eqref{eq:commdiag} with $V_i,V_{i+1}$ replaced by $U_i^+,U_{i+1}$ for any open subset $U^+_i\subseteq \mathbb{R}^2$ such that $\omega(\gamma_{i+1})>0$ on $U_i^+$, which we will choose it explicitly later. 

To see the largest possible extension $U^+_i$ and thus largest possible extension of the above diagram, we would want to know explicitly where $\omega(\gamma_{i+1})>0$. 
Viewing $B\cong \mathbb{C}$, we may take $U_i^+$ as the interior of the sector bounded by $l_{\gamma_i}$ and the ray by rotating $\frac{3\pi}{5}$ counter-clockwisely from $l_{\gamma_{i+1}}$ by Lemma \ref{central charge II} and this is the largest possible region (extending $U_i$ counter-clockwisely) such that $\omega(\gamma_{i+1})>0$ holds. In particular, $U_i^+$ occupies $U_i,U_{i+1}$ and half of $U_{i+2}$. 
Therefore, we have the following lemma 
\begin{lem}
   The inclusion $\alpha_i:\mathfrak{Trop}_i^{-1}(U_i)\hookrightarrow \check{X}$ can be extended to $\mathfrak{Trop}_i^{-1}(U_i^+)\hookrightarrow \check{X}$, $i=1,\cdots, 5$. We will still denote the inclusion map by $\alpha_i$. In particular, $\alpha_i(\mathfrak{Trop}^{-1}_i(U_i\cup U_{i+1}))\subseteq \check{X}$. Here we use the convention $U_{i+5}=U_i$. 
\end{lem}
Notice that the commutative diagram $\eqref{eq:commdiag}$ no longer holds on $U_{i+2}\setminus U_i^+$ since 
\begin{align}\label{trop mutation}
   \mbox{val}(z^{\gamma_{i}}\big(1+z^{\gamma_{i+1}}\big)^{-1})=\mbox{val}(z^{\gamma_i})-\mbox{val}(1+z^{\gamma_{i+1}})=\mbox{val}(z^{\gamma_{i}})-\mbox{val}(z^{\gamma_{i+1}})
\end{align} 
outside of $U^+_i$, which is no longer $\mbox{val}(z^{\gamma_{i}})$ on the right hand side as in \eqref{eq:valequal}. 
Now for $V_i$ disjoint from $U_i^+$ and $V_{i+1}\subseteq U_{i+2}\subseteq U_{i+1}^+$, the diagram becomes
\begin{equation} \label{eq:commdiag1}
        \begin{tikzcd}
  &\mathfrak{Trop}_i^{-1}(V_i) \supseteq \mathfrak{Trop}_i^{-1}(V_i\cap V_{i+1})\setminus \{1+z^{\gamma_{i+1}}=0\} \arrow[d, "\mathfrak{Trop}_i"]
  \arrow[r, "\Psi_{i,i+1}"] & \arrow[d, "\mathfrak{Trop}_{i+1}"] \mathfrak{Trop}_i^{-1}(V_i\cap V_{i+1})\subseteq (\mathbb{G}_m^{an})^2 \\
  &\R^2_i \supseteq V_i\cap V_{i+1} 
  \arrow[r, "\phi_{i,i+1}"] & \R^2_{i+1}
\end{tikzcd} 
\end{equation} Recall that
\begin{align*}
   \mbox{val}(z^{\gamma_i})=\int_{\gamma_i}\omega=x_i, \hspace{3mm}
   \mbox{val}(z^{\gamma_{i+1}})=\int_{\gamma_{i+1}}\omega=y_i 
\end{align*} from \eqref{ac}
and thus together with \eqref{trop mutation} we have 
\begin{align} \label{shearing transf}
    \phi_{i,i+1}: x_i&\mapsto x_i-y_i \notag \\
    y_i& \mapsto y_i
\end{align} on its domain. Notice that $\Psi_{i,i+1}$ is only defined when $1+z^{\gamma_{i+1}}\neq 0$. 
 Since the linear map \eqref{shearing transf} is well-defined on $\mathbb{R}^2$, we will still use the same notation for such natural extension.



\begin{lem}\label{shearing lem}
  $\phi_{i,i+1}(U_{i+2}\setminus U_i^+)\subseteq U_{i+1}^+$. In particular,
  \[\alpha_i\big(\mathfrak{Trop}_i^{-1}(U_{i+2})\big)\subseteq \alpha_{i+1}\big(\mathfrak{Trop}_{i+1}^{-1}(U_{i+1}^+)\big)\subseteq \check{X}.\]
\end{lem} 
\begin{proof} 
    The left boundary of $U_i^+$ is characterized by $x_{i+1}=0,y_{i+1}>0$ and the left boundary of $U_{i+1}^+$ is characterized by $x_{i+1}<0,y_{i+1}=0$ from Lemma \ref{central charge II}. Therefore, we may identify the region bounded by the above two affine lines with the third quadrant of $\mathbb{R}^2_{x_{i+1},y_{i+1}}$ as affine manifolds. Notice that this is a subset of $U_{i+1}^+$. Under such identification, we have $U_{i+2}\setminus U_i^+$ as the region bounded by $x_{i+1}+y_{i+1}=0$ and $y_{i+1}$-axis in the third quadrant by Lemma \ref{straight}. In terms of $(x_{i+1},y_{i+1})$, \eqref{shearing transf} becomes 
     \begin{align*}
         \phi_{i,i+1}: x_{i+1}&\mapsto x_{i+1} \\
                       y_{i+1}&\mapsto x_{i+1}+y_{i+1},
     \end{align*} from the relation $\gamma_i+\gamma_{i+2}=\gamma_{i+1}$. The lemma then follows from direct computation.
\end{proof}

To sum up, one can extend the original inclusion $\alpha_i\big(\mathfrak{Trop}_i^{-1}(U_i)\big)\subseteq \check{X}$ in the counter-clockwise direction to
\begin{align}\label{inclusion1}
\alpha_i\big(\mathfrak{Trop}_i^{-1}(\overline{U_i\cup U_{i+1}\cup U_{i+2}})\setminus \{1+z^{\gamma_{i+1}}=0\}\big) \subseteq \check{X}. 
\end{align} Here we use $\overline{U}$ to denote the interior of the compactification of $U$.
\begin{lem} \label{filling the hole}
  The inclusion \eqref{inclusion1} extends over $\{1+z^{\gamma_{i+1}}=0\}\setminus \mathfrak{Trop}_i^{-1}(0)$.
\end{lem}
\begin{proof}
   Let $W_i$ be small neighborhood of (a component of ) $\partial U_i^+$ such that $\{1+z^{\gamma_{i+1}}=0\}\subseteq \mathfrak{Trop}_i^{-1}(W_i)$. 
   Notice that from Lemma \ref{shearing lem}, we have that 
     $ \mathfrak{Trop}\left(\alpha_i(\mathfrak{Trop}_i^{-1}(W_i))\right)\subseteq U_{i+2}$.
   We will show that  
   \begin{align}\label{inclusion4}
        \alpha_i\big(\mathfrak{Trop}_i^{-1}(W_i)\big)\subseteq \alpha_{i+1}\big(\mathfrak{Trop}_{i+1}^{-1}(U_{i+1}^+)\big)\cup \alpha_{i+2}\big(\mathfrak{Trop}_{i+2}^{-1}(U_{i+2})\big)
        \cup \alpha_{i+3}\big(\mathfrak{Trop}_{i+3}^{-1}(U_{i+2})\big). 
   \end{align} 
   From the earlier discussion, we have 
   \begin{align*}
      \alpha_i\left(\mathfrak{Trop}_i^{-1}(W_i)\setminus \{1+z^{\gamma_{i+1}}=0\}\big)\subseteq \alpha_{i+1}\big(\mathfrak{Trop}_{i+1}^{-1}(U_{i+1}^+)\right).
   \end{align*}
  From the earlier discussion, we have 
   \begin{align}
    &  \Psi_{i+1,i+2}: \mathfrak{Trop}_{i+1}^{-1}(U_{i+2})\cong \mathfrak{Trop}_{i+2}^{-1}(U_{i+2}) \notag \\
     & \Psi_{i+3,i+2}: \mathfrak{Trop}_{i+3}^{-1}(U_{i+2})\cong  \mathfrak{Trop}_{i+2}^{-1}(U_{i+2}). 
   \end{align} Recall that $\Psi_{i,j}=\alpha_j^{-1}\circ\alpha_i$. 
   It suffices to check that
   \begin{align}
A=\{1+z^{\gamma_{i+1}}=0\}\subseteq \Psi_{i+2,i}\big(\mathfrak{Trop}_{i+2}^{-1}(U_{i+2})\big) 
\cup \Psi_{i+3,i}\big(\mathfrak{Trop}_{i+3}^{-1}(U_{i+2})\big)
   \end{align} as subsets of $(\mathbb{G}_m^{an})^2_i$. 
 Straight calculation shows that \begin{align*}
   \Psi_{i,i+2}:& \mathfrak{Trop}_i^{-1}(W_i) \rightarrow \mathfrak{Trop}_{i+2}(U_{i+2}) \\
   z^{\gamma}\mapsto & z^{\gamma}(1+z^{\gamma_{i+2}})^{\langle\gamma,\gamma_{i+2}\rangle} \left(1+\frac{z^{\gamma_{i+1}}}{1+z^{\gamma_{i+2}}} \right)^{\langle \gamma,\gamma_{i+2}\rangle}
 \end{align*} Since $\langle\gamma,\gamma_{i+2}\rangle>0$ and $\langle\gamma,\gamma_{i+1}\rangle>0$ over $U_{i+2}$. We have $\Psi_{i,i+2}$ is not defined only on
 \begin{align*}
    B= \{ 1+z^{\gamma_{i+2}}=0\} \cup \{1+z^{\gamma_{i+1}}+z^{\gamma_{i+2}}=0\}.
 \end{align*} Therefore, we have $\alpha_i$ can be extended over $\mathfrak{Trop}_i^{-1}(W_i)\setminus B$. 
 Similarly, $\Psi_{i,i+3}$ is defined except 
 \begin{align*}
    C= \{1+z^{\gamma_{i+3}}=0\} \cup \{1+z^{\gamma_{i+2}}+z^{\gamma_{i+3}}=0\} \cup \{1+z^{\gamma_{i+1}}+z^{\gamma_{i+2}}+z^{\gamma_{i+3}}=0\}. 
 \end{align*} Therefore, $\alpha_i$ can be extended over $\mathfrak{Trop}_i^{-1}(W_i)\setminus C$. It is easy to check that $A\cap B\cap C=\{z^{\gamma_{i+1}}=z^{\gamma_{i+2}}=-1\}\subseteq \mathfrak{Trop}^{-1}(0)$. Since $\Psi_{i,j}=\alpha_j^{-1}\circ \alpha_i$ and thus the extension is compatible. Now the lemma is proved. 
\end{proof}
For the same reason, one can extend the inclusion in the clockwise direction 
 \begin{align}\label{inclusion2}
    \alpha_i\big( \mathfrak{Trop}_i^{-1}(\overline{U_i\cup U_{i-1}\cup U_{i-2}})\big)\subseteq \check{X}.
\end{align} Notice that $l_{\gamma_{i+3}}=l_{\gamma_{i-2}}$ is the the boundary of both $U_{i+2}$ and $U_{i-2}$.    
Then \eqref{inclusion1}\eqref{inclusion2} together imply the inclusion 
\begin{align}\label{inclusion3}
  \alpha_i\big(\mathfrak{Trop}_i^{-1}(\mathbb{R}^2\backslash l_{\gamma_{i+3}})\big)\subseteq \check{X}.
\end{align}
Then  Lemma \ref{cancellation} below guarantees that the inclusion extends over the ray $l_{\gamma_{i+3}}$ and we reach an extension 
\begin{align*}
   \alpha_i: \mathfrak{Trop}_i^{-1}(\mathbb{R}^2\setminus \{0\} )\rightarrow \check{X}. 
\end{align*}
Finally we claim that $\alpha_i$ is an embedding restricting on $\mathfrak{Trop}^{-1}(U)$ for small enough open subset $U\subseteq \mathbb{R}^2$. On the other hand, $\alpha_i$ is fibre-preserving with respect to $\mathfrak{Trop}_i:(\mathbb{G}_m^{an})^2\rightarrow \mathbb{R}^2$ and $\mathfrak{Trop}:\check{X}\rightarrow B$ and the induced map on the base is piecewise-linear. Direct computation shows that induced map on the base is injective. Therefore, $\alpha_i$ is an embedding. Therefore $\check{X}$ has a partial compactification $\bigcup_{i=1}^5 (\mathbb{G}_m^{an})^2_i/\sim$, with the identification $\Psi_{i,j}:(\mathbb{G}_m^{an})^2_i\rightarrow (\mathbb{G}_m^{an})^2_j$.
The following lemma is due to Gross-Siebert \cite{GS} and we leave the proof in the appendix for self-containedness.
\begin{lem} \label{cancellation}
  The composition of the wall-crossing transformations cancel out the monodromy. Explicitly,
  \begin{align*}
     \mathcal{K}_{\gamma_5} \mathcal{K}_{\gamma_4}\mathcal{K}_{\gamma_3}\mathcal{K}_{\gamma_2}\mathcal{K}_{\gamma_1} (z^{\partial \gamma})=z^{M^{-1}(\partial \gamma)}.
  \end{align*} 
\end{lem}

\begin{rmk}
  One would naturally expect that the family Floer mirror of the hyperK\"ahler rotation of $X'_t$ still compactifies to the del Pezzo surface of degree five. In this case, there is only two families of holomorphic discs in each of the singularities and one can glue the local model in \cite{KS1}*{Section 8} and get a partial compactification of the family Floer mirror. The authors will compare it with the Gross-Siebert construction of the mirror in the future work. 
\end{rmk}
   

  \begin{remark}
   Shen, Zaslow and Zhou prove the homological mirror symmetry for the $A_2$ cluster variety featuring the canonical equivariant $\mathbb{Z}_5$-action \cite{SZZ}.
  \end{remark}

 \subsection{Comparison with GHK Mirror of $dP_5$} \label{comparison w/ GHK}
 Let $\YGHK$ be the del Pezzo surface of degree five and $\DGHK$ be the anti-conical divisor consists of wheel of five rational curves. 
 Here we will explain the comparison of the family Floer mirror of $X_{II}$ with the GHK mirror of $(\YGHK,\DGHK)$. Recall that in Lemma \ref{affine identification}, we identify the integral affine structures on $B_0$ and $B_{\GHK}$. Moreover, the BPS rays naturally divide $B_0$ into cones which is exactly the cone decomposition of $B_{\GHK}$.  
The canonical scattering diagram in this case is computed in \cite{GHK}*{Example 3.7} and all the $\mathbb{A}^1$- curves are shown in Figure \ref{fig:dp5curve}. 

\begin{lem}\label{identification space}
  There exists a homeomorphism $X_{II}\cong \YGHK\setminus \DGHK$.
\end{lem}
\begin{proof}
   From the explicit equation in Section \ref{section: setup}, a deformation of $X_{II}$ has two singular fibres of type $I_1$ and the vanishing cycles have intersection number $1$. On the other hand, \cite{A4}*{Example 3.1.2} provides the local model of Lagrangian fibration near the blow-up of a point on the surface. Since $\YGHK$ can be realized as the blow up of two non-toric boundary points on del Pezzo surface of degree $7$, One can topologically glue the pull-back of the moment map torus fibration with the local Lagrangian fibration to get a torus fibration on $\YGHK\setminus \DGHK$ with two nodal fibres such that the vanishing cycles have intersection $1$. This gives the homeomorphism between $X_{II}$ and $\YGHK\setminus \DGHK$ topologically and the identification of the class of tori among $H_2(X_{II},\mathbb{Z})\cong H_2(\YGHK\setminus \DGHK,\mathbb{Z})$. In particular, we can use $\YGHK$ as an auxiliary topological compactification of $X_{II}$. 
\end{proof}
For the rest of this subsection, we will prove Theorem \ref{main thm}. 
\begin{proof}(of Theorem \ref{main thm} for the case $X_{II}$)
We will take $P= \mathrm{NE}(Y)$ in the Gross-Hacking-Keel construction.
We have $P^{gp}\cong \mbox{Pic}(Y)^*\cong H_2(Y,\mathbb{Z})$, where the first isomorphism comes from the Poincare duality and $Y$ being projective while the second isomorphism comes from $H^{1,0}(Y)=H^{2,0}(Y)=0$ . The rank two lattice $H_1(L_u,\mathbb{Z})$ glues to a local system of lattices over $B_0$ and is naturally identified with $\Lambda_{B_0}$ by Remark \ref{rk:affine}.
 Then we have the commutative diagram except the middle map. Here $\underline{H_2(Y,\mathbb{Z})}$ denotes the constant sheaf with fibre $H_2(Y,\mathbb{Z})$ over $B_0$, $\Gamma$ (and $\Gamma_g$) is the local system of lattices over $B_0$ with fibre $H_2(Y,L_u;\mathbb{Z})$ ($H_1(L_u)$ respectively) over $u\in B_0$.
 \begin{align}\label{identification of short exact sequence}
     \xymatrix{
    0 \ar[r] & \underline{H_2(Y,\mathbb{Z})}\ar[r] \ar[d]^{\cong}& \Gamma\ar[r]_{\partial} \ar[d]^{\Psi} & \Gamma_g\ar[d]^{\cong}\ar[r] & 0     \\ 
     0 \ar[r] & \underline{P^{gp}}\ar[r] & \mathcal{P} \ar[r]^{r} & \Lambda_{B_0}\ar[r]&0
     }
 \end{align} Notice that the bottom short exact sequence is \eqref{eqn:exact seq}. 
Next we will construct the middle map $\Psi$. 
Recall that $D_i^2<0$, by a theorem of Grauert \cite{G8}, one can contract $D_i$ to an orbifold singularity locally modeled by a neighborhood of the origin in $\mathbb{C}^2/D_i^2$. Since the blow up of $\mathbb{C}^2/D_i^2$ is the total space of $\mathcal{O}_{\mathbb{P}^1}(D_i^2)$, a neighborhood of $D_i$ is biholomorphic to a neighborhood of the zero section in $\mathcal{O}_{\mathbb{P}^1}(D_i^2)$. 
Therefore, a neighborhood of $D$ is covered by charts $W_i=\{(x_i,y_i)\in \mathbb{C}^2||x_iy_i|<1\}$ such that 
\begin{enumerate}
    \item $(Y,D)$ is modeled by $(W_i,\{x_iy_i=0\})$ near a node $D_i\cap D_{i+1}$
    \item $D_i=\{x_i=0\}$ and $D_{i+1}=\{y_i=0\}$
    \item $x_{i+1}=y_i^{-1}, y_{i+1}=x_iy_i^{-D_i^2}\big(1+O(|x_i|,|y_i|)\big)$.
\end{enumerate} Notice that $N_{D_i/Y}\cong \mathcal{O}_{\mathbb{P}^1}(D_i^2)$, the last equation comes from the transition functions for $\mathcal{O}_{\mathbb{P}^1}(D_i^2)$. 
 The torus fibre in $Y\setminus D$ near the node $D_i\cap D_{i+1}$ is isotopic to $L=\{|x_i|=|y_i|=\frac{1}{2}\}$. It is easy to see that $L$ bounds two families of discs $\{|x_i|\leq \frac{1}{2}, y_i=const\}$ and $\{x_i=const, |y_i|\leq \frac{1}{2}\}$. Let $\beta_i\in H_2(Y,L)$ be the relative class of the disc intersecting exactly once with $D_i$ and represented by a $2$-chain $b_i$.  
Over the simply connected subset $U_i\subseteq B_0$, both of the short exact sequence in \eqref{identification of short exact sequence} splits (non-canonically) and we define the middle map by $\Psi(\beta_i)=\phi_{i}(v_i)$. From Remark \ref{rk:affine}, the right hand side square commutes and 
$\partial \beta_i$ (up to parallel transport) generate $H_1(L_u,\mathbb{Z})$.  To see that the middle map is independent of $i$ and the left hand side square commutes, one has the following observation: We may choose $u$ to be in a neighborhood of $D_i$, which is diffeomorphic to $N_{D_i/Y}\cong \mathcal{O}_{\mathbb{P}^1}(D_i^2)$. The relation $x_{i+1}=y_{i}^{-1}, y_{i+1}=x_iy_i^{-D_i^2}\big(1+O(|x_i|,|y_i|)\big)$ translates to $\partial \beta_{i-1}+D_i^2\partial \beta_i+\partial \beta_{i+1}=0\in H_1(L,\mathbb{Z})$, which is the analogue of  \eqref{affine relation}. Therefore, there exists a $2$-chain $C$ in $L$ such that $C\cup b_{i-1}\cup D_i^2b_{i}\cup b_{i+1}$ is a $2$-cycle in the neighborhood of $D_i$. To lift the relation \eqref{affine relation} to $\Gamma$, notice that the $2$-cycle falls in $Y\setminus \cup_{j\neq i-1,i,i+1}D_j$, which is homeomorphic to the total space of $\mathcal{O}_{\mathbb{P}^1}(D_i^2)$, which has its second homology generated by $[D_i]$. Therefore, the $2$-cycle must be a multiple of $[D_i]$. On the other hand, the intersection number of this $2$-cycle with $D_{i-1},D_{i+1}$ are both $1$ from the explicit representative chosen and thus 
\begin{align} \label{QQ}
  \beta_{i-1}+D_i^2 \beta_i+\beta_{i+1}= [D_i],
\end{align} which is exactly the analogue of \eqref{lifting of affine relation} and defines the gluing relation in $\mathcal{P}$. We also remark that $H_2(Y,\mathbb{Z})$ is generated by $[D_i]$. Thus, we have the left hand side square of \eqref{identification of short exact sequence} also commutes and $\Psi$ is an isomorphism from the five lemma. 

 \begin{figure}[H]
  \centering
\def\svgwidth{220pt}
    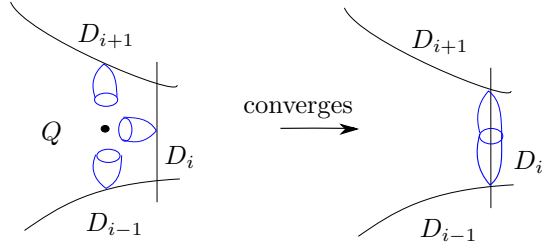
    \caption{Illustration for \eqref{QQ}.} \label{QQ'}
\end{figure}
Notice that $\beta_i+\gamma_i$ defines a $2$-cycle up to a multiple of the fibre. Since the fibre is contractible in $Y$, we may view $\beta_i+\gamma_i$ as a $2$-cycle in $H_2(Y,\mathbb{Z})$. Since $[E_i]$ is the unique class with intersections $[E_i].[D_j]=\delta_{ij}$, we have $z^{[E_i]-\phi_{i}(v_i)}$ identified with $ z^{\gamma_i}$ (see Figure \ref{fig:curves}).
\begin{figure}[H]
\centering
\def\svgwidth{180pt}
\begingroup%
  \makeatletter%
  \providecommand\color[2][]{%
    \errmessage{(Inkscape) Color is used for the text in Inkscape, but the package 'color.sty' is not loaded}%
    \renewcommand\color[2][]{}%
  }%
  \providecommand\transparent[1]{%
    \errmessage{(Inkscape) Transparency is used (non-zero) for the text in Inkscape, but the package 'transparent.sty' is not loaded}%
    \renewcommand\transparent[1]{}%
  }%
  \providecommand\rotatebox[2]{#2}%
  \newcommand*\fsize{\dimexpr\f@size pt\relax}%
  \newcommand*\lineheight[1]{\fontsize{\fsize}{#1\fsize}\selectfont}%
  \ifx\svgwidth\undefined%
    \setlength{\unitlength}{345.42510025bp}%
    \ifx\svgscale\undefined%
      \relax%
    \else%
      \setlength{\unitlength}{\unitlength * \real{\svgscale}}%
    \fi%
  \else%
    \setlength{\unitlength}{\svgwidth}%
  \fi%
  \global\let\svgwidth\undefined%
  \global\let\svgscale\undefined%
  \makeatother%
  \begin{picture}(1,0.70705314)%
    \lineheight{1}%
    \setlength\tabcolsep{0pt}%
    \put(0,0){\includegraphics[width=\unitlength,page=1]{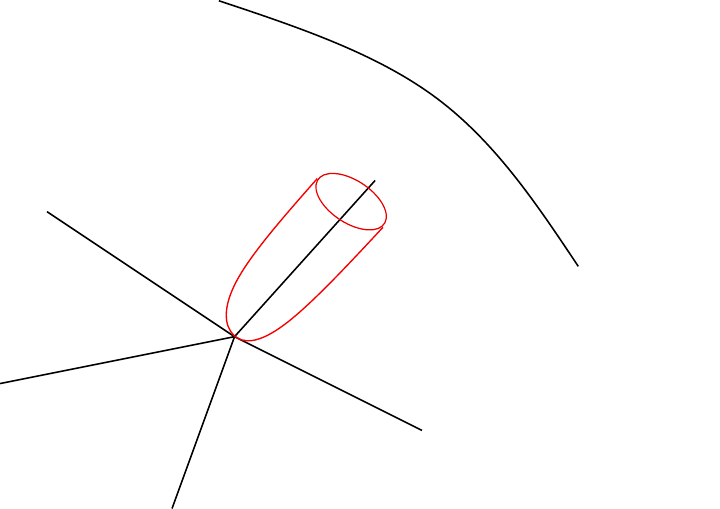}}%
    \put(0.74928992,0.26091947){\color[rgb]{0,0,0}\makebox(0,0)[lt]{\lineheight{1.25}\smash{\begin{tabular}[t]{l}$D_i$\end{tabular}}}}%
    \put(0.29333,0.39119336){\color[rgb]{0,0,0}\makebox(0,0)[lt]{\lineheight{1.25}\smash{\begin{tabular}[t]{l}$\gamma_i$\end{tabular}}}}%
    \put(0.41274816,0.55403656){\color[rgb]{0,0,0}\makebox(0,0)[lt]{\lineheight{1.25}\smash{\begin{tabular}[t]{l}$\beta_i$\end{tabular}}}}%
    \put(0,0){\includegraphics[width=\unitlength,page=2]{holocurves.pdf}}%
  \end{picture}%
\endgroup%

    \caption{The class $[E_i]$ decomposes into sum of $\gamma_i$ and $\beta_i$} \label{fig:curves}
\end{figure}
In particular, the transformation $\Psi_{i,i+1}$
coincides with the corresponding one in the canonical scattering diagram. This will leads to the identification of $\check{X}$ and the GHK mirror of $(\YGHK,\DGHK)$ as gluing of tori. Notice that the Gross-Hacking-Keel mirror of $(Y,D)$ comes with a family over $\mbox{Spec}\mathbb{C}[\mathrm{NE}(Y)]$.  
We will have to determine which particular point in $\mbox{Spec}\mathbb{C}[\mathrm{NE}(Y)]$ the family Floer mirror $\check{X}$ corresponds to. Notice that the monodromy sends $\gamma_i$ to $\gamma_{i+1}$. This implies that $\check{X}$ corresponds to the point such that the all values of 
$z^{[E_i]}$ coincide. From the explicit relation of curve classes $[E_i]$, $\check{X}$  corresponds to the point where $z^{[D_i]}=z^{[E_i]}=1$.

One can also see this from the identification $\check{X}$ with the subset of the analytification of del Pezzo surface of degree 5,
which is the cluster variety of type $A_2$ (see Section \ref{comparison A_2}).
Recall that the Gross-Hacking-Keel mirror is determined by the algebraic equations \eqref{theta function dp5} from the theta functions  \cite{GHK}*{Equation (3.2)}, 
 \begin{align*}
    \vartheta_{i-1}\vartheta_{i+1}=z^{[D_i]}(\vartheta_i+z^{[E_i]}).
 \end{align*}
 Comparing with \eqref{eq:rank2eq} (and later \eqref{theta function dp5}), we see that the family Floer mirror $\check{X}$ corresponds to the fibre with 
 \begin{align*}
     z^{[D_i]}=z^{[E_i]}=1.
 \end{align*}
 \end{proof}
Here we remark that the complex structure of the fibre defined by $z^{[D_i]}=z^{[E_i]}=1$ is expected to be mirror to the del Pezzo surface of degree five with monotone symplectic structure. However, the verification seems hard due to the analytic difficulty explained in Remark \ref{analytic difficulty}.
  \begin{figure}
  \centering
\def\svgwidth{220pt}
\begingroup%
  \makeatletter%
  \providecommand\color[2][]{%
    \errmessage{(Inkscape) Color is used for the text in Inkscape, but the package 'color.sty' is not loaded}%
    \renewcommand\color[2][]{}%
  }%
  \providecommand\transparent[1]{%
    \errmessage{(Inkscape) Transparency is used (non-zero) for the text in Inkscape, but the package 'transparent.sty' is not loaded}%
    \renewcommand\transparent[1]{}%
  }%
  \providecommand\rotatebox[2]{#2}%
  \newcommand*\fsize{\dimexpr\f@size pt\relax}%
  \newcommand*\lineheight[1]{\fontsize{\fsize}{#1\fsize}\selectfont}%
  \ifx\svgwidth\undefined%
    \setlength{\unitlength}{745.3323965bp}%
    \ifx\svgscale\undefined%
      \relax%
    \else%
      \setlength{\unitlength}{\unitlength * \real{\svgscale}}%
    \fi%
  \else%
    \setlength{\unitlength}{\svgwidth}%
  \fi%
  \global\let\svgwidth\undefined%
  \global\let\svgscale\undefined%
  \makeatother%
  \begin{picture}(1,0.97945429)%
    \lineheight{1}%
    \setlength\tabcolsep{0pt}%
    \put(0,0){\includegraphics[width=\unitlength,page=1]{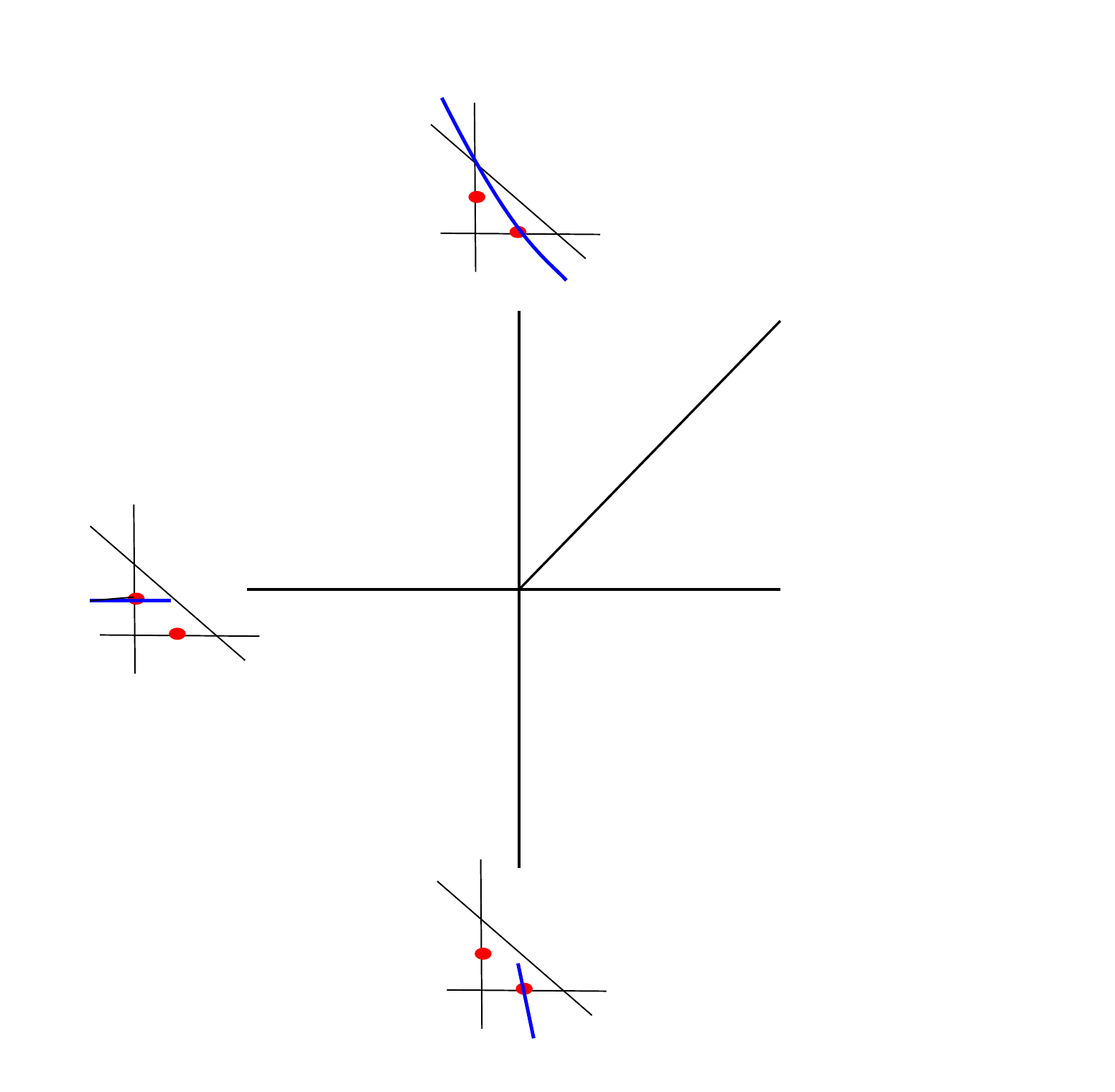}}%
    \put(0.4971546,0.01012156){\color[rgb]{0,0,0}\makebox(0,0)[lt]{\lineheight{1.25}\smash{\begin{tabular}[t]{l}$\small{E_y}$\end{tabular}}}}%
    \put(-0.00380788,0.33867484){\color[rgb]{0,0,0}\makebox(0,0)[lt]{\lineheight{1.25}\smash{\begin{tabular}[t]{l}$E_x$\end{tabular}}}}%
    \put(0,0){\includegraphics[width=\unitlength,page=2]{dp5curve.pdf}}%
    \put(0.77582861,0.31915687){\color[rgb]{0,0,0}\makebox(0,0)[lt]{\lineheight{1.25}\smash{\begin{tabular}[t]{l}$H-E_x$\end{tabular}}}}%
    \put(0.28896962,0.94073677){\color[rgb]{0,0,0}\makebox(0,0)[lt]{\lineheight{1.25}\smash{\begin{tabular}[t]{l}$H-E_y$\end{tabular}}}}%
    \put(0.66972186,0.83555634){\color[rgb]{0,0,0}\makebox(0,0)[lt]{\lineheight{1.25}\smash{\begin{tabular}[t]{l}$H-E_x-E_y$\end{tabular}}}}%
  \end{picture}%
\endgroup%

    \caption{The canonical scattering diagram and the $\mathbb{A}^1$-curves in del Pezzo surfaces of degree $5$ (illustrated by a projection to $\mathbb{P}^2$). } \label{fig:dp5curve}
\end{figure}
 We will show in Section \ref{comparison A_2} the following result:  
 \begin{thm}\label{mirror II}
     The analytification of $\cX$-cluster variety of type $A_2$ or the Gross-Hacking-Keel mirror of $(Y,D)$ is a partial compactification of the family Floer mirror of $X_{II}$. 
 \end{thm} 
\begin{rmk}
	Since here we do not include the singular fibre for the family Floer mirror, the family Floer mirror is missing $\mathfrak{Trop}^{-1}(0)\subseteq (\mathbb{G}_m^{an})^2$ in each rigid analytic torus. 
\end{rmk}

 \subsection{Comparison with $A_2$-Cluster Variety} \label{comparison A_2}
In this section, we will prove that the family Floer mirror constructed in Section \ref{section: construction} is simply the $\cX$-cluster variety of type $A_2$.
The $\cX$-cluster algebra of type $A_2$ are defined in Section \ref{sec: cluster} with $d_1= d_2 =1$.
The following observation helps to link the scattering diagram in Theorem \ref{913} and the scattering diagram of type $A_2$.

The operation below can be viewed as a symplectic analogue of ``pushing singularities to infinity" in \cite{GHK}. Recall that if one has a special Lagrangian fibration with a focus-focus singularity at $u_0$ and Lefschetz thimble $\gamma$. Then locally there exist two affine rays $l_{\pm\gamma}$ emanating from $u_0$ on the base, parametrizing special Lagrangian fibres bounding holomorphic discs in classes $\pm\gamma$ \cite{A}. Then $l_{\pm \gamma}$ divide a neighborhood of $u_0$ into two chambers $U_{\pm}$, where $U_{\pm}$ is characterized by $\int_{\pm \gamma}\mbox{Im}\Omega>0$. The corresponding wall crossings across $l_{\pm\gamma}$ from $U_-$ to $U_+$ are $\mathcal{K}_{\pm\gamma}$ and the monodromy around $u_0$ is given by $M$ in Claim \ref{focus-focus}. We make a branch cut from $u_0$ to infinity and the monodromy by $M$ when crossing the branch cut. Notice that the three transformations $\mathcal{K}_{\pm\gamma}$ and $M$ commute. If we choose the cut to coincide with $l_{-\gamma}$, then the transformation crossing $l_{-\gamma}$ from $U_-$ to $U_+$ is $K_{\gamma}$, which coincides with the transformation crossing $l_{\gamma}$ from $U_-$ to $U_+$. Similarly, if we choose the cut to coincide with $l_{\gamma}$, then the transformation crossing $l_{\gamma}$ from $U_+$ to $U_-$ is $K_{-\gamma}$, which coincides with the transformation crossing $l_{-\gamma}$ from $U_+$ to $U_-$.

 To sum up, choosing the branch cut coinciding with $l_{-\gamma}$ makes the transformation across $l_{\pm\gamma}$ from $U_-$ to $U_+$ both equal to $\mathcal{K}_{\gamma}$, as if the singularity $u_0$ is moved to infinity along $l_{-\gamma}$. Similarly, if we choose the branch cut coincides with $l_{\gamma}$, then the transformation from $U_-$ to $U_+$ is $K_{-\gamma}$ as if the singularity is moved to infinity along $l_{\gamma}$. 
 
 Now back to the scattering diagram in Theorem \ref{913}. We can express the underlying integral affine structure on $B_0$ in a different way by choosing different branch cuts. First we decompose $M=M_1M_2$, where $M_1,M_2$ are the Picard-Lefschetz transformations with vanishing cycles $\gone,\gtwo$. Choose the branch cut to be $l_{\gamma_1}$ (and $l_{\gamma_5}$) with the corresponding identifications to be $M_1$ (and $M_2$ respectively) as in Figure \ref{fig:dp5 branch cut'}. Then from the previous discussion in this section and the same argument in Section \ref{section: construction}, the family mirror is thus gluing of five tori with the gluing coincide with those of the $A_2$-cluster variety $\check{X}_{\mathbb{C}}$.

 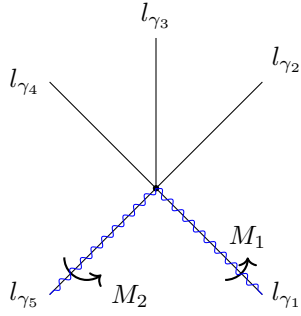
\begin{figure}[H]
\centering
\begin{tikzpicture}
\draw (0,0)--++(45:2) node[above right] {$l_{\gamma_2}$};
\draw (0,0)--++(90:2) node[above] {$l_{\gamma_3}$};
\draw (0,0)--++(135:2)node[left] {$ l_{\gamma_4}$};
\draw (0,0)--++(-45:2) node[right] {$l_{\gamma_1}$};
\draw (0,0)--++(-135:2) node[left] {$ l_{\gamma_5}$} ;
\filldraw(0,0) circle(1pt) ;
\draw[thick, ->] ([shift={(270:.3)}] -45:1.3)  arc (270:360:.3) node[above]{$M_1$};
\draw[thick, ->] ([shift={(180:.3)}] -135:1.3)  arc (180:310:.3) node[below right]{$M_2$};
\draw[color=blue, snake=snake,
segment amplitude=.5mm,
segment length=2mm]  (0,0)--++(-135:2);
\draw[color=blue, snake=snake,
segment amplitude=.5mm,
segment length=2mm]  (0,0)--++(-45:2);
\end{tikzpicture}
\caption{The different choice of branch cuts for $X_{II}$.}
\label{fig:dp5 branch cut'}
\end{figure}

Note that one can similarly define theta function in the analytic situation. Since we are working with finite type, we can express theta functions in different torus charts by path ordered products. The functions are well defined since the scattering diagram is consistent (see Lemma \ref{cancellation}). Further note that, in the finite case, we can replicate \eqref{eq:rank2eq} to define multiplications between theta functions without broken lines.\footnote{In general, the products of  theta functions can be expressed as the linear combination of theta functions \cites{GHK, GHKK}, which the coefficients can be computed via broken lines.} 
 Standard and straight-forward calculation shows that  
 \begin{align} \label{theta function dp5}
     \vartheta_{v_{i-1}} \cdot \vartheta_{v_{i+1}} &= 1+\vartheta_{v_{i}},
\end{align} 
where $v_i$ denotes the primitive generator of $l_{\gamma_i}$ $i \in \{1, \dots, 5\}$ ordered cyclically.
We can see it agrees with the exchange relations as in Section \ref{RES}. 
This gives a natural embedding of $\check{X}_{\mathbb{C}}$ into $\mathbb{P}^5$ after suitable homogenization of \eqref{theta function dp5} thus compactified to a del Pezzo surface of degree five. 


 \section{Family Floer Mirror of $X_{III}$}\label{section:dp6}
 In this section, we will consider the case when $Y'=Y'_{III}$ is a rational elliptic surface with singular configuration $III^*III$, $D'$ is the type $III^*$ fibre. We claim that the family Floer mirror of $X=X_{III}$ is then the del Pezzo surface of degree $6$. The argument is similar to that in Section \ref{section: dp5}.
 
 First of all, such $Y'$ has the explicit affine equation
    \begin{align*}
      y^2=x^4+u.
    \end{align*}
 It is easy to see that the fibre over $u=0$ is a singular fibre of type $III$, while the fibre at infinity is of type $III^*$. There is a natural deformation $Y'_t$ given by the minimal resolution of the surface 
   \begin{align*}
       \{ z^2y^2=x^4+4t^2x^2z^2+uz^4\} \subseteq \mathbb{P}^2_{(x:y:z)}\times \mathbb{P}^1_{(t:u)}
   \end{align*} such that there are two singular fibres of type $I_1, I_2$ with near $u=0$, $|t|\ll 1$. With vanishing thimbles $\gone$ and $\gtwo,\gthree$.
By Theorem \ref{399}, we have the analogue of Theorem \ref{913}. 
 \begin{thm}\cite{L12}*{Theorem 4.12} \label{BPS ray III}
   There exist $\gone,\gtwo,\gthree \in  H_2(X,L_u)\cong \mathbb{Z}^3$ such that $\langle \gone,\gtwo \rangle=\langle \gone,\gthree\rangle=1$, $\langle \gtwo,\gthree\rangle =0$ and $Z_{\gtwo}=Z_{\gthree}$. Moreover, if we set 
\[\gamma_1=-\gone,\ \gamma_2=\gtwo,\ \gamma_3=\gone+\gtwo+\gthree,\ \gamma_4=\gone+\gtwo,\ \gamma_5=\gone, \gamma_6=-\gthree.\] Then
	\begin{enumerate}
	    \item $f_{\gamma}(u)\neq 1$ if and only if $u\in l_{\gamma_i}$ and $\gamma=\gamma_i$ for some $i\in\{1,\cdots, 6\}$ .
	    \item In such cases, 
	    \begin{align*}
       f_{\gamma_i}=\begin{cases}1+T^{\omega(\gamma_i)}z^{\partial\gamma_i} & \mbox{ if $i$ odd,} \\
       (1+T^{\omega(\gamma_i)}z^{\partial\gamma_i})^2 & \mbox{ if $i$ even.}
       \end{cases}
	    \end{align*}
	\item If we choose the branch cut between $l_{\gamma_1}$ and $l_{\gamma_6}$, then the counter-clockwise monodromy $M$ across the branch cut is given by 
	\begin{align}
     \gone &\mapsto -(\gone+\gtwo+\gthree) \notag \\
     \gtwo &\mapsto \gone+\gtwo \notag \\
     \gthree &\mapsto \gone +\gthree.
	\end{align}
	\end{enumerate}
 \end{thm}
 Notice that from the condition $Z_{\gtwo}=Z_{\gthree}$, we have $l_{\gtwo}=l_{\gthree}$ and $l_{\gone+\gtwo}=l_{\gone+\gthree}$. 
 Then we compute the central charges $Z_{\gamma_i}$, which is parallel to Lemma \ref{central charge II}. Taking the branch cut between $l_{\gamma_1}$ and $l_{\gamma_6}$, we would obtain the diagram as in Figure \ref{fig:b2}.
 
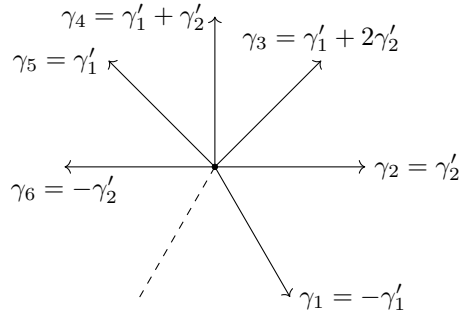
\begin{figure}[H]
\centering
\begin{tikzpicture}
\draw[->] (0,0)--++(300:2) node[right] {$\gamma_1 = -\gone$};
\draw[->] (0,0)--++(0:2) node[right] {$\gamma_2 =  \gtwo$};
\draw[->] (0,0)--++(45:2)node[above] {$ \gamma_3 = \gone+2 \gtwo$};
\draw[->] (0,0)--++(90:2) node[left] {$\gamma_4 = \gone +\gtwo$};
\draw[->] (0,0)--++(135:2) node[left] {$ \gamma_5 = \gone$} ;
\draw[->] (0,0)--++(180:2) node[below] {$ \gamma_6 = -\gtwo$} ;
\filldraw(0,0) circle(1pt) ;
\draw[dashed] (0,0)--++(240:2) ;
\end{tikzpicture}
\caption{BPS rays near the singular fibre in $X_{III}$.}  \label{fig:b2}
\end{figure}

 \begin{lem} \label{central charge III}
   With suitable choice of coordinate $u$ on $B_0\cong \mathbb{C}^*$, we have 
    \begin{align}
      Z_{\gamma_k}(u)=\begin{cases} e^{\pi i (k-1)\frac{3}{4}}u^{\frac{3}{4}} & \mbox{if $k$ odd,}\\
     \frac{1-i}{2}e^{\pi i (k-2)\frac{3}{4}} u^{\frac{3}{4}} & \mbox{if $k$ even.}
      \end{cases}
    \end{align}
     In particular, the angle between $l_{\gamma_k}$ and $l_{\gamma_{k+1}}$ is $\frac{\pi}{3}$. See how the BPS rays position as demonstrated in Figure {\ref{fig:b2}}.
 \end{lem}
 \begin{proof}
 Straight-forward calculation shows that $Z_{\gamma_k}(u)=O(|u|^{\frac{3}{4}})$. Normalize the coordinate $u$ such that $Z_{\gamma_1}(u)=u^{\frac{3}{4}}$. Notice that $M{\gamma_k}=\gamma_{k+2}$,
 the case for $k$ being odd follows immediately. Similarly,
 when $k$ is even, $Z_{\gamma_k}(u)=ce^{\pi i(k-2)\frac{3}{4}}u^{\frac{3}{4}}$, for some $c\in \mathbb{C}$.  
With $Z_{\gamma_2}+Z_{\gamma_4}=Z_{\gamma_3}$ we gets $c=\frac{1-i}{2}$.
 \end{proof}
 We will take $U_i$ be the sector bounded by $l_{\gamma_i}$ and $l_{\gamma_{i+1}}$.
 Let $\check{X}$ to be the family Floer mirror constructed by Tu \cite{T4}.
 Again we denote the embedding  by $\alpha_i:\mathfrak{Trop}_i^{-1}(U_i)\rightarrow \check{X}$.  From Lemma \ref{central charge III}, $x_i>0$ on a sector symmetric with respect to $l_{\gamma_{i}}$ and angle $\frac{2\pi}{3}\times 2$. Thus, $\alpha_i$ can be extended to $\mathfrak{Trop}_i^{-1}\bigg(\overline{\bigcup_{k=i-2}^{k=i+2}U_k}\bigg)$. Following the same line of Lemma \ref{shearing lem} and Lemma \ref{filling the hole}, $\alpha_i$ extends to $\mathfrak{Trop}^{-1}\bigg(\overline{\bigcup_{k=i-2}^{k=i+3}U_k}\bigg)$.
 Finally, $\alpha_i$ extends over $l_{\gamma_{i+4}}$ from the following analogue of Lemma \ref{cancellation}. The proof is similar and we will omit the proof.
 \begin{lem}
     The composition of the wall-crossing transformations cancel out the monodromy. Explicitly,
  \begin{align*}
    \mathcal{K}_{\gamma_6} \mathcal{K}_{\gamma_5} \mathcal{K}_{\gamma_4}\mathcal{K}_{\gamma_3}\mathcal{K}_{\gamma_2}\mathcal{K}_{\gamma_1} (z^{\gamma})=z^{M^{-1}\gamma}.
  \end{align*} 
 \end{lem}


Similar to the argument of Section \ref{comparison A_2}, we may change the branch cut in Figure \ref{fig:b2} into two, as in Figure \ref{branch cuts dp6}. The explicit gluing functions of $B_2$-cluster variety can be found in \cite{thesis}*{p.54 Figure 4.1}. Then the family Floer mirror $\check{X}$ can be partially compactified to gluing of six tori (up to GAGA) with the gluing function same as the $\cX$ cluster variety of type $B_2$. 
One can compute the product of the theta functions via broken lines and obtain
\begin{align} \label{theta relation dp6}
    \vartheta_{v_1} \vartheta_{v_3} &= 1+ \vartheta_{v_2}, \notag \\
    \vartheta_{v_2} \vartheta_{v_4} &= (1+ \vartheta_{v_3})^2, \notag \\
    \vartheta_{v_3} \vartheta_{v_5} &= 1+ \vartheta_{v_4}, \notag \\
    \vartheta_{v_4} \vartheta_{v_6} &= (1+ \vartheta_{v_5})^2, \notag \\
    \vartheta_{v_5} \vartheta_{v_1} &= 1+ \vartheta_{v_6}, \notag \\
    \vartheta_{v_6} \vartheta_{v_2} &= (1+ \vartheta_{v_1})^2,
 \end{align} 
where $v_i$ denotes the primitive generator of $l_{\gamma_i}$ for $i \in \{1, \dots, 6\}$ ordered cyclically.
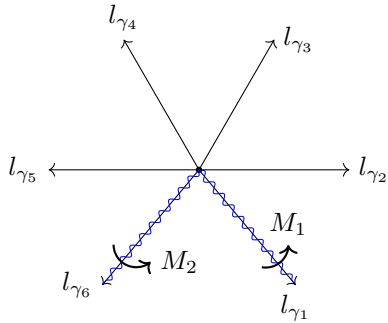
\begin{figure}[H]
\centering
\begin{tikzpicture}
\draw[->] (0,0)--++(0:2) node[right] {$l_{\gamma_2}$};
\draw[->] (0,0)--++(60:2) node[right] {$l_{\gamma_3}$};
\draw[->] (0,0)--++(120:2)node[above] {$l_{\gamma_4}$};
\draw[->] (0,0)--++(180:2) node[left] {$l_{\gamma_5}$};
\draw[->] (0,0)--++(230:2) node[left] {$l_{\gamma_6}$} ;
\draw[->] (0,0)--++(310:2) node[below] {$l_{\gamma_1}$} ;
\filldraw(0,0) circle(1pt) ;
\draw[color=blue, snake=snake,
segment amplitude=.5mm,
segment length=2mm]  (0,0)--++(230:2);
\draw[color=blue, snake=snake,
segment amplitude=.5mm,
segment length=2mm]  (0,0)--++(310:2);
\draw[thick, ->] ([shift={(270:.3)}] 310:1.3)  arc (260:360:.3) node[above]{$M_1$};
\draw[thick, ->] ([shift={(180:.3)}] 230:1.3)  arc (180:310:.3) node[right]{$M_2$};
\end{tikzpicture}
 \caption{The choice of a different branch cut for $X_{III}$.  }
 \label{branch cuts dp6}
\end{figure} 
Recall that \cite[Lemma 3.5]{GHK_bir} shows that the Looijenga pair comes from a non-toric blow up of a toric Looijenga pair, its Looijenga interior up to codimension two can be expressed as gluing of tori via a birational map determined by the ideal of the blow up. 
In particular, the $\mathcal{X}$ cluster variety of type $B_2$ up to codimension two can be compactified to a Looijenga pair. The relations in \eqref{theta relation dp6} indicate that the Looijenga pair is constructed by a non-toric blow up of $\mathbb{P}^2$, one point at $-1$ at the $x$-axis and two successive blow-ups at $-1$ of the $y$-axis. The geometry is an isotrivial degeneration of del Pezzo surface of degree six. 

To compare with the mirror constructed by Gross-Hacking-Keel, we take the corresponding log Calabi-Yau pair $(\YGHK,\DGHK)$ with $\YGHK$ being the del Pezzo surface of degree six. Since all del Pezzo surfaces of degree $6$ are isomorphic, we will identify it with the blow up of $\mathbb{P}^2$ at three points, two non-toric points on $y$-axis and one non-toric point on $x$-axis. The anti-canonical divisor $\DGHK$ is the proper transform of the $x,y,z$-axis of $\mathbb{P}^2$. Let $H$ be the pull-back of the hyperplane class, $E_1,$ (and $E_2,E_3$) be the exceptional divisor of the blow up on $x$-axis (and $y$-axis).  
\begin{lem} \label{affine dp6}
  There is an isomorphism of affine manifolds $B_{{\GHK}}\cong B$. 
\end{lem}
\begin{proof}
   From \cite{GHK}*{Lemma 1.6}, toric blow-ups correspond to the refinements of cone decompositions but does not change the integral affine structures. We will find a successive toric blow-ups of $(\tilde{Y},\tilde{D})\rightarrow (Y,D)$ such that not only the integral affine structure with singularity constructed by Gross-Hacking-Keel coincides with $B$ but also its cone decomposition is the same as the chamber structure bounded by the BPS rays. Such $\tilde{Y}$ can be constructed as the ordered blow ups at the intersection point of the $x,z$-axis, the proper transform of the $z$-axis and the exceptional divisor, the proper transform of $y,x$-axis. Then we take $\tilde{D}$ to be the union of the pull-back of the $x,y,z$-axis. If we take the proper transform of $y$-axis as $\tilde{D}_1$ and number the boundary divisors in counter-clockwise order, then we have $\tilde{D}_i^2=-1$ if $i$ odd and $\tilde{D}_{i}^2=-2$ if $i$ even. 
   
   Use \eqref{central charge III}, we have 
    \begin{align*}
        l_{\gamma_1}=&\{\check{x}>0, \check{y}=0\}  \\
        l_{\gamma_2}=&\{\check{y}>0,\check{x}=0 \}.
    \end{align*} and we will identify $l_{\gamma_1}=\mathbb{R}_{>0}(1,0)$ with $l_{\gamma_2}=\mathbb{R}_{>0}(0,1)$ and the rest of the proof is similar to the proof of Lemma \ref{affine identification}. 
\end{proof}

By the same argument of Lemma \ref{identification space}, we have a homeomorphism between $X_{III}\cong Y\setminus D\cong \tilde{Y}\setminus \tilde{D}$ and $\tilde{Y}$ provides a compactification of $X_{III}$. For the later discussion, we will replace $(Y,D)$ by $(\tilde{Y},\tilde{D})$ for the rest of the section (see Remark \ref{repacement}). Similarly, we have the identification of the short exact sequence \eqref{identification of short exact sequence}. 

Next we need to compute the canonical scattering diagram for $(Y,D)$. Let $D_i$ be the components of $D$ with $D_i$ are exceptional curves when $i$ even. 
\begin{lem}\label{identification wall functions}
  Under the identification of integral affine structures with singularities $B\cong \BGHK$, the canonical scattering diagram of Gross-Hacking-Keel coincides with the scattering diagram in Theorem \ref{BPS ray III} via identification $z^{[C_i]-\phi_{\rho_i}(v_i)}=z^{\gamma_i}$ (or $z^{[C^j_i]-\phi_{\rho_i}(v_i)}=z^{\gamma_i}$) for $i$ is odd (or even). 
  
\end{lem}
\begin{proof}
   We will first compute all the $\mathbb{A}^1$-curves of $(\YGHK,\DGHK)$, which is standard and we just include it for self-completeness. Any irreducible curves, in particular the irreducible $\mathbb{A}^1$ curves in $(\YGHK,\DGHK)$ are either exceptional curves of blow-up from $\mathbb{P}^2$ or proper transform of a curve $C\subseteq \mathbb{P}^2$. All the three exceptional curves are $\mathbb{A}^1$-curves intersecting  $D_i$ for $i$ odd.  
   If $C$ is of degree one and its proper transform is an $\mathbb{A}^1$-curve, then it either 
   \begin{enumerate}
       \item passes through two of the blow up points and its proper transform intersect $\tilde{D}_i$ for $i$ odd. There are three such lines. 
       \item passes through one blow up point and one intersection of toric $1$-stratum. There are three such lines and intersect $\tilde{D}_i$ for $i$ even.  
   \end{enumerate}
   There are no higher degree curves with proper transform are $\mathbb{A}^1$-curves and we draw the canonical scattering diagram and the corresponding $\mathbb{A}^1$-curves in Figure \ref{fig:dp6curve}. 
   
   Since $D\in |-K_Y|$ is ample, there is no holomorphic curves contained in $Y\setminus D$. In particular, all the simple $\mathbb{A}^1$-curves are irreducible and
   all the possible $\mathbb{A}^1$-curves are the multiple covers of the above ones. The contribution of multiple covers of degree $d$ is $(-1)^{d-1}/d^2$ by \cite{GPS}*{Proposition 6.1}. Then the lemma follows from the definition of the canonical scattering diagram \cite{GHK}*{Definition 3.3}. Then the function attached to the ray $\rho_i$ is 
   \begin{align}
   f_i=\begin{cases}
      (1+z^{[C_i]-\phi_{\rho_i}(v_i)}), & \mbox{ if $i$ is odd,} \\
       \prod_{j=1}^2 (1+z^{[C_i^j]-\phi_{\rho_i}(v_i)}), & \mbox{ if $i$ is even,} 
   \end{cases}
   \end{align}
    where $C_i,C_i^j$ are the $\mathbb{A}^1$-curve classes corresponding to $l_{\gamma_i}$  in Figure \ref{fig:dp6curve}. The assumption $Z_{\gamma_2}=Z_{\gamma_3}$ implies that $z^{E_2]}=z^{[E_3]}$.
Notice that the monodromy of the only singular fibre shifts $\gamma_i$ to $\gamma_{i+2}$. This implies that one would also need to identify 
 \begin{align*}
    z^{[E_1]}&=z^{[H-E_i]}=z^{[2H-E_1-E_2-E_3]}\\
    z^{[E_i]}&=z^{[H-E_i]}=z^{[H-E_1-E_i]}, i=2,3.
 \end{align*} Equivalently, this corresponds to 
    \begin{align*}
  z^{[D_i]}=z^{[C_i]}=z^{[C_i^j]}=1.   
 \end{align*}

\end{proof}

\begin{figure}[H]
\centering
\def\svgwidth{220pt}
    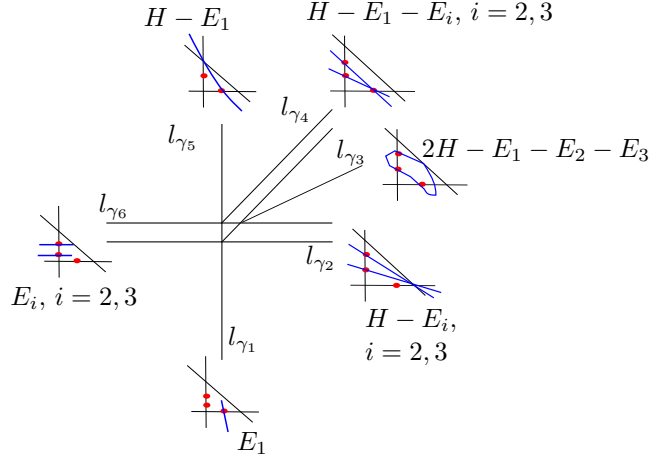
    \caption{The canonical scattering diagram and the $\mathbb{A}^1$-curves in del Pezzo surfaces of degree $6$ (illustrated by a projection to $\mathbb{P}^2$). } \label{fig:dp6curve}
\end{figure}

The GHK mirror can be computed via the spectrum of the algebra generated by theta functions. The products of the theta functions 

\begin{align*}
    \vartheta_{i-1} \vartheta_{i+1} & = z^{[D_i]} \prod_{j=1}^2 \left( \vartheta_{i} + z^{[C_i^j]}\right) \quad \text{for $i$ even,} \\
    \vartheta_{i-1} \vartheta_{i+1} & = z^{[D_i]} \left( \vartheta_{i} + z^{[C_i]}\right) \quad \text{for $i$ odd}.
\end{align*}

Again compare it with the analogue relations \eqref{theta relation dp6} from $\cX$-cluster algebra of type $B_2$, we conclude that the family Floer mirror $\check{X}$ corresponds to the particular fibre of the GHK mirror characterized by  
\begin{align*}
  z^{[D_i]}=z^{[C_i]}=z^{[C_i^j]}=1.   
 \end{align*}
Since the intersection matrix of components of $D$ is not negative semi-definite, the mirror family of $(Y,D)$ can be compactified to a family of Looijenga pairs which is deformation equivalent to $(Y,D)$ \cite[Theorem 1.8]{LZ}. From the previous discussion, the family mirror of $X$ corresponds to the unique Looijenga pair $(\check{Y},\check{D})$ with trivial periods in the mirror family. From \cite[Lemma 2.8]{GHK_tor}, such Looijenga pair $(\check{Y},\check{D})$ can be constructed explicitly: $\check{Y}$ is the blow up of $\mathbb{P}^2$ at $-1$ on the $x$-axis and two successive blow up at $-1$ at $y$-axis of $\mathbb{P}^2$. The boundary divisor $\check{D}$ is the proper transform of the toric boundary of $\mathbb{P}^2$. In particular, $\check{Y}$ is an isotrivial degeneration of del Pezzo surface of degree six. To sum up, we conclude the section with the following theorem.
\begin{thm}\label{mirror III}
  The family Floer mirror of $X_{III}$ has a partial compactification as the analytification of the $B_2$-cluster variety or the Gross-Hacking-Keel mirror of $(\YGHK,\DGHK)$ described above Lemma \ref{affine dp6}. In particular, the family Floer mirror of $X_{III}$ can be compactified as the analytification of certain isotrivial degeneration of del Pezzo surface of degree six. 
\end{thm}

 \section{Family Floer Mirror of $X_{IV}$} \label{section:dp4}
 In this section, we will consider the case when $Y'$ be a rational elliptic surface with singular configuration $IV^*IV$, $D'$ is the type $IV^*$ fibre and $X'=Y'\setminus D'$. Recall that $X$ is the hyperK\"ahler rotation of $X'$ such that the elliptic fibration becomes a special Lagrangian fibration. We claim that the family Floer mirror of $X$ is then the del Pezzo surface of degree $4$. The argument is also similar to that in Section \ref{section: dp5}.
 Such rational elliptic surface $Y'$ has Weiestrass model 
 \begin{align}
     y^2=x^3+t^2s^4.
 \end{align}
Moreover, it has a deformation to rational elliptic surface with its singular configuration $IV^*II \hspace{1mm} I_2$ from the classification of singular configuration of rational elliptic surface \cite{Per}. In other words, one can deform the complex structure of $X'$ such that it has two singular fibres, one is of type $II$ and another one is of type $I_2$. 

Del Pezzo surfaces of degree four are blow up of $\mathbb{P}^2$ at five generic points. Let $z_1,z_2,z_3$ be the homogeneous coordinates of $\mathbb{P}^2$. Choose the five points as follows: $z_1=z_2=0$, one point on $\{z_1=0\}$, one point on $\{z_2=0\}$ and two points on $\{z_3=0\}$. Denote the proper transform of the coordinate lines of $\mathbb{P}^2$ by $D$. Notice that the corner blow up does not affect the Looijenga interior, i.e. $X$. For each of the non-toric blow up, the local model in \cite{A4} introduces a nodal singular fibre in $X$ and the vanishing cycle is determined by the blown-up coordinate line. One can arrange the position of the blow up, which does not change the topological type of $X$, such that the singular fibres corresponding to the non-toric blow-ups on $\{z_1=0\},\{z_2=0\}$ are close and the singular fibres corresponding to the two non-toric blow-ups on $\{z_3=0\}$ are close. The vanishing cycles corresponding to the non-toric blow-ups on $\{z_1=0\},\{z_2=0\}$ have intersection one (up to choice of sign) and thus the fibration can be locally deformed topologically such that the corresponding singular fibres merged to a type $II$ singular fibre. The vanishing cycles corresponding to the non-toric blow-ups on $\{z_3=0\}$ have intersection zero and thus the fibration can be locally deformed topologically such that the corresponding singular fibres merged to a type $I_2$ singular fibre. Furthermore, the vanishing cycle of the $I_2$-fibre parallel transport to the $II$-fibre is one of the classes in Theorem \ref{local discs} that bounds holomorphic discs. 
From Lemma \ref{lem: monodromy Appendix}, if a type $IV$-fibre deforms to a type $II$-fibre and an $I_2$-fibre, then the vanishing cycle of the $I_2$-fibre (up to parallel transport) is one of the classes in Theorem \ref{local discs}. Recall that from Remark \ref{transitive}, the monodromy around the singular fibres permutes the classes in Theorem \ref{local discs} transitively. 
This shows that the two geometry $X$ and $Y\setminus D$ have exactly the same diffeomorphism type because hyperK\"ahler rotations do not change the underlying spaces.

  \begin{thm}\cite{L12}*{Theorem 4.14}\label{BPS dp4}
   There exist $\gone,\gtwo,\gthree,\gfour \in  H_2(X,L_u)\cong \mathbb{Z}^4$ such that $\langle \gone,\gi \rangle=1$, $\langle \gi,\gj\rangle =0$ and $Z_{\gi}=Z_{\gj}$, for $i,j\in \{2,3,4\}$. Moreover, if we set 
    \begin{align*}
    \gamma_1&=-\gone,\ \gamma_2=\gtwo,\ \gamma_3=\gone+\gtwo+\gthree+\gfour,\ \gamma_4=\gone+\gtwo+\gthree,\\
    \gamma_5&=2\gone+\gtwo+\gthree+\gfour,\ \gamma_6=\gone+\gtwo,\ 
    \gamma_7=\gone, \ 
    \gamma_8=-\gfour.
    \end{align*}
    Then
	\begin{enumerate}
	    \item $f_{\gamma}(u)\neq 1$ if and only if $u\in l_{\gamma_i}$ and $\gamma=\gamma_i$ for some $i\in\{1,\cdots, 8\}$ .
	    \item In such cases, 
	    \begin{align*}
       f_{\gamma_i}=\begin{cases}1+T^{\omega(\gamma_i)}z^{\partial\gamma_i} & \mbox{ if $i$ odd,} \\
       (1+T^{\omega(\gamma_i)}z^{\partial\gamma_i})^3 & \mbox{ if $i$ even.}
       \end{cases}
	    \end{align*}
	\item If we choose the branch cut between $l_{\gamma_1}$ and $l_{\gamma_8}$, then the counter-clockwise monodromy $M$ across the branch cut is given by 
	\begin{align}
     \gone &\mapsto -(\gone+\gtwo+\gthree+\gfour) \notag \\
     \gtwo &\mapsto \gone+\gtwo \notag \\
     \gthree &\mapsto \gone +\gthree \\
     \gfour &\mapsto \gone +\gfour .
	\end{align}
	\end{enumerate}
 \end{thm}

  \begin{lem} \label{central charge IV}
   With suitable choice of coordinate $u$ on $B_0\cong \mathbb{C}^*$, we have 
    \begin{align}
      Z_{\gamma_k}(u)=\begin{cases}e^{\frac{5\pi}{6} i (k-1)} u^{\frac{2}{3}} & \mbox{if $k$ odd,}\\
     \frac{1}{\sqrt{3}}e^{\frac{5\pi}{6} i (k-1)}e^{-\frac{\pi i }{6}} u^{\frac{2}{3}} & \mbox{if $k$ even.}
      \end{cases}
    \end{align}
       In particular, the angle between $l_{\gamma_i}$ and $l_{\gamma_{i+1}}$ is $\frac{\pi}{4}$. See how the BPS rays position as demonstrated in Figure {\ref{fig:g2}}. 
 \end{lem}
 
 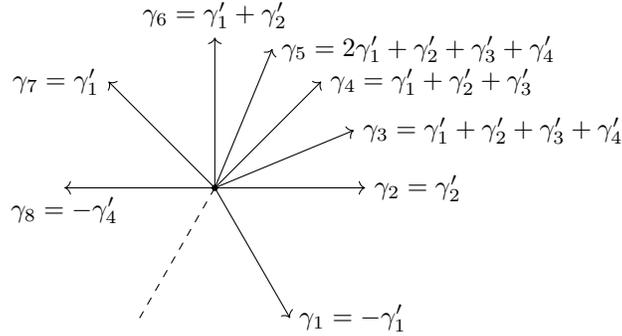
\begin{figure}[H]
\centering
\begin{tikzpicture}
\draw[->] (0,0)--++(300:2) node[right] {$\gamma_1 = -\gone$};
\draw[->] (0,0)--++(0:2) node[right] {$\gamma_2 =  \gtwo$};
\draw[->] (0,0)--++(22.5:2) node[right] {$\gamma_3 =  \gamma_1'+\gamma_2'+\gamma_3'+\gamma_4'$};
\draw[->] (0,0)--++(45:2)node[right] {$ \gamma_4 = \gamma_1'+\gamma_2'+\gamma_3'$};
\draw[->] (0,0)--++(67.5:2) node[right] {$\gamma_5 =  2\gamma_1'+\gamma_2'+\gamma_3'+\gamma_4'$};
\draw[->] (0,0)--++(90:2) node[above] {$\gamma_6 = \gone +\gtwo$};
\draw[->] (0,0)--++(135:2) node[left] {$ \gamma_7 = \gone$} ;
\draw[->] (0,0)--++(180:2) node[below] {$ \gamma_8 = -\gamma_4'$} ;
\filldraw(0,0) circle(1pt) ;
\draw[dashed] (0,0)--++(240:2) ;
\end{tikzpicture}
\caption{BPS rays near the singular fibre in $X_{IV}$. Note in Theorem \ref{BPS dp4}, we have $Z_{\gi}=Z_{\gj}$, for $i,j\in \{2,3,4\}$.}  \label{fig:g2}
\end{figure}

 \begin{proof}
    One can check that $Z_{\gamma}(u)=O(|u|^{\frac{2}{3}})$ and let $Z_{\gamma_k}(u)=c_ku^{\frac{2}{3}}$.
    Using the relations between $\gamma_i$ and straight-forward calculation show that
    \begin{align*}
     c_1=1, c_2=\frac{1}{\sqrt{3}}e^{-\frac{\pi i}{6}}, c_3=e^{-\frac{\pi i}{3}}, c_4=-\frac{i}{\sqrt{3}}
    \end{align*}
    after suitable normalization of the coordinate $u$. Then use the relation $M_{\gamma_i}=\gamma_{i+4}$ to determines the rest of $c_k$. 
 \end{proof}
 With the data above, the similar argument in Section \ref{section: construction} shows that the family Floer mirror of $X_{IV}$ is gluing of eight copies of $\mathfrak{Trop}^{-1}(\mathbb{R}^2\setminus\{0\})\subseteq (\mathbb{G}_m^{an})^2$ , with the gluing functions in Theorem \ref{BPS dp4}. 
 Similar to the argument of Section \ref{comparison A_2}, we may change the branch cut in Figure \ref{fig:g2} into two, as in Figure \ref{fig:linyushen}.
 
 \begin{figure}[H]
\centering
\begin{tikzpicture}
\draw[->] (0,0)--++(300:2) node[right] {$\gamma_1 = -\gone$};
\draw[->] (0,0)--++(0:2) node[right] {$\gamma_2 =  \gtwo$};
\draw[->] (0,0)--++(22.5:2) node[right] {$\gamma_3 =  \gamma_1'+\gamma_2'+\gamma_3'+\gamma_4'$};
\draw[->] (0,0)--++(45:2)node[right] {$ \gamma_4 = \gamma_1'+\gamma_2'+\gamma_3'$};
\draw[->] (0,0)--++(67.5:2) node[right] {$\gamma_5 =  2\gamma_1'+\gamma_2'+\gamma_3'+\gamma_4'$};
\draw[->] (0,0)--++(90:2) node[above] {$\gamma_6 = \gone +\gtwo$};
\draw[->] (0,0)--++(135:2) node[left] {$ \gamma_7 = \gone$} ;
\draw[->] (0,0)--++(180:2) node[below] {$ \gamma_8 = -\gamma_4'$} ;
\filldraw(0,0) circle(1pt) ;
\filldraw(180:1) circle(1pt);
\draw[color=blue, snake=snake,
segment amplitude=.5mm,
segment length=2mm]  (0,0)--++(300:2);
\draw[color=blue, snake=snake,
segment amplitude=.5mm,
segment length=2mm]  (0,0)--++(180:2);
\end{tikzpicture}
\caption{A choice of a different branch cut for $X_{IV}$} \label{fig:linyushen}
\end{figure}
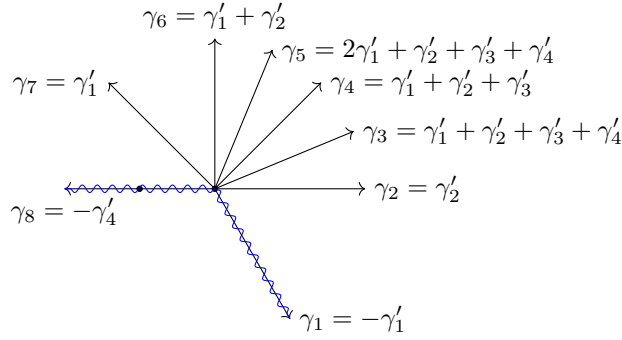
 
The scattering diagram of cluster type $G_2$ can be found in \cite{GHKK}*{Figure 1.2}. One can show that the corresponding gluing functions of the $\cX$ case are the same as those in Theorem \ref{BPS dp4} under suitable identification.
 Then the family Floer mirror of ${X}_{IV}$ can be partially compactified to gluing of eight tori (up to GAGA) with the gluing functions same as the $\cX$-cluster variety of type $G_2$. 
 
Next we will construct a log Calabi-Yau pair $(Y,D)$ such that the corresponding Gross-Hacking-Keel mirror corresponds to the family Floer mirror of $X_{IV}$. We will take 
\begin{enumerate}
    \item $Y$ to be the blow up of $\mathbb{P}^2$ at $4$ points, three of them are the non-toric points on $y$-axis and one non-toric point on $x$-axis.
    \item $D$ is the proper transform of $x,y,z$-coordinate axis. 
\end{enumerate}
  Let $\tilde{Y}$ be the successive toric blow up of $(Y,D)$ at the intersection of $x,z$-axis, the proper transform of $z$-axis and the exceptional divisor, the two nodes on the last exceptional divisor and then the proper transform of $y,z$-axis in order. Then take $\tilde{D}$ to be the proper transform of $D$. 
Denote $H$ to be the pull-back of the hyperplane class on $\mathbb{P}^2$, $E_1$ (and $E_2,E_3,E_4$) to be the exceptional divisor of the blow up on the non-toric point on the $x$-axis (and $y$-axis). 

Similar to the argument Section \ref{comparison w/ GHK} we have the following lemma.
\begin{lem}
  The complex affine structure on $B_0$ together with $l_{\gamma_i}$ is isomorphic to the integral affine manifold $B_{\GHK}$ of $(\tilde{Y},\tilde{D})$. Moreover, the BPS rays $l_{\gamma_i}$ give the corresponding cone decomposition on $B_{\GHK}$ from $(\tilde{Y},\tilde{D})$, the wall function with restriction $z^{[D_i]}=z^{[E_i]}=1$ and the identification $d$ coincide with the functions in Theorem \ref{BPS dp4}
\end{lem}

We then can compute the canonical scattering diagram for $(Y,D)$. Actually all the simple $\mathbb{A}^1$-curves contributing to the scattering diagram are toric transverse in $(\tilde{Y},\tilde{D})$, which are depicted in Figure \ref{fig:G2curves} below.
\begin{figure}[H]
\centering
\def\svgwidth{220pt}
    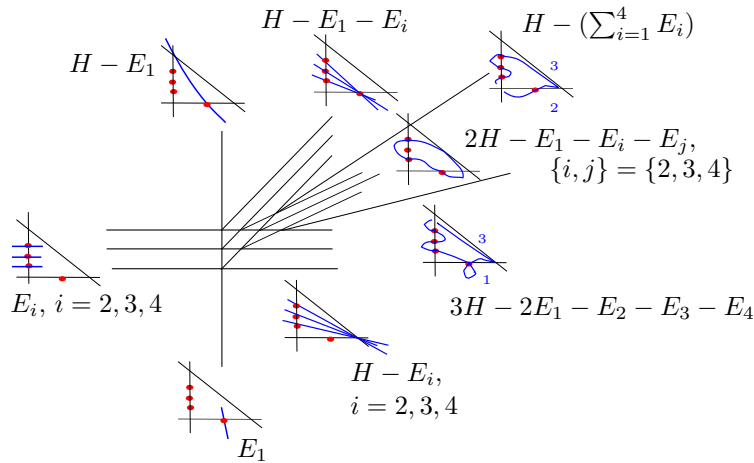
    \caption{The scattering diagram of Gross-Pandharipande-Siebert \cite{GPS} and the $\mathbb{A}^1$-curves corresponding to $X_{IV}$ (illustrated by a projection to $\mathbb{P}^2$). } \label{fig:G2curves}
\end{figure}

We conclude the section with the following theorem. 
\begin{thm} \label{mirror IV}
  The family Floer mirror of $X_{IV}$ has a partial compactification as the analytification of the $G_2$-cluster variety or the Gross-Hacking-Keel mirror of a suitable pair $(\YGHK,\DGHK)$. 
\end{thm}

\section{Further Remarks}
Here we consider the family Floer mirror of $X$ without the geometry of its compactification. Following the idea of the Gross-Hacking-Keel as summarized in Section \ref{RES}, one would need to use the theta functions, the tropicalization of the counting of Maslov index two discs, to construct a (partial) compactification of the original mirror. Assuming that $X=X_*$ in the previous sections admit a compactification to a rational surface with an anti-canonical cycle at infinity. Moreover, assume that the there is certain compatibility between the compactification and the asymptotic of the metric behavior. Then one can follow the similar argument in the work of the second author \cite{L14} and prove that the counting of the Maslov index two discs with Lagrangian fibre boundary conditions can be computed by the weighted count of broken lines. Such asymptotics behavior of the Ricci-flat metrics will be provided in the upcoming work of the second author with T. Collins \cite{CL2}.

One can further construct the pair $(\YGHK,\DGHK)$ such that the corresponding monodromy is conjugate to the monodromy of the type $IV^*,III^*,II^*, I_0^*$. For instance, the case of $I_0^*$ can be realized by a cubic surface with anti-canonical cycle consisting of three $(-1)$-curves \cite{ghks_cubic}. 
The authors would expect that the family Floer mirror of $X=Y\setminus D$ coincides with a particular fibre in the mirror family constructed by Gross-Hacking-Keel. Moreover, the families of Maslov index zero discs emanating from the singular fibres in $X$ are one-to-one corresponding to the $\mathbb{A}^1$-curves of the pair $(\YGHK,\DGHK)$. This may help to understand the Floer theory of more singular Lagrangians. 
In this case, the wall functions are algebraic functions and the GAGA can still apply. Although the walls are dense, it is likely the mirror can be covered by finitely many tori up to some codimension two locus.
In general, the wall functions may not be algebraic a priori and GAGA may not apply directly. 
The authors will leave it for the future work. 

\appendix
\section{Proof of Lemma \ref{cancellation}}
	We will use the identification as in \eqref{mutation}.
	Let us consider $a=a_1 \gone+ a_2 \gtwo \in H_1(L_p,\mathbb{Z})$, where $p\in B_0$ is a reference point, and a loop from $l_{\gamma_1}$ anticlockwise to itself:
	We will first compute the case without any singularities.
	This is very standard from \cite{GPS}. We are only repeating it as there may be confusion about signs.
	\begin{figure}[H]
		\centering
		\begin{tikzpicture}
			\draw (0,0)--++(45:2) node[above right] {$1+z^{\gtwo}$};
			\draw (0,0)--++(90:2) node[above] {$1+z^{\gone + \gtwo}$};
			\draw (0,0)--++(135:2)node[left] {$ 1+z^{\gone}$};
			\draw (0,0)--++(-45:1.5);
			\draw (0,0)--++(-135:1.5) ;
			\node[right] at ++(45:1.6) {$l_{\gamma_2}$};
			\node[right] at ++(90:1.6) {$l_{\gamma_3}$};
			\node[right] at ++(135:1.6) {$l_{\gamma_4}$};
			\node[right] at ++(-135:1.6) {$l_{\gamma_5}$};
			\node[right] at ++(-45:1.6) {$l_{\gamma_1}$};
			\node[left] at (-.6,0) {$\delta$};
			\filldraw(0, 0) circle(1pt) ;
			\draw[thick, ->] ([shift={(0:0.5)}] 0, 0)  arc (0:330:.5) ;
		\end{tikzpicture}
	\end{figure}

	\begin{remark} \label{rk:loop}
		Before we go into the calculation, let us unfold the sign convention in Theorem \ref{thm:wallcrossing}. 
		To determine the sign, we have the condition 
		$\mbox{Arg}Z_{\gamma}(u_-)<\mbox{Arg}Z_{\gamma}(u_+)$. 
		This means that the loop $\delta$ is going in  anti-clockwise direction.
		
		In the calculation of the exponents, we consider $\gamma \mapsto \langle \cdot , \gamma \rangle$.
		Note that $\langle \cdot, \cdot \rangle$ is the intersection pairing but not the usual inner product. 
		Together with
		$\langle \gone, \gtwo \rangle =1$, 
		we have $ \langle \cdot , \gamma \rangle$ is the normal of $l_{\gamma}$ pointing in the same direction as $\delta$ in the language of \cite{GPS}.
	\end{remark}

	Let us consider the transformation $\mathcal{K}_{\delta}=\mathcal{K}_{{\delta}, l_{\gamma_1}}
	\mathcal{K}_{{\delta}, l_{\gamma_5}} \mathcal{K}_{{\delta}, l_{\gamma_4}}
	\mathcal{K}_{{\delta}, l_{\gamma_3}}
	\mathcal{K}_{{\delta}, l_{\gamma_2}}$, 
	where $\mathcal{K}_{{\delta}, l_{\gamma_k}} = \mathcal{K}_{\gamma_k}$ for $k = 1, 2, 3$;
	$\mathcal{K}_{{\delta}, l_{\gamma_k+3}} = \mathcal{K}_{\gamma_k}$ for $k = 1, 2$.
	
	To simplify the notation, we will denote 
	$\xmapsto{\mathcal{K}_{{\delta}, l_{\gamma_k}} }$ for the wall crossing over the wall $l_{\gamma_k}$ according to the curve $\delta$. 
	\begin{align*}
		z^a &\xmapsto{\mathcal{K}_{{\delta}, l_{\gamma_2}} } z^a (1+z^{\gtwo})^{a_1}, \\
		&\xmapsto{\mathcal{K}_{\delta, l_{\gamma_3}}}  z^a (1+z^{\gone + \gtwo})^{a_1-a_2}  \left(1+z^{\gtwo}(1+z^{\gone + \gtwo})^{-1} \right)^{a_1}, \\
		&= z^a (1+z^{\gone + \gtwo})^{-a_2} (1+z^{\gtwo}+ z^{\gone + \gtwo})^{a_1}, \\
		&\xmapsto{\mathcal{K}_{\delta, l_{\gamma_4}}} z^a (1+z^{\gone})^{-a_2} 
		\left(1+z^{\gone + \gtwo}(1+z^{\gone})^{-1}\right)^{-a_2}
		\left( 1+ z^{\gtwo}(1+z^{\gone})^{-1}(1+z^{\gone})\right)^{a_1}, \\
		&= z^a (1+z^{\gone} + z^{\gone + \gtwo} )^{-a_2} (1+z^{\gtwo})^{a_1}, \\
		&\xmapsto{\mathcal{K}_{\delta, l_{\gamma_5}}} z^a (1+z^{\gtwo})^{-a_1} 
		\left(1+z^{\gone}(1+z^{\gtwo})^{-1}  (1 + z^{\gtwo }) \right)^{-a_2}
		(1+z^{\gtwo})^{a_1}\\
		&=  z^a (1+z^{\gone})^{-a_2},\\
		&\xmapsto{\mathcal{K}_{\delta, l_{\gamma_1}}} z^a (1+z^{\gone})^{a_2}(1+z^{\gone})^{-a_2},\\
		&=z^a.
	\end{align*}
	
	Thus we obtain the consistency as usual.
	Next we investigate the wall crossing transformation over the monodromy deduced by focus-focus singularities on $l_{\gtwo}$. 
	\begin{figure}[H]
		\centering
		\begin{tikzpicture}
			\draw
			(-3,-1) node[left] {$1+z^{-\gtwo}$} -- (1,1/3) node[right] {$ 1+z^{\gtwo}$};
			\filldraw(0, 0) circle(1pt) node[above]{$0$};
			\fill(-1.5, -1.5/3) node[sing] {};
			\draw[thick,<-] ([shift={(250:.5)}] -1.5, -1.5/3)  arc (250:-30:.5);
			\node at (-1, 0.2) {$\beta$};
		\end{tikzpicture}
	\end{figure}
	
	Let us consider the wall crossing $\mathcal{K}_{\beta}=\mathcal{K}_{\beta,2}\mathcal{K}_{\beta,1}$ over the curve $\beta$, where $\mathcal{K}_{\beta,1}=\mathcal{K}_{\gtwo}$, and $\mathcal{K}_{\beta,2}=\mathcal{K}_{-\gtwo}$.
	The first wall crossing will lead us to 
	\begin{align*}
		\mathcal{K}_{\beta,1} (z^a) = z^a (1+z^{\gtwo})^{a_1}. 
	\end{align*}
	Then passing over the wall again by using $\beta$ will get us
	\begin{align*}
		\mathcal{K}_{\beta}(z^a)
		&= \mathcal{K}_{\beta, 2} \circ \mathcal{K}_{\beta, 1 } (z^a)= z^a (1+z^{-\gtwo})^{-a_1}(1+z^{\gtwo})^{a_1} \\
		&= z^{a_1\gone + (a_1+a_2) \gtwo}. 
	\end{align*}
	
	To have $z^{a_1\gone + (a_1+a_2) \gtwo}$ goes back to $z^a$, we have the monodromy $M_2$
	\begin{align}
		\gone &\mapsto \gone - \gtwo, \\
		\gtwo & \mapsto \gtwo.
	\end{align}
	
	Let us first consider the monodromy over the focus-focus singularities on $l_{\gone}$: 
	\begin{figure}[H]
		\centering
		\begin{tikzpicture}
			\draw
			(-3,1) node[left] {$1+z^{\gone}$} -- (1,-1/3) node[right] {$1+z^{-\gone}$};
			\filldraw(-2, 2/3) circle(1pt) node[above]{$0$};
			\fill(0, 0) node[sing] {};
			\draw[thick, ->] ([shift={(100:.5)}] 0, 0)  arc (100:390:.5);
			\node at (.8, .3) {$\alpha$};
		\end{tikzpicture}
	\end{figure}
	
	Consider the transformation according to the loop $\alpha$. Let $\mathcal{K}_{\alpha, 1}=\mathcal{K}_{\gone}$, and $\mathcal{K}_{\alpha, 2}=\mathcal{K}_{-\gone}$.
	We have
	\begin{align*}
		\mathcal{K}_{\alpha, 1} (z^a)= z^a (1+z^{\gone})^{-a_2}. 
	\end{align*}
	Then the whole loop $\alpha$ leads us to 
	\begin{align*}
		\mathcal{K}_{\alpha} &=  \mathcal{K}_{\alpha, 2} \circ  \mathcal{K}_{\alpha, 1} (z^a) 
		= z^a (1+z^{-\gone})^{a_2}  (1+z^{\gone})^{-a_2} \\
		&= z^{(a_1-a_2 )\gone + a_2 \gtwo}. 
	\end{align*}
	Then we obtain the monodromy $M_1$
	\begin{align}
		\gone &\mapsto \gone, \\
		\gtwo & \mapsto \gone+ \gtwo. 
	\end{align}
	
	Thus, we can compute the monodromy while singularity is at the origin by decomposing the singularity at the origin into two focus-focus singularities to check consistency similar to \cite{GHK}.
	Explicitly, there are two ways checking it. The first one is doing a similar calculation as in the beginning of the proof. Now we consider
	
	\begin{figure}[H]
		\centering
		\begin{tikzpicture}
			\draw (0,0)--++(45:2) node[above right] {$1+z^{\gtwo}$};
			\draw (0,0)--++(90:2) node[above] {$1+z^{\gone + \gtwo}$};
			\draw (0,0)--++(135:2)node[left] {$ 1+z^{\gone}$};
			\draw (0,0)--++(-45:2) node[right] {$ 1+z^{-\gone}$};
			\draw (0,0)--++(-135:2) node[left] {$ 1+z^{-\gtwo}$} ;
			\node[right] at ++(45:1.6) {$l_{\gamma_2}$};
			\node[right] at ++(90:1.6) {$l_{\gamma_3}$};
			\node[right] at ++(135:1.6) {$l_{\gamma_4}$};
			\node[right] at ++(-135:1.6) {$l_{\gamma_5}$};
			\node[right] at ++(-45:1.6) {$l_{\gamma_1}$};
			\filldraw(0, 0) circle(1pt) ;
			\draw[thick, ->] ([shift={(0:1)}] 0, 0)  arc (0:330:1) ;
			\draw [snake=snake,
			segment amplitude=.4mm,
			segment length=2mm] (0,0)--++(-90:2);
			\fill(0,0) node[sing] {};
		\end{tikzpicture}
	\end{figure}
	
	The first three wall crossings are the same and let us recap here:
	\begin{align*}
		\mathcal{K}_{\delta, l_{\gamma_4}}\mathcal{K}_{\delta, l_{\gamma_3}}\mathcal{K}_{\delta, l_{\gamma_2}}(z^a) &= 
		z^a (1+z^{\gone} + z^{\gone + \gtwo} )^{-a_2} (1+z^{\gtwo})^{a_1}.
	\end{align*}
	
	Now to pass over $l_{\gamma_5}$, we will have
	\begin{align*}
		\mathcal{K} ( z^a (1+z^{\gone} + z^{\gone + \gtwo} )^{-a_2} (1+z^{\gtwo})^{a_1})&= z^a (1+z^{-\gtwo})^{-a_1} \left( 1+z^{\gone} (1+z^{-\gtwo})^{-1} (1+z^{\gtwo}) \right)^{-a_2} (1+z^{\gtwo})^{a_1} \\
		&= z^{a_1 \gone + (a_1+a_2) \gtwo} (1+z^{\gone + \gtwo})^{-a_2}.
	\end{align*}
	
	The monodromy $M$ would then be
	\begin{align*}
		\gone &\mapsto -\gtwo; \\
		\gtwo &\mapsto \gone + \gtwo.
	\end{align*}
	and gives us
	\begin{align*}
		\mathcal{K}_{M} ( z^{a_1 \gone + (a_1+a_2) \gtwo} (1+z^{\gone + \gtwo})^{-a_2}) = z^{(a_1+a_2)\gone + a_2 \gtwo} (1+z^{\gone})^{-a_2}.
	\end{align*}
	
	The last wall crossing would then be 
	\begin{align*}
		\mathcal{K}_{\delta, l_{\gamma_1}} \left((z^{(a_1+a_2)\gone + a_2 \gtwo} (1+z^{\gone})^{a_1}\right)
		&= z^{(a_1+a_2)\gone + a_2 \gtwo} (1+z^{-\gone})^{a_2}(1+z^{\gone})^{-a_2}\\
		&= z^a.
	\end{align*}
	
	The second way is to use the following meta-lemma by direct computation
	\begin{claim} \label{focus-focus}
		$
		\mathcal{K}_{-\gamma}\mathcal{K}_{\gamma}(z^{\gamma'})= z^{M^{-1}\gamma'}$,
		where $M$ is transformation $\gamma'\mapsto \gamma' +\langle \gamma,\gamma'\rangle\gamma$.
	\end{claim}
	
	Note that if $\gamma$ is primitive, then $M$ is the Picard-Lefschetz transformation of a focus-focus singularity with Lefschetz thimble $\gamma$.
	Recall that if $\langle \gamma',\gamma\rangle=1 $, then the pentagon equation reads 
	\begin{align}\label{pentagon}
		\mathcal{K}_{\gamma}\mathcal{K}_{\gamma'}=\mathcal{K}_{\gamma'}\mathcal{K}_{\gamma+\gamma'}\mathcal{K}_{\gamma}.
	\end{align}
	Let $M_1,M_2$ denote the transformation in the Claim \ref{focus-focus} with respect to $\gamma'_1,\gamma'_2$ respectively.

	Then   
	\begin{align*}
		\mathcal{K}_{\gamma_5} \mathcal{K}_{\gamma_4}\mathcal{K}_{\gamma_3}\mathcal{K}_{\gamma_2}\mathcal{K}_{\gamma_1} =
		\bigg(\mathcal{K}_{-\gamma_2}\mathcal{K}_{\gamma_2}\bigg)\bigg(\mathcal{K}_{\gamma_2}^{-1}\mathcal{K}_{\gamma_4}\mathcal{K}_{\gamma_3}\mathcal{K}_{\gamma_2}\mathcal{K}_{-\gamma_1}^{-1}\bigg)\bigg(\mathcal{K}_{-\gamma_1}\mathcal{K}_{\gamma_1}\bigg).  
	\end{align*} Notice that the middle of the right hand side is identity by the pentagon identity \eqref{pentagon}. From Lemma \ref{focus-focus}, we have \begin{align*}
		\mathcal{K}_{\gamma_5} \mathcal{K}_{\gamma_4}\mathcal{K}_{\gamma_3}\mathcal{K}_{\gamma_2}\mathcal{K}_{\gamma_1}(z^{\gamma})=z^{M_2^{-1}M_1^{-1}\gamma}=z^{(M_1M_2)^{-1}\gamma}
	\end{align*} and the lemma follows from the fact that $M=M_1M_2$. Notice that the proof is motivated by deforming the type $II$ singular fibre into two $I_1$ singular fibres as in Figure \ref{decompositon}. However, the proof does NOT depend on the actual geometric deformation. 
	\begin{figure}[H]
		\centering
		\begin{tikzpicture}
			\draw (0,0)--++(45:2) node[above right] {$1+z^{\gtwo}$};
			\draw (0,0)--++(90:2) node[above] {$1+z^{\gone + \gtwo}$};
			\draw (0,0)--++(135:2)node[left] {$ 1+z^{\gone}$};
			\draw (0,0)--++(-45:2) node[right] {$ 1+z^{-\gone}$};
			\draw (0,0)--++(-135:2) node[left] {$ 1+z^{-\gtwo}$} ;
			\filldraw(0,0) circle(1pt) ;
			\draw [snake=snake,
			segment amplitude=.4mm,
			segment length=2mm] (-135:1.3)--++(-90:1);
			\fill(0,0) node[sing] {};
			\draw [snake=snake,
			segment amplitude=.4mm,
			segment length=2mm](-45:1.3)--++(-90:1) ;
			\fill(0,0) node[sing] {};
			\fill(-135:1.3) node[sing] {};
			\fill(-45:1.3) node[sing] {};
			\draw[thick, ->] ([shift={(0:.5)}] 0, 0)  arc (0:330:.5) ;
			\draw[thick, ->] ([shift={(0:.3)}] -45:1.3)  arc (0:330:.3) ;
			\draw[thick, ->] ([shift={(0:.3)}] -135:1.3)  arc (0:330:.3) ;
		\end{tikzpicture}
		\caption{Geometric interpretation of Lemma \ref{cancellation}.}
		\label{decompositon}
	\end{figure}

\section{A Lemma about the Monodromy}
Here we prove the following statement which is used in Section \ref{section:dp4} for proving the geometry $X_{IV}$ and certain Looijenga interior are diffeomorphic. 
\begin{lem}\label{lem: monodromy Appendix}
	If a type $IV$-fibre deforms to a type $II$-fibre and a $I_2$-fibre, then the vanishing cycle of the $I_2$-fibre is one of the classes in Theorem \ref{local discs} for the type $II$-fibre. 
\end{lem}
\begin{proof}
	 Recall that the unique matrix with $\begin{pmatrix} p \\ q \end{pmatrix}$ as eigenvector and conjugate to $\begin{pmatrix} 1 & k \\ 0  & 1\end{pmatrix}$ is given by $\begin{pmatrix} 1-kpq & kp^2\\ -kq^2 & 1+kpq   \end{pmatrix}$. The monodromy of the deformation of the type $IV$-fibre into a type $II$-fibre and a $I_2$-fibre implies that 
	  \begin{align*}
	  	  \begin{pmatrix}0 & 1\\ -1 & 1 \end{pmatrix}\begin{pmatrix} 1-2pq & 2p^2\\ -2q^2 & 1+2pq   \end{pmatrix}\sim    \begin{pmatrix}0 & 1\\ -1 & -1 \end{pmatrix},
	  \end{align*} for some $p,q\in \mathbb{Z}$. 
  By looking at the trace of both sides implies that $p^2-pq+q^2=1$. It's then easy to check that the only integer solutions are one-to-one corresponding to those classes in Theorem \ref{local discs}. 
\end{proof}

%
%

%
\bibliographystyle{alpha}
\bibliography{biblio} 

\noindent{\sc{Man-Wai Cheung\\
Kavli Institute for the Physics and Mathematics of the Universe (IPMU), 5-1-5 Kashiwanoha, Kashiwa, Chiba, 277-8583, Japan. }}\\
{\it{e-mail:}} \href{mailto:manwai.cheung@ipmu.jp}{manwai.cheung@ipmu.jp} \medskip

\noindent{\sc{Yu-Shen Lin\\
Department of Mathematics and Statistics, 665 Commonwealth Ave, Boston, Boston University, MA 02215, USA }}\\
{\it{e-mail:}} \href{mailto:yslin@bu.edu}{yslin@bu.edu}\medskip

\end{document}